\documentclass[a4paper,11pt,pdf,reqno]{amsart}
\usepackage{enumerate, amsmath, amsfonts, amssymb, amsthm, thmtools, wasysym, graphics, graphicx, xcolor, frcursive,comment,bbm}
\usepackage[colorlinks=true,citecolor=cyan,backref=none]{hyperref}

\usepackage{etex}
\usepackage{lscape}
\usepackage{tikz-cd}

\makeatletter
\newcommand{\thickhline}{%
    \noalign {\ifnum 0=`}\fi \hrule height 1pt
    \futurelet \reserved@a \@xhline
}

\definecolor{darkblue}{rgb}{0.0,0,0.7} 
 
\definecolor{darkred}{rgb}{0.7,0,0} 
\usepackage{hyperref}
\usepackage[all]{xy}
\usepackage{mathtools}
\usepackage{lipsum}
\usepackage[T1]{fontenc}
\usepackage{stmaryrd}
\usepackage{adjustbox}
\usepackage{footnote}
\makesavenoteenv{table}
\usepackage{vmargin}            
\setmarginsrb{3.5cm}{2.5cm}{3.5cm}{2.5cm}{0cm}{0.6cm}{0cm}{0cm}

\usepackage{caption,lipsum}
\usepackage{graphicx}                  
\usepackage{pstricks,pst-plot,pst-text,pst-tree,pst-eps,pst-fill,pst-node,pst-math}
\usepackage{setspace}
\usepackage{multicol}

\usepackage{yhmath}
\hypersetup{
    citecolor=magenta,
    colorlinks=true,
    linkcolor=blue,
    filecolor=cyan,      
    urlcolor=magenta,
}

\numberwithin{equation}{section}

\newcommand{\Z}{\mathbb{Z}}

\newcommand{\llangle}{\langle\langle}
\newcommand{\rrangle}{\rangle\rangle}
\newcommand{\N}{\mathbb{N}}
\newcommand{\C}{\mathbb{C}}
\newcommand{\B}{\mathcal{B}}
\newcommand{\im}{\text{Im}}

\newcommand{\Inn}{\text{Inn}}

\newcommand{\multl}{\{\!\!\{}
\newcommand{\multr}{\}\!\!\}}

\newcommand{\xto}{\xrightarrow}

\newcommand{\coker}{\text{coker}}

\newcommand{\ivl}[2]{\tilde\varphi([#1]_\mathfrak c)=[#2]_\mathfrak c}

\newcommand{\col}[6]{\left\{\!\!\!\left\{\begin{pmatrix}#1\\#2\end{pmatrix},\begin{pmatrix}#3\\#4\end{pmatrix},\begin{pmatrix}#5\\#6\end{pmatrix}\right\}\!\!\!\right\}}
\newcommand*\notocchapter[1]{%
  \if@openright\cleardoublepage\else\clearpage\fi
  \thispagestyle{empty}\global\@topnum\z@
  \@afterindenttrue
  \let\@secnumber\@empty
  \@makeschapterhead{#1}\@afterheading
}
\theoremstyle{plain}

\newtheorem{theorem}{Theorem}[section]
\newtheorem{lemma}[theorem]{Lemma}
\newtheorem{proposition}[theorem]{Proposition}
\newtheorem{corollary}[theorem]{Corollary}

\theoremstyle{definition}

\newtheorem{definition}[theorem]{Definition}
\newtheorem{example}[theorem]{Example}
\newtheorem{remark}[theorem]{Remark}

\newtheorem{question}[theorem]{Question}
\newtheorem{notation}[theorem]{Notation}

\title[A generalisation of rank two complex reflection groups]{A combinatorial generalisation of rank two complex reflection groups via generators and relations}

\author{Igor Haladjian}
\address{Institut Denis Poisson, CNRS UMR 7013, Faculté des Sciences et Techniques, Université de Tours, Parc de Grandmont, 
37200 TOURS, France}
\DeclareRobustCommand{\SkipTocEntry}[5]{}
\begin{document}

%\tableofcontents
\thispagestyle{empty}

\begin{abstract}
Complex reflection groups of rank two are precisely the finite groups in the family of groups that we call $J$-reflection groups. These groups are particular cases of $J$-groups as defined by Achar \& Aubert in 2008. The family of $J$-reflection groups generalises both complex reflection groups of rank two and toric reflection groups, a family of groups defined and studied by Gobet. We give uniform presentations by generators and relations of $J$-reflection groups, which coincide with the presentations given by Broué, Malle and Rouquier when the groups are finite. In particular, these presentations provide uniform presentations for complex reflection groups of rank two where the generators are reflections (however the proof uses the classification of irreducible complex reflection groups). Moreover, we show that the center of $J$-reflection groups is cyclic, generalising what happens for irreducible complex reflection groups of rank two and toric reflection groups. Finally, we classify $J$-reflection groups up to reflection isomorphisms.
\end{abstract}
\maketitle
\tableofcontents
\section{Introduction}

Given a finite dimensional complex vector space $V$, a pseudo-reflection of $V$ is an element of $\text{GL}(V)$ of finite order whose set of fixed points is a hyperplane.
A complex reflection group is a finite subgroup of $\text{GL}(V)$ generated by pseudo-reflections of some finite dimensional complex vector space $V$. For basics on complex reflection groups, see \cite{Broué} and \cite{LT}. Irreducible complex reflection groups have been completely classified by Shephard and Todd in 1954 (\cite{ST}) (see also \cite{LT} for a detailed exposition). They fall into two families:\\
$\bullet$ The infinite series denoted by $G(de,e,r)$ with $d,e,r\geq 1$.\\
$\bullet$ The $34$ exceptional complex reflection groups denoted by $G_4,\dots,G_{37}$.\\
Well known examples of complex reflection groups are real reflection groups, but the majority of the groups in the above families are not real reflection groups. The finite real reflection groups are generalised by so-called Coxeter groups. More precisely, a well known theorem of Coxeter states the following:
\begin{theorem}[{\cite{CoxeterClassification2} and \cite{CoxeterClassification}}]\label{CoxiCoxou}
A group is a finite Coxeter group if and only if it is a finite real reflection group.
\end{theorem}
It would be desirable to have an analog of Coxeter groups for complex reflection groups. In 2008, Achar and Aubert introduced in \cite{AA} a family of groups they called $J$-groups which generalize rank two complex reflection groups. The definition of $J$-groups is as follows:
\begin{definition}
 For $k,n,m\in \N_{\geq 2}$, define $J\begin{pmatrix} k & n & m \\ & & \end{pmatrix}$ to be the group with presentation
 $$\langle s,t,u\, | \, s^k=t^n=u^m=1, \, stu=tus=ust\rangle.$$
More generally, for pairwise coprime elements $k',n',m'\in \N^*$ such that $x'$ divides $x$ for all $x\in \{k,n,m\}$, the group $J\begin{pmatrix} k & n & m \\ k'&n' &m' \end{pmatrix}$ is defined as the normal closure of $\{s^{k'},t^{n'},u^{m'}\}$ in $J\begin{pmatrix} k& n & m \\ & & \end{pmatrix}$. Finally, whenever $k',n'$ or $m'$ is equal to $1$, we omit it in the notation.
\end{definition}

The main result of \cite{AA} is the following theorem, which has a similar flavor to that of Theorem \ref{CoxiCoxou}:
\begin{theorem}[{\cite[Theorem 1.2]{AA}}]\label{THMAA}
A group is a finite $J$-group if and only if it is a rank two complex reflection group.
\end{theorem}

While Theorem \ref{THMAA} precisely describes which $J$-groups are finite, when a $J$-group $H$ is infinite the definition does not in general give a group presentation of $H$, nor a way to determine its center. A first step in this direction was done by Gobet in \cite{Gobet Toric}. He defined and studied what he called \textit{toric reflection groups}, which are quotients of torus knot groups for which the knot groups naturally play a role of "associated braid groups". In \cite{Gobet Toric}, Gobet showed that toric reflection groups are isomorphic to a family of $J$-groups whose finite groups are precisely the irreducible rank two complex reflection groups having exactly one conjugacy class of reflecting hyperplanes. In particular, he defines a notion of generalised reflections for $J$-groups and classifies toric reflection groups up to reflection isomorphisms (see Definition \ref{DefReflectionIso}), which is useful to define their "associated braid group". Moreover, whenever a toric reflection group is finite, its set of generalised reflections coincide with its set of complex reflections.\medskip

In this article, we introduce a family of $J$-groups that we call \textit{$J$-reflection groups}, of which toric reflection groups are a particular case. These groups generalise rank two complex reflection groups, in the sense that every such group is a $J$-reflection group. A $J$-reflection group depends on parameters $k,b,c,n,m\in \N^*$ where $n\wedge m=1$ and is defined as $W_b^c(k,bn,cm):=J\begin{pmatrix} k & bn & cm \\ &n&m \end{pmatrix}$.\\
In Theorem \ref{Pres2} we obtain a group presentation for all $J$-reflection groups where the generators are generalised reflections.
In the case where the group is finite, Theorem \ref{Pres2} recovers the BMR presentation (see \cite[Tables 1-3]{BMR}). Moreover, generalised reflections of finite $J$-reflection groups coincide with their complex reflections. Theorem \ref{Pres2} thus provides a uniform presentation by generators and relations of rank two complex reflection groups where generators are reflections, as part of a larger family of (in general, infinite) reflection groups. To the knowledge of the author, the only previous uniform description by generators and relations of rank two complex reflection groups was given by Etingof and Rains in \cite{ER}\footnote{The author thanks Ivan Marin for pointing out the work of Etingof and Rains.}. However, the generators of these presentations are not always complex reflections.\\
\noindent
It is not clear to us whether $J$-reflection groups would be a better suited generalisation of rank two complex reflection groups than the whole family of $J$-groups. In particular, one would expect from Coxeter-analogous groups to be finitely presented but we do not know if this is the case for every $J$-group. In fact, we do not know if all $J$-groups are finitely generated or not, and an answer to that question would be of interest to pursue the search for such analogs:
\begin{question}
Are all $J$-groups finitely generated ? Are they finitely presented ?
\end{question}

Afterwards, we determine the center of $J$-reflection groups. In order to do so, we use a strategy similar to the one used in \cite[Theorem 1.4]{Gobet Toric} to obtain the center of toric reflection groups. To state the result, we denote by $W^+_{k,n,m}$ the alternating subgroup of the triangle Coxeter group with edge labels $k,n,m$. We obtain the following result, simultaneously generalising the known situation for irreducible rank two complex reflection groups (\cite[Tables 1-4]{BMR}) and toric reflection groups (\cite[Theorem 1.4 and Corollary 1.6]{Gobet Toric}):
\begin{theorem}[Theorem \ref{CenterFINDANAME}]\label{CenterIntro}
With the notation of Theorem \ref{Pres2}, writing $\Delta$ for the element $(x_1\cdots x_ny)^mz^n\in W_b^c(k,bn,cm)$, the following hold:\\\\
(1) We have a short exact sequence
\begin{equation}\begin{tikzcd}
	1 & {\langle \Delta\rangle} & {W_b^c(k,bn,cm)} & {W^+_{k,bn,cm}} & 1
	\arrow[from=1-1, to=1-2]
	\arrow[from=1-2, to=1-3]
	\arrow[from=1-3, to=1-4]
	\arrow[from=1-4, to=1-5]
\end{tikzcd}
\end{equation}
(2) The center of $W_b^c(k,bn,cm)$ is cyclic, generated by $\Delta$.\\
\noindent
In particular, the group $W_b^c(k,bn,cm)$ is a central extension of $W_{k,bn,cm}^+$ by the subgroup $\langle \Delta\rangle$.
\end{theorem}

A further step will be to attach to each $J$-reflection group $W$ a "braid group" $\B(W)$, enjoying similar properties to that of complex braid groups. This associated braid group will be defined in future work by the author, where it will be shown that it only depends on the (reflection) isomorphism type of $W$ (see Definition \ref{DefReflectionIso} for the meaning of a reflection isomorphism). The following classification result will be useful to show that braid groups associated to $J$-reflection groups are well defined up to reflection isomorphisms:
\begin{theorem}[Theorem \ref{Classification3}]\label{Theorem B}
The groups $W_b^c(k,bn,cm)$ and $W_{b'}^{c'}(k',b'n',c'm')$ are isomorphic in reflection if and only if the multisets of columns $\col{k}{1}{bn}{n}{cm}{m}$ and $\col{k'}{1}{b'n'}{n'}{c'm'}{m'}$ are equal.
\end{theorem}
\medskip

The paper is organised as follows. In section 2, we start by recalling the definition of $J$-groups and some results proven in \cite{AA} and \cite{Gobet Toric}. Afterwards, we give a presentation of $J$-reflection groups (Theorem \ref{Pres2}) and determine their center (Theorem \ref{CenterIntro}). In section 3, we recall some results proven in \cite{Gobet Toric} and study the "reflection structure" of $J$-reflection groups. In particular, we classify $J$-reflection groups up to reflection isomorphisms (Theorem \ref{Theorem B}).
\addtocontents{toc}{\SkipTocEntry}
\section*{Acknowledgments}
This work is part of my PhD thesis, done under the supervision of Thomas Gobet and Cédric Lecouvey at Université de Tours, Institut Denis Poisson. 
I thank Thomas Gobet for his careful and numerous readings of all redaction stages of this article. I also thank him for all his precious advice, suggestions and his answers to my questions. I also thank Cédric Lecouvey for his careful reading at numerous stages of this article and his advice. Moreover, I thank Ivan Marin for pointing out \cite{ER}. I thank Eddy Godelle and Luis Paris for suggesting the name $J$-reflection groups. In addition, I thank David Wahiche and Owen Garnier for several helpful discussions. Finally, I thank the reviewer for their precise reading and useful suggestions. This work was partially supported by the ANR
project CORTIPOM (ANR-21-CE40-0019).

\section{\texorpdfstring{$J$-reflection groups}{}}

In \cite{AA}, Achar and Aubert introduced $J$-groups as a family of groups generalising rank two complex reflection groups. Not only $J$-groups are abstract generalisations of rank two complex reflection groups, but the way they are defined suggests a relevant notion of algebraic reflections for these groups. A first infinite family of $J$-groups called toric reflection groups was studied by Gobet in \cite{Gobet Toric}, where he studied general results about $J$-groups and gave several properties of toric reflection groups, some of which are stated below.\\
In this section, we start by defining $J$-groups and state some known results about them. Afterwards, we extend the notion of toric reflection groups to that of $J$-reflection groups. We give a presentation for these groups and prove similar results to those proven about toric reflection groups in \cite{Gobet Toric}. In particular, we classify reflections of $J$-reflection groups, identify their center and compute their inner automorphism group.

\subsection{\texorpdfstring{Definitions and preliminary results}{}}

\begin{definition}[{\texorpdfstring{\cite[Introduction]{AA}}{}}]\label{defJ}
Let $k,n,m\in \N_{\geq 2}$. Define $G:=J\begin{pmatrix} k & n & m\\ & & \end{pmatrix}$ to be the group with presentation 
\begin{equation}\label{PresParent}\left\langle 
\begin{array}{c|c}s,t,u & s^k=t^n=u^m=1,\, \, stu=tus=ust
\end{array}\right\rangle,\end{equation}
and define $s,t$ and $u$ to be the \textbf{canonical generators} of $G$.\\
For positive pairwise coprime parameters $k',n',m'$ such that $k'|k,n'|n$ and $m'|m$, define $H:=J\begin{pmatrix} k & n & m \\ k' & n' & m'\end{pmatrix}$ to be $\llangle s^{k'},t^{n'},u^{m'}\rrangle_{G}$, where $\llangle X \rrangle_G$ denotes the normal closure of $X$ in $G$. Such groups are called \textbf{$J$-groups} and were defined by Achar and Aubert in \cite{AA}. In this context $G$ is called the \textbf{parent $J$-group} of $H$.\\
Finally, define $C(H)$ to be the multiset $\left\{\!\!\!\left\{\begin{pmatrix}k\\k'\end{pmatrix},\begin{pmatrix}n\\n'\end{pmatrix},\begin{pmatrix}m\\m'\end{pmatrix}\right\}\!\!\!\right\}$.
\end{definition}

\begin{remark}
For all $k,n,m\in \N_{\geq 2}$, we have $J\begin{pmatrix} k & n & m\\ & & \end{pmatrix}=J\begin{pmatrix} k & n & m\\1 &1 &1 \end{pmatrix}$. Following this observation, whenever $k',n'$ or $m'$ is equal to $1$, we omit it from the notation. Moreover, we will use the notation $J(k,n,m)$ for $J\begin{pmatrix} k & n & m\\ & & \end{pmatrix}$.
\end{remark}

\begin{remark}\label{isopermute}
Let $k,n,m\in \N_{\geq 2}$. Let $s,t,u$ be the canonical generators of $J(k,n,m)$ and $s',t',u'$ be the canonical generators of $J(k,m,n)$. The map sending $s'$ to $s^{-1}$, $t'$ to $u^{-1}$ and $u'$ to $t^{-1}$ extends to an isomorphism $$J(k,m,n)\xto \cong J(k,n,m).$$
Moreover, since the normal closure of the inverse of an element is equal to the normal closure of this element in any group, this isomorphism restricts to an isomorphism 
$$J\begin{pmatrix} k & m & n\\ k' & m' & n' \end{pmatrix}\xto \cong J\begin{pmatrix} k & n & m\\ k' & n'& m' \end{pmatrix}.$$
More generally, the isomorphism type of $J\begin{pmatrix} k & n & m\\ k'& n' &m' \end{pmatrix}$ is invariant under column permutations, so that if two $J$-groups $H_1$ and $H_2$ are such that $C(H_1)=C(H_2)$, they are isomorphic.
\end{remark}

Complex reflection groups are defined as finite subgroups of $\mathrm{GL}_n(\C)$ for some $n\in \N^*$ generated by reflections of $\C^n$ (that is, finite order automorphisms whose set of fix points forms a complex hyperplane). Such a group $G$ is called irreducible if the natural representation $G\to\mathrm{GL}_n(\C)$ is irreducible. Moreover, the rank of $G$ is defined as the dimension of the orthogonal complement of the subspace of $\C^n$ fixed pointwise by $G$. In \cite{AA}, Achar and Aubert show the following result:

\begin{theorem}[{\cite[Theorem 1.2]{AA}}]\label{Classification1}
A group is a complex reflection group of rank two if and only if it is a finite $J$-group. Moreover, the reflection isomorphism type of a finite $J$-group $H$ is uniquely determined by $C(H)$.
\end{theorem}

It turns out that whenever $J\begin{pmatrix} k & n & m\\ k' & n'& m' \end{pmatrix}$ is finite, its reflections as a complex reflection group are the conjugates in $J(k,n,m)$ of non trivial powers of $s^{k'},t^{n'}$ or $u^{m'}$. This suggests the following definition:

\begin{definition}[\texorpdfstring{\cite[Definition 2.2]{Gobet Toric}}{}]\label{reflections}
Let $H:=J\begin{pmatrix} k & n & m\\ k' & n'& m' \end{pmatrix}$ be a $J$-group. Define the set $R(H)$ of \textbf{algebraic reflections} of $H$ to be the set of conjugates in $J(k,n,m)$ of non-trivial powers of $s^{k'},t^{n'}$ or $u^{m'}$.
\end{definition}

In practice, we omit the term algebraic and talk about the reflections of $J$-groups, even though these elements are not complex reflections. With this definition, $J$-groups naturally appear as generalised (algebraic) reflection groups. Indeed, by definition, we have $$J\begin{pmatrix} k & n & m\\ k' & n'& m' \end{pmatrix}=\langle \{gxg^{-1}\}_{g\in J(k,n,m),x\in \{s^{k'},t^{n'},u^{m'}\}}\rangle_{J(k,n,m)}$$
so that $J\begin{pmatrix} k & n & m\\ k' & n'& m' \end{pmatrix}$ is generated by its reflections and is a 'reflection subgroup' of $J(k,n,m)$.\\
In \cite{Gobet Toric}, Gobet counts the number of conjugacy classes of reflections of $J(k,n,m)$ and shows that the reflections of $J(k,n,m)$ that lie in $H$ are precisely the reflections of $H$:
\begin{lemma}[{\cite[Lemma 2.5]{Gobet Toric}}]\label{NoConjRef}
No two elements in $$\{s,s^2,\dots,s^{k-1},t,t^2,\dots,t^{n-1},u,u^2,\dots,u^{m-1}\}$$ are conjugate to each other in $J(k,n,m)$. In particular, the group $J(k,n,m)$ admits $k+n+m-3$ conjugacy classes of reflections.
\end{lemma}
\begin{lemma}[{\cite[Lemma 2.6]{Gobet Toric}}]\label{RefHinG}
Let $H$ be a $J$-group and $G$ its parent $J$-group. Then $R(H)=R(G)\cap H$.
\end{lemma}

Since $J$-groups are meant to generalise reflection groups, we would like in principle to classify them in a way that keeps track of the reflections. More concretely, there is a notion of reflections of $J$-groups whose morphisms send reflections to reflections:

\begin{definition}[\texorpdfstring{\cite[Definition 2.3]{Gobet Toric}}{}]\label{DefReflectionIso}
Let $H_1$, $H_2$ be two $J$-groups. A group morphism $H_1\xto\varphi H_2$ is a \textbf{reflection morphism} if $\varphi(R(H_1))\subset R(H_2)\cup\{1_{H_2}\}$. Moreover, the groups $H_1$ and $H_2$ are said to be \textbf{isomorphic in reflection} if there exists a group isomorphism $H_1\to H_2$ sending $R(H_1)$ onto $R(H_2)$. In this case, we write $$H_1\cong_{ref} H_2.$$
\end{definition}

\begin{remark}
The group isomorphisms between $J$-groups consisting of permuting the columns are reflection isomorphisms.
\end{remark}

In \cite{Gobet Toric}, Gobet introduced toric reflection groups, which are defined as quotients of toric knot groups, and proved that toric reflection groups are $J$-groups:

\begin{definition}[{\cite[Introduction]{Gobet Toric}}]\label{ToricQuotientKnot}
Let $k,n,m\in \N_{\geq 2}$ and assume that $n$ and $m$ are coprime.
The group $W(k,n,m)$ is the group defined by the presentation 
\begin{equation}\label{PresToric}
\left\langle \begin{array}{c|c}
    x_1,\dots,x_n\, &\,  x_i^k=1\, \forall i\in [n],\,\, x_i\cdots x_{i+m-1}=x_j\cdots x_{j+m-1}\, \forall 1\leq i<j\leq n
\end{array}
\right\rangle,\end{equation}
where indices are taken modulo $n$.
\end{definition}

\begin{theorem}[{\cite[Theorem 2.12]{Gobet Toric}}]\label{ToricPresentation}

Let $k,n,m\in \N_{\geq 2}$ and assume that $n$ and $m$ are coprime. The $J$-group $J\begin{pmatrix} k & n & m\\ &n &m \end{pmatrix}$ is isomorphic to $W(k,n,m)$. 
More precisely, an isomorphism is given by sending $x_i\in W(k,n,m)$ to $t^{i-1}st^{1-i}$ for all $i\in [n]$.
\end{theorem}
\begin{notation}
From now on, we denote the group $J\begin{pmatrix} k & n & m\\ &n &m \end{pmatrix}$ by $W(k,n,m)$ and call such at group a \textbf{toric reflection group}, following conventions from \cite{Gobet Toric}.
\end{notation}

In \cite{Gobet Toric}, in addition to Theorem \ref{ToricPresentation}, Gobet determines the center of toric reflection groups and classifies them up to reflection isomorphisms. In the next sections, we state and extend these results to a wider class of $J$-groups, that we call $J$-reflection groups:
\begin{definition}
Let $k,b,n,c,m\in\N^*$ with $k,bn,cm\geq 2$ and assume $n\wedge m=1$. Define $W_b^c(k,bn,cm)$ to be the $J$-group $J\begin{pmatrix} k & bn & cm\\ &n &m \end{pmatrix}$. A $J$-group of this form is called a \textbf{$J$-reflection group}.
\end{definition}

\begin{remark}
One can see that toric reflection groups are $J$-reflection groups, setting the parameters $b$ and $c$ to be equal to $1$ in the above definition. Moreover, using the correspondence of Theorem \ref{Classification1} and looking at Table 1 of \cite{AA}, every rank two complex reflection group is a $J$-reflection group.
\end{remark}

\begin{notation}
Whenever $b$ or $c$ is equal to $1$ we omit it from the notation so that we write $W_1^1(k,n,m)$ as $W(k,n,m)$, $W_1^c(k,n,cm)$ as $W^c(k,n,cm)$ and $W_b^1(k,bn,m)$ as $W_b(k,bn,m)$.
\end{notation}

%%%%%%%%%%%%%%%%%%%%%%%%%%%%%%%%%%%%%%%%%%%%%%%%%%%%%%%%%%%%
%%%%%%%%%%%%%%%%%%%%%%%%%%%%%%%%%%%%%%%%%%%%%%%%%%%%%%%%%%%%
%%%%%%%%%%%%%%%%%%%%%%%%%%%%%%%%%%%%%%%%%%%%%%%%%%%%%%%%%%%%%%%%%%%%%%%%%%%%%%%%%%%%%%%%%%%%%%%%%%%%%%%%%%%%%%%%%%%%%%%%%%%%%%%%%%%%%%%%%%%%%%%%%%%%%%%%%%%%%%%%%%%%%%%%%%%%%%%%%%%%
%%%%%%%%%%%%%%%%%%%%%%%%%%%%%%%%%%%%%%%%%%%%%%%%%%%%%%%%%%%%
%%%%%%%%%%%%%%%%%%%%%%%%%%%%%%%%%%%%%%%%%%%%%%%%%%%%%%%%%%%%%%%%%%%%%%%%%%%%%%%%%%%%%%%%%%%%%%%%%%%%%%%%%%%%%%%%%%%%%%%%
%%%%%%%%%%%%%%%%%%%%%%%%%%%%%%%%%%%%%%%%%%%%%%%%%%%%%%%%%%%%
%%%%%%%%%%%%%%%%%%%%%%%%%%%%%%%%%%%%%%%%%%%%%%%%%%%%%%%%%%%%
%%%%%%%%%%%%%%%%%%%%%%%%%%%%%%%%%%%%%%%%%%%%%%%%%%%%%%%%%%%%%%%%%%%%%%%%%%%%%%%%%%%%%%%%%%%%%%%%%%%%%%%%%%%%%%%%%%%%%%%%%%%%%%%%%%%%%%%%%%%%%%%%%%%%%%%%%%%%%%%%%%%%%%%%%%%%%%%%%%%%
%%%%%%%%%%%%%%%%%%%%%%%%%%%%%%%%%%%%%%%%%%%%%%%%%%%%%%%%%%%%
%%%%%%%%%%%%%%%%%%%%%%%%%%%%%%%%%%%%%%%%%%%%%%%%%%%%%%%%%%%%%%%%%%%%%%%%%%%%%%%%%%%%%%%%%%%%%%%%%%%%%%%%%%%%%%%%%%%%%%%%%%%%%%%%%%%%%%%%%%%%%%%%%%%%%%%%%%%%%%%%%%%%%%%%%%%%%%%%%%%%
%%%%%%%%%%%%%%%%%%%%%%%%%%%%%%%%%%%%%%%%%%%%%%%%%%%%%%%%%%%%
%%%%%%%%%%%%%%%%%%%%%%%%%%%%%%%%%%%%%%%%%%%%%%%%%%%%%%%%%%%%%%%%%%%%%%%%%%%%%%%%%%%%%%%%%%%%%%%%%%%%%%%%%%%%%%%%%%%%%%%%%%%%%%%%%%%%%%%%%%%%%%%%%%%%%%%%%%%%%%%%%%%%%%%%%%%%%%%%%%%%
%%%%%%%%%%%%%%%%%%%%%%%%%%%%%%%%%%%%%%%%%%%%%%%%%%%%%%%%%%%%
%%%%%%%%%%%%%%%%%%%%%%%%%%%%%%%%%%%%%%%%%%%%%%%%%%%%%%%%%%%%%%%%%%%%%%%%%%%%%%%%%%%%%%%%%%%%%%%%%%%%%%%%%%%%%%%%%%%%%%%%%%%%%%%%%%%%%%%%%%%%%%%%%%%%%%%%%%%%%%%%%%%%%%%%%%%%%%%%%%%%

\subsection{Reidemeister-Schreier Algorithm for cyclic quotients}
Given a group $G$ defined by a finite presentation $\langle S \, | \, R \rangle$ and a subgroup $H$ of $G$ of finite index, the Reidemeister-Schreier algorithm (see \cite[Section 2.3]{MKS}) provides a way to compute a finite presentation of $H$. In this subsection, we explain the Reidemeister-Schreier algorithm in the case where $H$ is a normal subgroup of $G$ and $G/H$ is a finite cyclic group. Finally, under some additional assumptions, we give an explicit presentation of $H$, which we use in Section 2.3.
\begin{definition}
Let $G$ be a group with finite presentation $\langle S\, | \, R\rangle$ and $H$ be a subgroup of $G$ of finite index $n$. A subset $\{s_1,\dots,s_n\}\subset (S\cup S^{-1})^*$ is a \textbf{Schreier transversal} for $H$ in $G$ if it is a system of representatives of right cosets of $H$ in $G$ that is stable by prefixes. 
\end{definition}

If $H=\llangle S\backslash \{s\}, s^k\rrangle_G$ for some $s\in S$ and $G/H=\{1,s,\dots,s^{k-1}\}\cong \Z/k$, a Schreier Transversal for $H$ in $G$ is $\{1,s,\dots,s^{k-1}\}$. In this case, the Reidemeister-Schreier algorithm applies as follows: Consider the alphabet $\sqcup_{t\in S}\{t_i\}_{i=0,\dots,k-1}$. From now on, we consider indices modulo $k$.\\
Let then $r=r_1^{\epsilon_1}\dots r_l^{\epsilon_l}\in (S\cup S^{-1})^*$ with $(r_i,\epsilon_i)\in S\times\{1,-1\}$. We define $\tau(r)$ by the inductive formula
\begin{equation}\label{TauRS}
    \tau(r_1^{\epsilon_1}\dots r_i^{\epsilon_i})=\begin{cases} \begin{cases}(r_{1})_0 \, \text{if $i=1$ and $\epsilon_1=1$,}\\  
    (r_{1})_{-\delta_{r_1,s}}^{-1} \, \text{if $i=1$ and $\epsilon_1=-1$,}\end{cases}
    \\ \begin{cases}\tau(r_1^{\epsilon_1}\dots r_{i-1}^{\epsilon_{i-1}})(r_{i})_{\lambda(i)} \, \text{if $i>1$ and $\epsilon_i=1$},\\
    \tau(r_1^{\epsilon_1}\dots r_{i-1}^{\epsilon_{i-1}})(r_{i})_{\lambda(i)-\delta_{r_i,s}}^{-1} \, \text{if $i>1$ and $\epsilon_i=-1$,}
    \end{cases}
    \end{cases}
\end{equation}
where $\lambda(i)$ is defined as $\sum\limits_{p\leq i-1,\, r_p=s}\epsilon_p$ and $\delta_{i,j}$ is the Kronecker delta (on words).
\begin{proposition}[{Reformulation of \cite[Theorem 2.9]{MKS}}]\label{Algo}
With the above setting, the group $H$ admits the presentation 
\begin{equation}\label{SubRS}
    \left\langle \begin{array}{c|c}
    \{t_j\}_{t\in S, j\in \llbracket 0,k-1\rrbracket} \, &\,s_i=1 \, \forall i\in \llbracket 0,k-2\rrbracket; \, \tau(s^{j}rs^{j})=1 \, \forall r\in R, j\in \llbracket 0,k-1\rrbracket
    \end{array}
    \right\rangle.
\end{equation}
Moreover, in terms of the elements of $G$, we have $t_j=s^{j}ts^{-j}$ for all $t\in S\backslash\{s\}$ and $j\in \llbracket 0,k-1\rrbracket$. Finally, we have $s_j=1$ for all $j\in \llbracket 0,k-2\rrbracket$ and $s_{k-1}=s^k$.
\end{proposition}

In the case where $r\in R$ is of the form $t_{i_1}\cdots t_{i_k}t_{j_l}^{-1}\cdots t_{j_1}^{-1}$ with $t_{i_p},t_{j_p}\in S$ for all $p$ and the number of occurrences of $s$ in $(t_{i_1},\dots,t_{i_k})$ is the same as in $(t_{j_1},\dots,t_{j_l})$, Presentation \eqref{SubRS} can be simplified as follows:\\
Write $a$ to be the number of occurrences of $s$ in $(t_{i_1},\dots,t_{i_k})$. Then, in $\tau(s^prs^{-p})$, the index of $t_{i_1}$ is $p$, the index of $t_{i_2}$ is $p+\delta_{t_{i_1},s}$ and more generally the index of $t_{i_q}$ is $p+\sum_{b=1}^{q-1}\delta_{t_{i_b},s}$. Similarly, since $\sum_{b=1}^{k}\delta_{t_{i_b},s}=a$ by assumption, the index of $(t_{j_l})^{-1}$ is $p+a-\delta_{j_l,s}$ and more generally the index of $(t_{j_{l-q}})^{-1}$ is $p+a-\sum_{b=0}^q \delta_{t_{j_{l-b}},s}$. In particular, since $\sum_{b=0}^{l-1} \delta_{t_{j_{l-b}},s}=a$, the index of $(t_{j_1})^{-1}$ is $p$. Finally, observe that $a-\sum_{b=0}^q \delta_{t_{j_{l-b}},s}=\sum_{b=q+1}^{l-1} \delta_{t_{j_{l-b}},s}=\sum_{b=1}^{l-q-1} \delta_{t_{j_b},s}$. In particular, the index of $(t_{j_q})^{-1}$ is $p+\sum_{b=1}^{q-1}\delta_{t_{j_b},s}$. We thus get that $\tau(s^prs^{-p})$ is equal to 
\begin{equation}\label{tau1}
s_0\cdots s_{p-1}(t_{i_1})_p(t_{i_2})_{p+\delta_{t_{i_1},s}}\cdots (t_{i_k})_{p+\sum_{b=1}^{k-1}\delta_{t_{i_b},s}}(t_{j_{l}})^{-1}_{p+\sum_{b=1}^{l-1}\delta_{t_{j_b},s}}\cdots (t_{j_1})_p^{-1}s_{p-1}^{-1}\cdots s_0^{-1}.
\end{equation}
Using Equation \eqref{tau1}, the relation $\tau(s^prs^{-p})=1$ is then equivalent to the relation 
\begin{equation}\label{tau2}
(t_{i_1})_p(t_{i_2})_{p+\delta_{t_{i_1},s}}\cdots (t_{i_k})_{p+\sum_{b=1}^{k-1}\delta_{t_{i_b},s}}=(t_{j_1})_p(t_{j_2})_{p+\delta_{t_{j_1},s}}\cdots (t_{j_l})_{p+\sum_{b=1}^{l-1}\delta_{t_{j_b},s}}.
\end{equation}

While this description might be a bit counter intuitive at first it is actually quite simple, and we illustrate it with the following example:
\begin{example}
Let $G=\Z^2$ be presented by $\langle a,b\, | \, ab=ba\rangle$. Consider the subgroup $H=\langle a^n,b\rangle_G$. We thus have $G/H=\{1,a,\dots,a^{n-1}\}\cong \Z/n$, and combining Proposition \ref{Algo} with Equation \eqref{tau2}, a presentation for $H$ is the following:
\begin{equation*}
    \left\langle \begin{array}{c|c}
   a_0,\dots,a_{n-1}\, &\, a_i=1 \, \forall i\in \llbracket 0,n-2\rrbracket\\
   b_0,\dots,b_{n-1} \, & a_ib_{i+1}=b_ia_i \, \forall i=0,\dots,n-1
    \end{array}
    \right\rangle.
\end{equation*}
For $i=0,\dots,n-2$, the relation $a_ib_{i+1}=b_ia_i$ combined with $a_i=1$ gives $b_{i+1}=b_i$. Moreover, the relation $a_{n-1}b_0=b_{n-1}a_{n-1}$ yields $a_{n-1}b_0=b_0a_{n-1}$. This shows that $H$ admits the presentation
\begin{equation*}
    \left\langle \begin{array}{c|c}
    a_{n-1},b_0\, & \, a_{n-1}b_0=b_0a_{n-1}
    \end{array}
    \right\rangle,
\end{equation*}
and that in terms of the canonical generators $a$ and $b$ of $G$, we have $a_{n-1}=a^n$ and $b_0=b$.
\end{example}

\subsection{\texorpdfstring{Presentations of $J$-reflection groups}{}}
In \cite{BMR} Broué, Malle \& Rouquier give presentations for all complex reflection groups (see \cite[Tables 1-5]{BMR}) where the generators represent complex reflections. We refer to these presentations as the $\textbf{BMR}$ $\textbf{presentations}$.

The goal of this subsection is to give presentations of $J$-reflection groups that agree with the BMR presentation whenever the group is finite. A first step in this direction was made in \cite{Gobet Toric} by Gobet:\\
Recall from Theorem \ref{ToricPresentation} that given $k,n,m\in \N_{\geq 2}$ with $n$ and $m$ coprime, the toric reflection group $W(k,n,m)$ is isomorphic to the $J$-group $J\begin{pmatrix} k & n &m \\ & n & m\end{pmatrix}$. In particular, we have a presentation for all $J$-groups of the form $J\begin{pmatrix} k & n & m\\ &n &m \end{pmatrix}$.\\
\noindent
Whenever $W(k,n,m)$ is finite, Presentation \eqref{PresToric} agrees with the BMR presentation of the corresponding complex reflection group. Moreover, in \cite{Gobet Toric} Gobet characterises reflections of toric reflection groups in terms of the $x_i$'s:

%%%%%%%%%%%%%%%%%%%%%%%%%%%%%%%%%%%%%%%%%%%%%%%%%%%%%%%%%%%%%%%%%%%%%%%%%%%%%%%%%%%%%%%%%%%%%%%%%%%%%%%%%%%%%%%%%%%%%%%%%%%%%%%%%%%%%%%%%%%%%%%%%%%%%%%%%%%%%%%%%%%%%%%%%%%%%%%%%%%%
%%%%%%%%%%%%%%%%%%%%%%%%%%%%%%%%%%%%%%%%%%%%%%%%%%%%%%%%%%%%
%%%%%%%%%%%%%%%%%%%%%%%%%%%%%%%%%%%%%%%%%%%%%%%%%%%%%%%%%%%%%%%%%%%%%%%%%%%%%%%%%%%%%%%%%%%%%%%%%%%%%%%%%%%%%%%%%%%%%%%%
%%%%%%%%%%%%%%%%%%%%%%%%%%%%%%%%%%%%%%%%%%%%%%%%%%%%%%%%%%%%
%%%%%%%%%%%%%%%%%%%%%%%%%%%%%%%%%%%%%%%%%%%%%%%%%%%%%%%%%%%%
%%%%%%%%%%%%%%%%%%%%%%%%%%%%%%%%%%%%%%%%%%%%%%%%%%%%%%%%%%%%%%%%%%%%%%%%%%%%%%%%%%%%%%%%%%%%%%%%%%%%%%%%%%%%%%%%%%%%%%%%%%%%%%%%%%%%%%%%%%%%%%%%%%%%%%%%%%%%%%%%%%%%%%%%%%%%%%%%%%%%
%%%%%%%%%%%%%%%%%%%%%%%%%%%%%%%%%%%%%%%%%%%%%%%%%%%%%%%%%%%%
%%%%%%%%%%%%%%%%%%%%%%%%%%%%%%%%%%%%%%%%%%%%%%%%%%%%%%%%%%%%%%%%%%%%%%%%%%%%%%%%%%%%%%%%%%%%%%%%%%%%%%%%%%%%%%%%%%%%%%%%
%%%%%%%%%%%%%%%%%%%%%%%%%%%%%%%%%%%%%%%%%%%%%%%%%%%%%%%%%%%%
%%%%%%%%%%%%%%%%%%%%%%%%%%%%%%%%%%%%%%%%%%%%%%%%%%%%%%%%%%%%
%%%%%%%%%%%%%%%%%%%%%%%%%%%%%%%%%%%%%%%%%%%%%%%%%%%%%%%%%%%%%%%%%%%%%%%%%%%%%%%%%%%%%%%%%%%%%%%%%%%%%%%%%%%%%%%%%%%%%%%%%%%%%%%%%%%%%%%%%%%%%%%%%%%%%%%%%%%%%%%%%%%%%%%%%%%%%%%%%%%%
%%%%%%%%%%%%%%%%%%%%%%%%%%%%%%%%%%%%%%%%%%%%%%%%%%%%%%%%%%%%
%%%%%%%%%%%%%%%%%%%%%%%%%%%%%%%%%%%%%%%%%%%%%%%%%%%%%%%%%%%%%%%%%%%%%%
\begin{proposition}[{\cite[Proposition 2.9]{Gobet Toric}}]\label{ToricReflections}
Let $k,n,m\in \N_{\geq 2}$ with $n$ and $m$ coprime. Every reflection of $W(k,n,m)$ is conjugate in $W(k,n,m)$ (and not only in $J(k,n,m)$) to a non trivial power of some $x_i$.
\end{proposition}

In fact, Theorem \ref{ToricPresentation} and Proposition \ref{ToricReflections} extend slightly outside the scope of $J$-groups in the sense of Achar \& Aubert, as stated in the following remarks:
\begin{remark}\label{nmNotCoprime}
In the proofs of Theorem \ref{ToricPresentation} \cite[Proof of Theorem 2.12]{Gobet Toric} and Proposition \ref{ToricReflections} \cite[Proof of Proposition 2.9]{Gobet Toric}, the fact that $n$ and $m$ are coprime is not used. In particular, Theorem \ref{ToricPresentation} and Proposition \ref{ToricReflections} hold for the group $W(k,n,m):=\llangle s\rrangle_{J(k,n,m)}$ even if $n$ and $m$ are not coprime. Note that in this case $W(k,n,m)$ is not a $J$-group, nor a toric reflection group.
\end{remark}
\begin{remark}\label{ToricParameters=1}
Allowing $n$ or $m$ to be equal to $1$ in the definition of $J(k,n,m)$ includes the product of finite cyclic groups in the notation $J(k,n,m)$. In this case $\llangle s\rrangle_{J(k,n,m)}$ is equal to $\Z/k$. Moreover, in Presentation \eqref{PresToric}, allowing $n$ or $m$ to be equal to $1$ includes the group $\Z/k$ in the groups with notation $W(k,n,m)$. In particular, with this notation we still have $W(k,n,m)\cong \llangle s\rrangle_{J(k,n,m)}$ for $n$ or $m$ being equal to $1$. We \textit{do not} call this group a toric reflection group, but it will be useful for the next section to allow the notation $W(k,1,m)$ or $W(k,n,1)$ to talk about $\Z/k$. Moreover, with this notation Diagram \eqref{CommSquareFINDANAME} is still valid.
\end{remark}

With this enlarged notation, we prove a lemma which will be useful in the classification of $J$-reflection groups in terms of reflection isomorphisms:

%%%%%%%%%%%%%%%%%%%%%%%%%%%%%%%%%%%%%%%%%%%%%%%%%%%%%%%%%%%%%%%%%%%%%%%%%%%%%%%%%%%%%%%%%%%%%%%%%%%%%%%%%%%%%%%%%%%%%%%%%%%%%%%%%%%%%%%%%%%%%%%%%%%%%%%%%%%%%%%%%%%%%%%%%%%%%%%%%%%%
%%%%%%%%%%%%%%%%%%%%%%%%%%%%%%%%%%%%%%%%%%%%%%%%%%%%%%%%%%%%
%%%%%%%%%%%%%%%%%%%%%%%%%%%%%%%%%%%%%%%%%%%%%%%%%%%%%%%%%%%%%%%%%%%%%%%%%%%%%%%%%%%%%%%%%%%%%%%%%%%%%%%%%%%%%%%%%%%%%%%%
%%%%%%%%%%%%%%%%%%%%%%%%%%%%%%%%%%%%%%%%%%%%%%%%%%%%%%%%%%%%
%%%%%%%%%%%%%%%%%%%%%%%%%%%%%%%%%%%%%%%%%%%%%%%%%%%%%%%%%%%%
%%%%%%%%%%%%%%%%%%%%%%%%%%%%%%%%%%%%%%%%%%%%%%%%%%%%%%%%%%%%%%%%%%%%%%%%%%%%%%%%%%%%%%%%%%%%%%%%%%%%%%%%%%%%%%%%%%%%%%%%%%%%%%%%%%%%%%%%%%%%%%%%%%%%%%%%%%%%%%%%%%%%%%%%%%%%%%%%%%%%
%%%%%%%%%%%%%%%%%%%%%%%%%%%%%%%%%%%%%%%%%%%%%%%%%%%%%%%%%%%%
%%%%%%%%%%%%%%%%%%%%%%%%%%%%%%%%%%%%%%%%%%%%%%%%%%%%%%%%%%%%%%%%%%%%%%%%%%%%%%%%%%%%%%%%%%%%%%%%%%%%%%%%%%%%%%%%%%%%%%%%
%%%%%%%%%%%%%%%%%%%%%%%%%%%%%%%%%%%%%%%%%%%%%%%%%%%%%%%%%%%%
%%%%%%%%%%%%%%%%%%%%%%%%%%%%%%%%%%%%%%%%%%%%%%%%%%%%%%%%%%%%
%%%%%%%%%%%%%%%%%%%%%%%%%%%%%%%%%%%%%%%%%%%%%%%%%%%%%%%%%%%%%%%%%%%%%%%%%%%%%%%%%%%%%%%%%%%%%%%%%%%%%%%%%%%%%%%%%%%%%%%%%%%%%%%%%%%%%%%%%%%%%%%%%%%%%%%%%%%%%%%%%%%%%%%%%%%%%%%%%%%%
%%%%%%%%%%%%%%%%%%%%%%%%%%%%%%%%%%%%%%%%%%%%%%%%%%%%%%%%%%%%
%%%%%%%%%%%%%%%%%%%%%%%%%%%%%%%%%%%%%%%%%%%%%%%%%%%%%%%%%%%%%%%%%%%%%%
\begin{lemma}\label{ToricCyclic}
For $n,m\in \N^*$ with $n\wedge m=1$ and $k\in \N_{\geq 2}$, the group $W(k,n,m)$ is cyclic if and only if $1\in \{n,m\}$, in which case it is isomorphic to $\Z/k$. 
\end{lemma}

%%%%%%%%%%%%%%%%%%%%%%%%%%%%%%%%%%%%%%%%%%%%%%%%%%%%%%%%%%%%%%%%%%%%%%%%%%%%%%%%%%%%%%%%%%%%%%%%%%%%%%%%%%%%%%%%%%%%%%%%%%%%%%%%%%%%%%%%%%%%%%%%%%%%%%%%%%%%%%%%%%%%%%%%%%%%%%%%%%%%
%%%%%%%%%%%%%%%%%%%%%%%%%%%%%%%%%%%%%%%%%%%%%%%%%%%%%%%%%%%%
%%%%%%%%%%%%%%%%%%%%%%%%%%%%%%%%%%%%%%%%%%%%%%%%%%%%%%%%%%%%%%%%%%%%%%%%%%%%%%%%%%%%%%%%%%%%%%%%%%%%%%%%%%%%%%%%%%%%%%%%
%%%%%%%%%%%%%%%%%%%%%%%%%%%%%%%%%%%%%%%%%%%%%%%%%%%%%%%%%%%%
%%%%%%%%%%%%%%%%%%%%%%%%%%%%%%%%%%%%%%%%%%%%%%%%%%%%%%%%%%%%
%%%%%%%%%%%%%%%%%%%%%%%%%%%%%%%%%%%%%%%%%%%%%%%%%%%%%%%%%%%%%%%%%%%%%%%%%%%%%%%%%%%%%%%%%%%%%%%%%%%%%%%%%%%%%%%%%%%%%%%%%%%%%%%%%%%%%%%%%%%%%%%%%%%%%%%%%%%%%%%%%%%%%%%%%%%%%%%%%%%%
%%%%%%%%%%%%%%%%%%%%%%%%%%%%%%%%%%%%%%%%%%%%%%%%%%%%%%%%%%%%
%%%%%%%%%%%%%%%%%%%%%%%%%%%%%%%%%%%%%%%%%%%%%%%%%%%%%%%%%%%%%%%%%%%%%%%%%%%%%%%%%%%%%%%%%%%%%%%%%%%%%%%%%%%%%%%%%%%%%%%%
%%%%%%%%%%%%%%%%%%%%%%%%%%%%%%%%%%%%%%%%%%%%%%%%%%%%%%%%%%%%
%%%%%%%%%%%%%%%%%%%%%%%%%%%%%%%%%%%%%%%%%%%%%%%%%%%%%%%%%%%%
%%%%%%%%%%%%%%%%%%%%%%%%%%%%%%%%%%%%%%%%%%%%%%%%%%%%%%%%%%%%%%%%%%%%%%%%%%%%%%%%%%%%%%%%%%%%%%%%%%%%%%%%%%%%%%%%%%%%%%%%%%%%%%%%%%%%%%%%%%%%%%%%%%%%%%%%%%%%%%%%%%%%%%%%%%%%%%%%%%%%
%%%%%%%%%%%%%%%%%%%%%%%%%%%%%%%%%%%%%%%%%%%%%%%%%%%%%%%%%%%%
%%%%%%%%%%%%%%%%%%%%%%%%%%%%%%%%%%%%%%%%%%%%%%%%%%%%%%%%%%%%%%%%%%%%%%
\begin{proof}
If $1\in \{n,m\}$, this is the content of Remark \ref{ToricParameters=1}.\\
Assume now that $n,m\geq 2$. If $W(k,n,m)$ is finite, by inspecting Table 1 of \cite{AA} one sees that it is an irreducible rank two complex reflection group, none of which are cyclic. If $W(k,n,m)$ is infinite, it cannot be isomorphic to $\Z$ since it has torsion (see \cite[Proof of Lemma 2.5]{Gobet Toric}).
\end{proof}

The rest of this subsection is devoted to giving a presentation for all $J$-reflection groups generalising Presentation \eqref{PresToric}. In order to get a presentation for the group $W_b^c(k,bn,cm)$, we need some preliminary results and observations.

\begin{remark}\label{QuotientAbParent}
Let $k,b,n,c,m\in \N^*$ with $k,bn,cm\geq 2$ and $n\wedge m$=1. It follows from the definitions that \begin{equation*}
\begin{aligned}J(k,bn,cm)/W_b^c(k,bn,cm)&\cong\langle s,t,u\,|\,s=t^n=u^m=1,\, \, stu=tus=ust\rangle \\&\cong \Z/n\times \Z/m\cong \Z
/nm\end{aligned}.\end{equation*}
\end{remark}

Proposition \ref{ToricReflections} allows us to find a generating family for $W_b^c(k,bn,cm)$:

%%%%%%%%%%%%%%%%%%%%%%%%%%%%%%%%%%%%%%%%%%%%%%%%%%%%%%%%%%%%%%%%%%%%%%%%%%%%%%%%%%%%%%%%%%%%%%%%%%%%%%%%%%%%%%%%%%%%%%%%%%%%%%%%%%%%%%%%%%%%%%%%%%%%%%%%%%%%%%%%%%%%%%%%%%%%%%%%%%%%
%%%%%%%%%%%%%%%%%%%%%%%%%%%%%%%%%%%%%%%%%%%%%%%%%%%%%%%%%%%%
%%%%%%%%%%%%%%%%%%%%%%%%%%%%%%%%%%%%%%%%%%%%%%%%%%%%%%%%%%%%%%%%%%%%%%%%%%%%%%%%%%%%%%%%%%%%%%%%%%%%%%%%%%%%%%%%%%%%%%%%
%%%%%%%%%%%%%%%%%%%%%%%%%%%%%%%%%%%%%%%%%%%%%%%%%%%%%%%%%%%%
%%%%%%%%%%%%%%%%%%%%%%%%%%%%%%%%%%%%%%%%%%%%%%%%%%%%%%%%%%%%
%%%%%%%%%%%%%%%%%%%%%%%%%%%%%%%%%%%%%%%%%%%%%%%%%%%%%%%%%%%%%%%%%%%%%%%%%%%%%%%%%%%%%%%%%%%%%%%%%%%%%%%%%%%%%%%%%%%%%%%%%%%%%%%%%%%%%%%%%%%%%%%%%%%%%%%%%%%%%%%%%%%%%%%%%%%%%%%%%%%%
%%%%%%%%%%%%%%%%%%%%%%%%%%%%%%%%%%%%%%%%%%%%%%%%%%%%%%%%%%%%
%%%%%%%%%%%%%%%%%%%%%%%%%%%%%%%%%%%%%%%%%%%%%%%%%%%%%%%%%%%%%%%%%%%%%%%%%%%%%%%%%%%%%%%%%%%%%%%%%%%%%%%%%%%%%%%%%%%%%%%%
%%%%%%%%%%%%%%%%%%%%%%%%%%%%%%%%%%%%%%%%%%%%%%%%%%%%%%%%%%%%
%%%%%%%%%%%%%%%%%%%%%%%%%%%%%%%%%%%%%%%%%%%%%%%%%%%%%%%%%%%%
%%%%%%%%%%%%%%%%%%%%%%%%%%%%%%%%%%%%%%%%%%%%%%%%%%%%%%%%%%%%%%%%%%%%%%%%%%%%%%%%%%%%%%%%%%%%%%%%%%%%%%%%%%%%%%%%%%%%%%%%%%%%%%%%%%%%%%%%%%%%%%%%%%%%%%%%%%%%%%%%%%%%%%%%%%%%%%%%%%%%
%%%%%%%%%%%%%%%%%%%%%%%%%%%%%%%%%%%%%%%%%%%%%%%%%%%%%%%%%%%%
%%%%%%%%%%%%%%%%%%%%%%%%%%%%%%%%%%%%%%%%%%%%%%%%%%%%%%%%%%%%%%%%%%%%%%

\begin{lemma}\label{GeneratorsFINDANAME}
Let $k,b,n,c,m\in \N^*$ be such that $k,bn,cm\geq 2$ and $n\wedge m=1$ and let $s,t,u$ be the canonical generators of $G:=J(k,bn,cm)$. Defining $x_i$ to be $t^{i-1}st^{1-i}$ for all $i=1,\dots,bn$, $y$ to be $t^n$ and $z$ to be $u^m$, the set $\{x_1,\dots,x_n,y,z\}$ forms a generating family for $H:=W_b^c(k,bn,cm)$. Moreover, every reflection of $H$ is conjugate in $H$ (and not only in $G$) to a non trivial power of an element of $\{x_1,\dots,x_n,y,z\}$.
\end{lemma}

%%%%%%%%%%%%%%%%%%%%%%%%%%%%%%%%%%%%%%%%%%%%%%%%%%%%%%%%%%%%%%%%%%%%%%%%%%%%%%%%%%%%%%%%%%%%%%%%%%%%%%%%%%%%%%%%%%%%%%%%%%%%%%%%%%%%%%%%%%%%%%%%%%%%%%%%%%%%%%%%%%%%%%%%%%%%%%%%%%%%
%%%%%%%%%%%%%%%%%%%%%%%%%%%%%%%%%%%%%%%%%%%%%%%%%%%%%%%%%%%%
%%%%%%%%%%%%%%%%%%%%%%%%%%%%%%%%%%%%%%%%%%%%%%%%%%%%%%%%%%%%%%%%%%%%%%%%%%%%%%%%%%%%%%%%%%%%%%%%%%%%%%%%%%%%%%%%%%%%%%%%
%%%%%%%%%%%%%%%%%%%%%%%%%%%%%%%%%%%%%%%%%%%%%%%%%%%%%%%%%%%%
%%%%%%%%%%%%%%%%%%%%%%%%%%%%%%%%%%%%%%%%%%%%%%%%%%%%%%%%%%%%
%%%%%%%%%%%%%%%%%%%%%%%%%%%%%%%%%%%%%%%%%%%%%%%%%%%%%%%%%%%%%%%%%%%%%%%%%%%%%%%%%%%%%%%%%%%%%%%%%%%%%%%%%%%%%%%%%%%%%%%%%%%%%%%%%%%%%%%%%%%%%%%%%%%%%%%%%%%%%%%%%%%%%%%%%%%%%%%%%%%%
%%%%%%%%%%%%%%%%%%%%%%%%%%%%%%%%%%%%%%%%%%%%%%%%%%%%%%%%%%%%
%%%%%%%%%%%%%%%%%%%%%%%%%%%%%%%%%%%%%%%%%%%%%%%%%%%%%%%%%%%%%%%%%%%%%%%%%%%%%%%%%%%%%%%%%%%%%%%%%%%%%%%%%%%%%%%%%%%%%%%%
%%%%%%%%%%%%%%%%%%%%%%%%%%%%%%%%%%%%%%%%%%%%%%%%%%%%%%%%%%%%
%%%%%%%%%%%%%%%%%%%%%%%%%%%%%%%%%%%%%%%%%%%%%%%%%%%%%%%%%%%%
%%%%%%%%%%%%%%%%%%%%%%%%%%%%%%%%%%%%%%%%%%%%%%%%%%%%%%%%%%%%%%%%%%%%%%%%%%%%%%%%%%%%%%%%%%%%%%%%%%%%%%%%%%%%%%%%%%%%%%%%%%%%%%%%%%%%%%%%%%%%%%%%%%%%%%%%%%%%%%%%%%%%%%%%%%%%%%%%%%%%
%%%%%%%%%%%%%%%%%%%%%%%%%%%%%%%%%%%%%%%%%%%%%%%%%%%%%%%%%%%%
%%%%%%%%%%%%%%%%%%%%%%%%%%%%%%%%%%%%%%%%%%%%%%%%%%%%%%%%%%%%%%%%%%%%%%
\begin{proof}
By definition, the group $H$ is generated by elements of the form $ghg^{-1}$ with $h\in \{s,t^n,u^m\}=\{x_1,y,z\}$ and $g\in G$ and every reflection of $H$ is a power of such an element. In particular, it is enough to show that every element of the stated form is conjugate in $\langle x_1,\dots,x_n,y,z\rangle$ to an element of $\{x_1,\dots,x_n,y,z\}$ to conclude both claims of the proof.\\
Let then $r$ be of the form $gsg^{-1}$. In this case, the element $r$ lies in $W(k,bn,cm)$ as defined in Remark \ref{nmNotCoprime}, since $bn$ and $cm$ are not necessarily coprime. It follows from Proposition \ref{ToricReflections} and Remark \ref{nmNotCoprime} that $r$ is conjugate in $W(k,bn,cm)=\langle x_1,\dots,x_{bn}\rangle$ to a non trivial power of some $x_i$ with $i\in [bn]$. For all $i,k\in \N$ we have $x_{i+kn}=y^kx_iy^{-k}$ so that
\begin{equation}\label{GroupInclusion1}
\langle x_1,\dots,x_{bn}\rangle\subset \langle x_1,\dots,x_n,y\rangle,
\end{equation}
thus $r$ is conjugate in $\langle x_1,\dots,x_n,y\rangle$ to a non trivial power of some $x_i$ with $i\in [n]$.\\
Assume now that $r$ is equal to $g_1\cdots g_ly(g_1\cdots g_l)^{-1}$ with $g_1,\dots,g_l\in \{s^{\pm 1},t^{\pm 1},u^{\pm 1}\}$. We prove by induction on $l\in \N^*$ that $r$ is conjugate in $\langle x_1,\dots,x_n,y\rangle$ to $y$. To simplify the proof, observe that $$s^{k-1}=s^{-1},t^{bn-1}=t^{-1}, u^{cm-1}=u^{-1}$$ so that it is enough to show the result for $g_1,\dots,g_l\in \{s,t,u\}$. \\
If $l=1$ the computation is as follows, using that $stu$ is central:
\begin{equation*}
\begin{aligned}
sys^{-1}&=x_1sx_1^{-1},\\
tyt^{-1}&=y,\\
uyu^{-1}&=t^{-1}s^{-1}(stu)yu^{-1}=t^{-1}s^{-1}yst=x_{bn}^{-1}yx_{bn}=y^{b-1}x_n^{-1}yx_ny^{1-b}.
\end{aligned}
\end{equation*}
Assume now that the result is true for $l\in \N^*$. Since every conjugate in $G$ to one of the $x_j$'s is conjugate in $\langle x_1,\dots,x_n,y\rangle$ to a non trivial power of some $x_i$, we get the result for $l+1$.\\
To conclude the proof, by a similar induction argument it is enough to show that every element of $\{szs^{-1},tzt^{-1},uzu^{-1}\}$ is conjugate in $\langle x_1,\dots,x_n,y,z\rangle$ to $z$. Using that $ust$ is central in $G$, we have
\begin{equation*}
\begin{aligned}
szs^{-1}&=x_1zx_1^{-1},\\
tzt^{-1}&=s^{-1}u^{-1}(ust)zt^{-1}=s^{-1}u^{-1}zus=x_1^{-1}zx_1,\\
uzu^{-1}&=z.
\end{aligned}
\end{equation*}
This concludes the proof.
\end{proof}

%%%%%%%%%%%%%%%%%%%%%%%%%%%%%%%%%%%%%%%%%%%%%%%%%%%%%%%%%%%%%%%%%%%%%%%%%%%%%%%%%%%%%%%%%%%%%%%%%%%%%%%%%%%%%%%%%%%%%%%%%%%%%%%%%%%%%%%%%%%%%%%%%%%%%%%%%%%%%%%%%%%%%%%%%%%%%%%%%%%%
%%%%%%%%%%%%%%%%%%%%%%%%%%%%%%%%%%%%%%%%%%%%%%%%%%%%%%%%%%%%
%%%%%%%%%%%%%%%%%%%%%%%%%%%%%%%%%%%%%%%%%%%%%%%%%%%%%%%%%%%%%%%%%%%%%%%%%%%%%%%%%%%%%%%%%%%%%%%%%%%%%%%%%%%%%%%%%%%%%%%%
%%%%%%%%%%%%%%%%%%%%%%%%%%%%%%%%%%%%%%%%%%%%%%%%%%%%%%%%%%%%
%%%%%%%%%%%%%%%%%%%%%%%%%%%%%%%%%%%%%%%%%%%%%%%%%%%%%%%%%%%%
%%%%%%%%%%%%%%%%%%%%%%%%%%%%%%%%%%%%%%%%%%%%%%%%%%%%%%%%%%%%%%%%%%%%%%%%%%%%%%%%%%%%%%%%%%%%%%%%%%%%%%%%%%%%%%%%%%%%%%%%%%%%%%%%%%%%%%%%%%%%%%%%%%%%%%%%%%%%%%%%%%%%%%%%%%%%%%%%%%%%
%%%%%%%%%%%%%%%%%%%%%%%%%%%%%%%%%%%%%%%%%%%%%%%%%%%%%%%%%%%%
%%%%%%%%%%%%%%%%%%%%%%%%%%%%%%%%%%%%%%%%%%%%%%%%%%%%%%%%%%%%%%%%%%%%%%%%%%%%%%%%%%%%%%%%%%%%%%%%%%%%%%%%%%%%%%%%%%%%%%%%
%%%%%%%%%%%%%%%%%%%%%%%%%%%%%%%%%%%%%%%%%%%%%%%%%%%%%%%%%%%%
%%%%%%%%%%%%%%%%%%%%%%%%%%%%%%%%%%%%%%%%%%%%%%%%%%%%%%%%%%%%
%%%%%%%%%%%%%%%%%%%%%%%%%%%%%%%%%%%%%%%%%%%%%%%%%%%%%%%%%%%%%%%%%%%%%%%%%%%%%%%%%%%%%%%%%%%%%%%%%%%%%%%%%%%%%%%%%%%%%%%%%%%%%%%%%%%%%%%%%%%%%%%%%%%%%%%%%%%%%%%%%%%%%%%%%%%%%%%%%%%%
%%%%%%%%%%%%%%%%%%%%%%%%%%%%%%%%%%%%%%%%%%%%%%%%%%%%%%%%%%%%
%%%%%%%%%%%%%%%%%%%%%%%%%%%%%%%%%%%%%%%%%%%%%%%%%%%%%%%%%%%%%%%%%%%%%%
\begin{corollary}\label{NormalClosureSmall}
Let $k,b,b',n,n',c,c',m,m'\in \N^*$ with $k,bn,cm\geq 2$, $b'n'=n$, $c'm'=m$ and assume that $n\wedge m=1=n'\wedge m'$. Let $s,t,u$ be the canonical generators of $J(k,bn,cm)$. Then $$W_b^c(k,bn,cm)=\llangle x_1,\dots,x_n,t^n,u^m\rrangle_{W_{bb'}^{cc'}(k,bb'n',cc'm')}.$$
\end{corollary}

%%%%%%%%%%%%%%%%%%%%%%%%%%%%%%%%%%%%%%%%%%%%%%%%%%%%%%%%%%%%%%%%%%%%%%%%%%%%%%%%%%%%%%%%%%%%%%%%%%%%%%%%%%%%%%%%%%%%%%%%%%%%%%%%%%%%%%%%%%%%%%%%%%%%%%%%%%%%%%%%%%%%%%%%%%%%%%%%%%%%
%%%%%%%%%%%%%%%%%%%%%%%%%%%%%%%%%%%%%%%%%%%%%%%%%%%%%%%%%%%%
%%%%%%%%%%%%%%%%%%%%%%%%%%%%%%%%%%%%%%%%%%%%%%%%%%%%%%%%%%%%%%%%%%%%%%%%%%%%%%%%%%%%%%%%%%%%%%%%%%%%%%%%%%%%%%%%%%%%%%%%
%%%%%%%%%%%%%%%%%%%%%%%%%%%%%%%%%%%%%%%%%%%%%%%%%%%%%%%%%%%%
%%%%%%%%%%%%%%%%%%%%%%%%%%%%%%%%%%%%%%%%%%%%%%%%%%%%%%%%%%%%
%%%%%%%%%%%%%%%%%%%%%%%%%%%%%%%%%%%%%%%%%%%%%%%%%%%%%%%%%%%%%%%%%%%%%%%%%%%%%%%%%%%%%%%%%%%%%%%%%%%%%%%%%%%%%%%%%%%%%%%%%%%%%%%%%%%%%%%%%%%%%%%%%%%%%%%%%%%%%%%%%%%%%%%%%%%%%%%%%%%%
%%%%%%%%%%%%%%%%%%%%%%%%%%%%%%%%%%%%%%%%%%%%%%%%%%%%%%%%%%%%
%%%%%%%%%%%%%%%%%%%%%%%%%%%%%%%%%%%%%%%%%%%%%%%%%%%%%%%%%%%%%%%%%%%%%%%%%%%%%%%%%%%%%%%%%%%%%%%%%%%%%%%%%%%%%%%%%%%%%%%%
%%%%%%%%%%%%%%%%%%%%%%%%%%%%%%%%%%%%%%%%%%%%%%%%%%%%%%%%%%%%
%%%%%%%%%%%%%%%%%%%%%%%%%%%%%%%%%%%%%%%%%%%%%%%%%%%%%%%%%%%%
%%%%%%%%%%%%%%%%%%%%%%%%%%%%%%%%%%%%%%%%%%%%%%%%%%%%%%%%%%%%%%%%%%%%%%%%%%%%%%%%%%%%%%%%%%%%%%%%%%%%%%%%%%%%%%%%%%%%%%%%%%%%%%%%%%%%%%%%%%%%%%%%%%%%%%%%%%%%%%%%%%%%%%%%%%%%%%%%%%%%
%%%%%%%%%%%%%%%%%%%%%%%%%%%%%%%%%%%%%%%%%%%%%%%%%%%%%%%%%%%%
%%%%%%%%%%%%%%%%%%%%%%%%%%%%%%%%%%%%%%%%%%%%%%%%%%%%%%%%%%%%%%%%%%%%%%
\begin{proof}
By definition, the group $W_b^c(k,bn,cm)$ is a normal subgroup of $J(k,bn,cm)$ hence it is normal in $W_{bb'}^{cc'}(k,bb'n',cc'm')$. By Lemma \ref{GeneratorsFINDANAME}, the group $W_b^c(k,bn,cm)$ is generated by $\{x_1,\dots,x_n,t^n,u^m\}$ so that
\begin{equation*}
\begin{aligned}
W_b^c(k,bn,cm)&=\llangle W_b^c(k,bn,cm)\rrangle_{W_{bb'}^{cc'}(k,bb'n',cc'm')}\\&=\llangle x_1,\dots,x_n,t^n,u^m\rrangle_{W_{bb'}^{cc'}(k,bb'n',cc'm')},
\end{aligned}
\end{equation*}
which concludes the proof.
\end{proof}

%%%%%%%%%%%%%%%%%%%%%%%%%%%%%%%%%%%%%%%%%%%%%%%%%%%%%%%%%%%%%%%%%%%%%%%%%%%%%%%%%%%%%%%%%%%%%%%%%%%%%%%%%%%%%%%%%%%%%%%%%%%%%%%%%%%%%%%%%%%%%%%%%%%%%%%%%%%%%%%%%%%%%%%%%%%%%%%%%%%%
%%%%%%%%%%%%%%%%%%%%%%%%%%%%%%%%%%%%%%%%%%%%%%%%%%%%%%%%%%%%
%%%%%%%%%%%%%%%%%%%%%%%%%%%%%%%%%%%%%%%%%%%%%%%%%%%%%%%%%%%%%%%%%%%%%%%%%%%%%%%%%%%%%%%%%%%%%%%%%%%%%%%%%%%%%%%%%%%%%%%%
%%%%%%%%%%%%%%%%%%%%%%%%%%%%%%%%%%%%%%%%%%%%%%%%%%%%%%%%%%%%
%%%%%%%%%%%%%%%%%%%%%%%%%%%%%%%%%%%%%%%%%%%%%%%%%%%%%%%%%%%%
%%%%%%%%%%%%%%%%%%%%%%%%%%%%%%%%%%%%%%%%%%%%%%%%%%%%%%%%%%%%%%%%%%%%%%%%%%%%%%%%%%%%%%%%%%%%%%%%%%%%%%%%%%%%%%%%%%%%%%%%%%%%%%%%%%%%%%%%%%%%%%%%%%%%%%%%%%%%%%%%%%%%%%%%%%%%%%%%%%%%
%%%%%%%%%%%%%%%%%%%%%%%%%%%%%%%%%%%%%%%%%%%%%%%%%%%%%%%%%%%%
%%%%%%%%%%%%%%%%%%%%%%%%%%%%%%%%%%%%%%%%%%%%%%%%%%%%%%%%%%%%%%%%%%%%%%%%%%%%%%%%%%%%%%%%%%%%%%%%%%%%%%%%%%%%%%%%%%%%%%%%
%%%%%%%%%%%%%%%%%%%%%%%%%%%%%%%%%%%%%%%%%%%%%%%%%%%%%%%%%%%%
%%%%%%%%%%%%%%%%%%%%%%%%%%%%%%%%%%%%%%%%%%%%%%%%%%%%%%%%%%%%
%%%%%%%%%%%%%%%%%%%%%%%%%%%%%%%%%%%%%%%%%%%%%%%%%%%%%%%%%%%%%%%%%%%%%%%%%%%%%%%%%%%%%%%%%%%%%%%%%%%%%%%%%%%%%%%%%%%%%%%%%%%%%%%%%%%%%%%%%%%%%%%%%%%%%%%%%%%%%%%%%%%%%%%%%%%%%%%%%%%%
%%%%%%%%%%%%%%%%%%%%%%%%%%%%%%%%%%%%%%%%%%%%%%%%%%%%%%%%%%%%
%%%%%%%%%%%%%%%%%%%%%%%%%%%%%%%%%%%%%%%%%%%%%%%%%%%%%%%%%%%%%%%%%%%%%%

We are now ready to compute a presentation of $W_b^c(k,bn,cm)$, using Reidemeister-Schreier Algorithm. We separate the computation in several results, starting by proving a useful isomorphism:

\begin{lemma}\label{PresBMR}
Let $n\in \N_{\geq 2}$ and let $G$ be the group with presentation 
\[\left\langle\begin{array}{c|c}
    x_0,\dots,x_{n-1} & x_0x_1=x_ix_{i+1}\, \forall i=1,\dots,n-2  \\
   z & x_0x_1z=x_{n-1}zx_0=zx_0x_1 
\end{array}
\right\rangle.\]
Then $G$ admits the following presentation:
\begin{equation*}\left\langle\begin{array}{c|c}
    x_0,x_1,z & x_0x_1z=zx_0x_1,\, \, x_1z\prod(n-1;x_0,x_1)=z\prod(n;x_0,x_1) \\
\end{array}\right\rangle,\end{equation*}
where $\prod(k;x_0,x_1)$ denotes the word $x_0x_1x_0\cdots$ of length $k$.
Moreover, the words $x_0,x_1$ and $z$ represent the same element in both presentations.
\end{lemma}

%%%%%%%%%%%%%%%%%%%%%%%%%%%%%%%%%%%%%%%%%%%%%%%%%%%%%%%%%%%%%%%%%%%%%%%%%%%%%%%%%%%%%%%%%%%%%%%%%%%%%%%%%%%%%%%%%%%%%%%%%%%%%%%%%%%%%%%%%%%%%%%%%%%%%%%%%%%%%%%%%%%%%%%%%%%%%%%%%%%%
%%%%%%%%%%%%%%%%%%%%%%%%%%%%%%%%%%%%%%%%%%%%%%%%%%%%%%%%%%%%
%%%%%%%%%%%%%%%%%%%%%%%%%%%%%%%%%%%%%%%%%%%%%%%%%%%%%%%%%%%%%%%%%%%%%%%%%%%%%%%%%%%%%%%%%%%%%%%%%%%%%%%%%%%%%%%%%%%%%%%%
%%%%%%%%%%%%%%%%%%%%%%%%%%%%%%%%%%%%%%%%%%%%%%%%%%%%%%%%%%%%
%%%%%%%%%%%%%%%%%%%%%%%%%%%%%%%%%%%%%%%%%%%%%%%%%%%%%%%%%%%%
%%%%%%%%%%%%%%%%%%%%%%%%%%%%%%%%%%%%%%%%%%%%%%%%%%%%%%%%%%%%%%%%%%%%%%%%%%%%%%%%%%%%%%%%%%%%%%%%%%%%%%%%%%%%%%%%%%%%%%%%%%%%%%%%%%%%%%%%%%%%%%%%%%%%%%%%%%%%%%%%%%%%%%%%%%%%%%%%%%%%
%%%%%%%%%%%%%%%%%%%%%%%%%%%%%%%%%%%%%%%%%%%%%%%%%%%%%%%%%%%%
%%%%%%%%%%%%%%%%%%%%%%%%%%%%%%%%%%%%%%%%%%%%%%%%%%%%%%%%%%%%%%%%%%%%%%%%%%%%%%%%%%%%%%%%%%%%%%%%%%%%%%%%%%%%%%%%%%%%%%%%
%%%%%%%%%%%%%%%%%%%%%%%%%%%%%%%%%%%%%%%%%%%%%%%%%%%%%%%%%%%%
%%%%%%%%%%%%%%%%%%%%%%%%%%%%%%%%%%%%%%%%%%%%%%%%%%%%%%%%%%%%
%%%%%%%%%%%%%%%%%%%%%%%%%%%%%%%%%%%%%%%%%%%%%%%%%%%%%%%%%%%%%%%%%%%%%%%%%%%%%%%%%%%%%%%%%%%%%%%%%%%%%%%%%%%%%%%%%%%%%%%%%%%%%%%%%%%%%%%%%%%%%%%%%%%%%%%%%%%%%%%%%%%%%%%%%%%%%%%%%%%%
%%%%%%%%%%%%%%%%%%%%%%%%%%%%%%%%%%%%%%%%%%%%%%%%%%%%%%%%%%%%
%%%%%%%%%%%%%%%%%%%%%%%%%%%%%%%%%%%%%%%%%%%%%%%%%%%%%%%%%%%%%%%%%%%%%%
\begin{proof}
We claim that the set of relations 
\begin{equation}\label{PresBMReq1}
    x_0x_1=x_ix_{i+1}\, \forall i=1,\dots,n-2
\end{equation}
is equivalent to the set of relations
\begin{equation}\label{PresBMReq2}
x_j=\begin{cases} \prod(j-1;x_1,x_0)^{-1}x_0\prod(j-1;x_1,x_0) \, \, \text{if} \,\, j \, \, \text{is even}\\
\prod(j-1;x_0,x_1)^{-1}x_1\prod(j-1;x_0,x_1)\, \, \text{if} \,\, j \, \, \text{is odd}
\end{cases} \, \, \forall j=0,\dots,n-1.
\end{equation}
\fbox{\eqref{PresBMReq1}$\Rightarrow $\eqref{PresBMReq2}:}
We prove this implication by induction on $j$. For $j=0,1$, this is clear. Assume this is true up to $j$ even. Then we have
\begin{eqnarray*}
x_0x_1&=&x_jx_{j+1}\\
&=&\prod(j-1;x_1,x_0)^{-1}x_0\prod(j-1;x_1,x_0)x_{j+1} \,\, \text{by induction hypothesis,}
\end{eqnarray*}
which is equivalent to $x_{j+1}=\prod(j;x_0,x_1)^{-1}x_1\prod(j;x_0,x_1)$. The proof for $j$ odd is entirely similar.\\
\fbox{\eqref{PresBMReq2}$\Rightarrow$\eqref{PresBMReq1}:} Let $i\in [n-1]$ be even. We have
\begin{eqnarray*}
x_ix_{i+1}&=& \prod(i-1;x_1,x_0)^{-1}x_0\prod(i-1;x_1,x_0)\prod(i;x_0,x_1)^{-1}x_1\prod(i;x_0,x_1)\\
&=&\prod(i-1;x_1,x_0)^{-1}x_1\prod(i;x_0,x_1)\\
&=&x_0x_1.
\end{eqnarray*}
The proof for $i$ odd  is entirely similar. This shows that equations \eqref{PresBMReq1} and \eqref{PresBMReq2} are equivalent.\\
Now, if $n$ is even, using \eqref{PresBMReq2} the relation $x_0x_1z=x_{n-1}zx_0$ can be rewritten as
$$x_0x_1z=\prod(n-2;x_0,x_1)^{-1}x_1\prod(n-2;x_0,x_1)zx_0,$$
which is equivalent to 
$$\prod(n-2;x_0,x_1)x_0x_1z=x_1\prod(n-2,x_0,x_1)zx_0.$$
Since $z$ commutes with $x_0x_1$, we get the claimed presentation.\\
If $n$ is odd, using \eqref{PresBMReq2} the relation $x_0x_1z=x_{n-1}zx_0$ can be rewritten as 
$$x_0x_1z=\prod(n-2;x_1,x_0)^{-1}x_0\prod(n-2;x_1,x_0)zx_0,$$
which is equivalent to
$$\prod(n;x_1,x_0)z=x_0\prod(n-2;x_1,x_0)zx_0.$$
Since $z$ commutes with $x_0x_1$, we get the claimed presentation.
\end{proof}

%%%%%%%%%%%%%%%%%%%%%%%%%%%%%%%%%%%%%%%%%%%%%%%%%%%%%%%%%%%%%%%%%%%%%%%%%%%%%%%%%%%%%%%%%%%%%%%%%%%%%%%%%%%%%%%%%%%%%%%%%%%%%%%%%%%%%%%%%%%%%%%%%%%%%%%%%%%%%%%%%%%%%%%%%%%%%%%%%%%%
%%%%%%%%%%%%%%%%%%%%%%%%%%%%%%%%%%%%%%%%%%%%%%%%%%%%%%%%%%%%
%%%%%%%%%%%%%%%%%%%%%%%%%%%%%%%%%%%%%%%%%%%%%%%%%%%%%%%%%%%%%%%%%%%%%%%%%%%%%%%%%%%%%%%%%%%%%%%%%%%%%%%%%%%%%%%%%%%%%%%%
%%%%%%%%%%%%%%%%%%%%%%%%%%%%%%%%%%%%%%%%%%%%%%%%%%%%%%%%%%%%
%%%%%%%%%%%%%%%%%%%%%%%%%%%%%%%%%%%%%%%%%%%%%%%%%%%%%%%%%%%%
%%%%%%%%%%%%%%%%%%%%%%%%%%%%%%%%%%%%%%%%%%%%%%%%%%%%%%%%%%%%%%%%%%%%%%%%%%%%%%%%%%%%%%%%%%%%%%%%%%%%%%%%%%%%%%%%%%%%%%%%%%%%%%%%%%%%%%%%%%%%%%%%%%%%%%%%%%%%%%%%%%%%%%%%%%%%%%%%%%%%
%%%%%%%%%%%%%%%%%%%%%%%%%%%%%%%%%%%%%%%%%%%%%%%%%%%%%%%%%%%%
%%%%%%%%%%%%%%%%%%%%%%%%%%%%%%%%%%%%%%%%%%%%%%%%%%%%%%%%%%%%%%%%%%%%%%%%%%%%%%%%%%%%%%%%%%%%%%%%%%%%%%%%%%%%%%%%%%%%%%%%
%%%%%%%%%%%%%%%%%%%%%%%%%%%%%%%%%%%%%%%%%%%%%%%%%%%%%%%%%%%%
%%%%%%%%%%%%%%%%%%%%%%%%%%%%%%%%%%%%%%%%%%%%%%%%%%%%%%%%%%%%
%%%%%%%%%%%%%%%%%%%%%%%%%%%%%%%%%%%%%%%%%%%%%%%%%%%%%%%%%%%%%%%%%%%%%%%%%%%%%%%%%%%%%%%%%%%%%%%%%%%%%%%%%%%%%%%%%%%%%%%%%%%%%%%%%%%%%%%%%%%%%%%%%%%%%%%%%%%%%%%%%%%%%%%%%%%%%%%%%%%%
%%%%%%%%%%%%%%%%%%%%%%%%%%%%%%%%%%%%%%%%%%%%%%%%%%%%%%%%%%%%
%%%%%%%%%%%%%%%%%%%%%%%%%%%%%%%%%%%%%%%%%%%%%%%%%%%%%%%%%%%%%%%%%%%%%%

\begin{proposition}\label{PresentationNorM1}
Let $k,b,n,c,m\in \N^*$ be such that $k,bn,cm\geq 2$. The group $W_{bn}^c(k,bn,cm)$ admits the presentation

\begin{equation}\label{NorM1}\left\langle\begin{array}{c|c}
     &  s^k=t^{bn}=z^c=1 \\
   s,t,z  & stz=zst\\
   & tz(st)^{m-1}s=z(st)^m
   
\end{array}
\right\rangle.\end{equation}
In terms of the canonical generators of $J(k,bn,cm)$, we have $s=s,t=t$ and $z=u^m$.
\end{proposition}

%%%%%%%%%%%%%%%%%%%%%%%%%%%%%%%%%%%%%%%%%%%%%%%%%%%%%%%%%%%%%%%%%%%%%%%%%%%%%%%%%%%%%%%%%%%%%%%%%%%%%%%%%%%%%%%%%%%%%%%%%%%%%%%%%%%%%%%%%%%%%%%%%%%%%%%%%%%%%%%%%%%%%%%%%%%%%%%%%%%%
%%%%%%%%%%%%%%%%%%%%%%%%%%%%%%%%%%%%%%%%%%%%%%%%%%%%%%%%%%%%
%%%%%%%%%%%%%%%%%%%%%%%%%%%%%%%%%%%%%%%%%%%%%%%%%%%%%%%%%%%%%%%%%%%%%%%%%%%%%%%%%%%%%%%%%%%%%%%%%%%%%%%%%%%%%%%%%%%%%%%%
%%%%%%%%%%%%%%%%%%%%%%%%%%%%%%%%%%%%%%%%%%%%%%%%%%%%%%%%%%%%
%%%%%%%%%%%%%%%%%%%%%%%%%%%%%%%%%%%%%%%%%%%%%%%%%%%%%%%%%%%%
%%%%%%%%%%%%%%%%%%%%%%%%%%%%%%%%%%%%%%%%%%%%%%%%%%%%%%%%%%%%%%%%%%%%%%%%%%%%%%%%%%%%%%%%%%%%%%%%%%%%%%%%%%%%%%%%%%%%%%%%%%%%%%%%%%%%%%%%%%%%%%%%%%%%%%%%%%%%%%%%%%%%%%%%%%%%%%%%%%%%
%%%%%%%%%%%%%%%%%%%%%%%%%%%%%%%%%%%%%%%%%%%%%%%%%%%%%%%%%%%%
%%%%%%%%%%%%%%%%%%%%%%%%%%%%%%%%%%%%%%%%%%%%%%%%%%%%%%%%%%%%%%%%%%%%%%%%%%%%%%%%%%%%%%%%%%%%%%%%%%%%%%%%%%%%%%%%%%%%%%%%
%%%%%%%%%%%%%%%%%%%%%%%%%%%%%%%%%%%%%%%%%%%%%%%%%%%%%%%%%%%%
%%%%%%%%%%%%%%%%%%%%%%%%%%%%%%%%%%%%%%%%%%%%%%%%%%%%%%%%%%%%
%%%%%%%%%%%%%%%%%%%%%%%%%%%%%%%%%%%%%%%%%%%%%%%%%%%%%%%%%%%%%%%%%%%%%%%%%%%%%%%%%%%%%%%%%%%%%%%%%%%%%%%%%%%%%%%%%%%%%%%%%%%%%%%%%%%%%%%%%%%%%%%%%%%%%%%%%%%%%%%%%%%%%%%%%%%%%%%%%%%%
%%%%%%%%%%%%%%%%%%%%%%%%%%%%%%%%%%%%%%%%%%%%%%%%%%%%%%%%%%%%
%%%%%%%%%%%%%%%%%%%%%%%%%%%%%%%%%%%%%%%%%%%%%%%%%%%%%%%%%%%%%%%%%%%%%%

\begin{proof}[Proof of Proposition \ref{PresentationNorM1}]
By Remark \ref{QuotientAbParent}, the group $J(k,bn,cm)/W_{bn}^c(k,bn,cm)$ is isomorphic to $\Z/m$ and a Schreier transversal for $W_{bn}^c(k,bn,cm)$ is $\{1,u,\dots,u^{m-1}\}$. Using Reidemeister-Schreier algorithm as in Proposition \ref{Algo} with Equation \eqref{tau2} and considering Presentation \eqref{PresParent} for $J(k,bn,cm)$, the group $W_{bn}^c(k,bn,cm)$ admits the following presentation: 

\begin{equation*}\left\langle 
       \begin{array}{l|cl}
           s_0,\dots,s_{m-1}     \,        & \, u_i=1 \,\, \forall i=0,\dots,m-2 \\
     t_0,\dots,t_{m-1} \,  & \,s_i^k=t_i^{bn}=(u_iu_{i+1}\cdots u_{i+m-1})^c=1\,\, \forall i=0,\dots,m-1 \\
        u_0,\dots,u_{m-1}     \,        &\, s_it_iu_i=t_iu_is_{i+1}=u_is_{i+1}t_{i+1} \,\, \forall i=0,\dots,m-1   \\
                          \end{array}
     \right\rangle.\end{equation*}
where indices are taken modulo $m$.\\
The set of relations $u_i=1;\, (u_iu_{i+1}\cdots u_{i+m-1})^c=1\,\, \forall i\in\llbracket 0,m-1\rrbracket$ implies $u_{m-1}^c=1$. Moreover, for each $0\leq i \leq m-2$, the relation $s_it_iu_i=t_iu_is_{i+1}=u_is_{i+1}t_{i+1}$ reads $s_it_i=t_is_{i+1}=s_{i+1}t_{i+1}$ and for $i=m-1$ we get $s_{m-1}t_{m-1}u_{m-1}=t_{m-1}u_{m-1}s_0=u_{m-1}s_0t_0$.
We thus get the following presentation for $W_{bn}^c(k,bn,cm)$:

\begin{equation}\label{ReidemeisterSchreier2}\left\langle 
       \begin{array}{l|cl}
           s_0,\dots,s_{m-1}     \,        &\, u_{m-1}^c=1,\,\,s_i^k=t_i^{bn}=1\,\, \forall i=0,\dots,m-1\\
     t_0,\dots,t_{m-1} \,  & \, s_it_i=t_is_{i+1}=s_{i+1}t_{i+1} \, \, \forall i=0,\dots,m-2\\
  \, \, \,\,\,\,\,\,\,u_{m-1}     \,        &\, s_{m-1}t_{m-1}u_{m-1}=t_{m-1}u_{m-1}s_0=u_{m-1}s_0t_0  \\
                          \end{array}
     \right\rangle.\end{equation}
\noindent
Now, observe that for all $i=0,\dots,m-2$, the relation $s_it_i=t_is_{i+1}$ implies that $s_i$ and $s_{i+1}$ are conjugate in $W_{bn}^c(k,bn,cm)$, thus all $s_i$'s are conjugate to each other in $W_{bn}^c(k,bn,cm)$. Similarly, all $t_i$'s are conjugate to each other in $W_{bn}^c(k,bn,cm)$. In particular, the relations $z^c=s_0^k=t_0^{bn}=1$ are enough to describe the first line of Presentation \eqref{ReidemeisterSchreier2}.\\
Writing $x_{2i}:=s_i$ and $x_{2i+1}=t_i$ for all $i=0,\dots,m-1$ and $z:=u_{m-1}$, we thus get the following presentation for $W_{bn}^c(k,bn,cm)$:
\begin{equation*}\left\langle 
       \begin{array}{l|cl}
             &\, z^c=1,\,\,x_{0}^{k}=x_{1}^{bn}=1\\
 x_0,\dots,x_{2m-1},z \,  & \, x_{2i}x_{2i+1}=x_{2i+1}x_{2i+2}=x_{2i+2}x_{2i+3} \, \, \forall i=0,\dots,m-2\\
  &\, x_{2m-2}x_{2m-1}z=x_{2m-1}zx_0=zx_0x_1  \\
                          \end{array}
     \right\rangle.\end{equation*}
The set of relations  $x_{2i}x_{2i+1}=x_{2i+1}x_{2i+2}=x_{2i+2}x_{2i+3} \, \, \forall i=0,\dots,m-2$ is equivalent to the set of relations $x_0x_1=x_ix_{i+1} \, \, \forall i=0,\dots,2m-2$.\\
Moreover, since $x_{2m-2}x_{2m-1}=x_0x_1$, the relations $x_{2m-2}x_{2m-1}z=x_{2m-1}zx_0=zx_0x_1$ can be written $x_0x_1z=x_{2m-1}zx_0=zx_0x_1$.\\
Thus, the group $W_{bn}^c(k,bn,cm)$ admits the presentation 

\begin{equation*}\left\langle 
       \begin{array}{l|cl}
             &\, z^c=1,\,\,x_{0}^{k}=x_{1}^{bn}=1\\
 x_0,\dots,x_{2m-1},z \,  & \, x_0x_1=x_ix_{i+1} \, \, \forall i=1,\dots,2m-2\\
  &\, x_0x_1z=x_{2m-1}zx_0=zx_0x_1  \\
                          \end{array}
     \right\rangle.\end{equation*}
By Lemma \ref{PresentationNorM1}, the group $W_{bn}^c(k,bn,cm)$ admits the presentation 

\begin{equation*}\left\langle 
       \begin{array}{l|cl}
             &\, z^c=1,\,\,x_0^k=x_1^{bn}=1\\
 x_0,x_1,z \,  & \, x_0x_1z=zx_0x_1\\
  &\, x_1z(x_0x_1)^{m-1}x_0=z(x_0x_1)^m  \\
                          \end{array}
     \right\rangle.\end{equation*}
By Proposition \ref{Algo} we have $x_0=s_0=s$, $x_1=t_0=t$ and $z=u_{m-1}=u^m$, which concludes the proof. 
\end{proof}

\begin{remark}\label{PresentationBMR}
Consulting \cite[Tables 1-4]{BMR}, one sees that the presentation given in  Proposition \ref{PresentationNorM1} coincides with the BMR presentation for $G(2cd,2c,2)\cong W_2^c(2,2,cd)$ and $G_{15}\cong W_2^3(2,4,3)$.
\end{remark}
%%%%%%%%%%%%%%%%%%%%%%%%%%%%%%%%%%%%%%%%%%%%%%%%%%%%%%%%%%%%%%%%%%%%%%%%%%%%%%%%%%%%%%%%%%%%%%%%%%%%%%%%%%%%%%%%%%%%%%%%%%%%%%%%%%%%%%%%%%%%%%%%%%%%%%%%%%%%%%%%%%%%%%%%%%%%%%%%%%%%
%%%%%%%%%%%%%%%%%%%%%%%%%%%%%%%%%%%%%%%%%%%%%%%%%%%%%%%%%%%%
%%%%%%%%%%%%%%%%%%%%%%%%%%%%%%%%%%%%%%%%%%%%%%%%%%%%%%%%%%%%%%%%%%%%%%%%%%%%%%%%%%%%%%%%%%%%%%%%%%%%%%%%%%%%%%%%%%%%%%%%
%%%%%%%%%%%%%%%%%%%%%%%%%%%%%%%%%%%%%%%%%%%%%%%%%%%%%%%%%%%%
%%%%%%%%%%%%%%%%%%%%%%%%%%%%%%%%%%%%%%%%%%%%%%%%%%%%%%%%%%%%
%%%%%%%%%%%%%%%%%%%%%%%%%%%%%%%%%%%%%%%%%%%%%%%%%%%%%%%%%%%%%%%%%%%%%%%%%%%%%%%%%%%%%%%%%%%%%%%%%%%%%%%%%%%%%%%%%%%%%%%%%%%%%%%%%%%%%%%%%%%%%%%%%%%%%%%%%%%%%%%%%%%%%%%%%%%%%%%%%%%%
%%%%%%%%%%%%%%%%%%%%%%%%%%%%%%%%%%%%%%%%%%%%%%%%%%%%%%%%%%%%
%%%%%%%%%%%%%%%%%%%%%%%%%%%%%%%%%%%%%%%%%%%%%%%%%%%%%%%%%%%%%%%%%%%%%%%%%%%%%%%%%%%%%%%%%%%%%%%%%%%%%%%%%%%%%%%%%%%%%%%%
%%%%%%%%%%%%%%%%%%%%%%%%%%%%%%%%%%%%%%%%%%%%%%%%%%%%%%%%%%%%
%%%%%%%%%%%%%%%%%%%%%%%%%%%%%%%%%%%%%%%%%%%%%%%%%%%%%%%%%%%%
%%%%%%%%%%%%%%%%%%%%%%%%%%%%%%%%%%%%%%%%%%%%%%%%%%%%%%%%%%%%%%%%%%%%%%%%%%%%%%%%%%%%%%%%%%%%%%%%%%%%%%%%%%%%%%%%%%%%%%%%%%%%%%%%%%%%%%%%%%%%%%%%%%%%%%%%%%%%%%%%%%%%%%%%%%%%%%%%%%%%
%%%%%%%%%%%%%%%%%%%%%%%%%%%%%%%%%%%%%%%%%%%%%%%%%%%%%%%%%%%%
%%%%%%%%%%%%%%%%%%%%%%%%%%%%%%%%%%%%%%%%%%%%%%%%%%%%%%%%%%%%%%%%%%%%%%

\begin{theorem}\label{Pres2}
Let $k,b,n,c,m\in \N^*$ be such that $k,bn,cm\geq 2$ and assume $n\wedge m=1$. Write $m=qn+r$ with $q\geq 0$ and $0\leq r\leq n-1$\footnote{Since $n\wedge m=1$, the case $r=0$ corresponds to $n=1$.}. Then $W_b^c(k,bn,cm)$ admits the following presentation:
\begin{equation}\label{GeneralPresW} 
\begin{aligned}
&(1) \,\, \mathrm{Generators}\!:\,  \{x_1,\dots,x_n,y,z\};\\
&(2) \,\, \mathrm{Relations}\!: \,\\
&x_i^k=y^b=z^c=1\, \forall i=1,\cdots,n, \\
  &  x_1\cdots x_nyz=zx_1\cdots x_ny,\\
       & x_{i+1}\cdots x_nyz\delta^{q-1}x_1\cdots x_{i+r}=x_i\cdots x_nyz\delta^{q-1}x_1\cdots x_{i+r-1}, \, \forall 1\leq i \leq n-r,\\
      &  x_{i+1}\cdots x_nyz\delta^qx_1\cdots x_{i+r-n}=x_i\cdots x_nyz\delta^qx_1\cdots x_{i+r-n-1},\, \forall n-r+1\leq i \leq n.
\end{aligned}
\end{equation}
where $\delta$ denotes $x_1\cdots x_ny$.\\
In terms of the canonical generators of $J(k,bn,cm)$, we have $x_i=t^{i-1}st^{1-i}$ for all $i\in [n]$, $y=t^n$ and $z=u^m$.
\end{theorem}

Before giving the proof of Theorem \ref{Pres2}, we give some presentations illustrating Theorem \ref{Pres2} and state some corollaries.

\begin{example}
A presentation of $W_2^3(4,6,12)=J\begin{pmatrix} 4 & 6 & 12 \\ & 3 & 4 \end{pmatrix}$ is

\begin{equation*}\left\langle
\begin{array}{l|cl}
& x_1^4=x_2^4=x_3^4=y^2=z^3=1, \\
x_1,x_2,x_3\,&\, x_1x_2x_3yz=zx_1x_2x_3y\\
\,\,\,\,\,\,\,y,z\,& \,x_1x_2x_3yzx_1=x_2x_3yzx_1x_2=x_3yzx_1x_2x_3\\
& x_3yzx_1x_2x_3y=yzx_1x_2x_3yx_1
\end{array}
\right\rangle.\end{equation*}
\end{example}

%%%%%%%%%%%%%%%%%%%%%%%%%%%%%%%%%%%%%%%%%%%%%%%%%%%%%%%%%%%%%%%%%%%%%%%%%%%%%%%%%%%%%%%%%%%%%%%%%%%%%%%
%%%%%%%%%%%%%%%%%%%%%%%%%%%%%%%%%%%%%%%%%%%%%%%%%%%%%%%%%%%%%%%%%%%%%%%%%%%%%%%%%%%%%%%%%%%%%%%%%%%%%%%
%%%%%%%%%%%%%%%%%%%%%%%%%%%%%%%%%%%%%%%%%%%%%%%%%%%%%%%%%%%%%%%%%%%%%%%%%%%%%%%%%%%%%%%%%%%%%%%%%%%%%%%
%%%%%%%%%%%%%%%%%%%%%%%%%%%%%%%%%%%%%%%%%%%%%%%%%%%%%%%%%%%%%%%%%%%%%%%%%%%%%%%%%%%%%%%%%%%%%%%%%%%%%%%
%%%%%%%%%%%%%%%%%%%%%%%%%%%%%%%%%%%%%%%%%%%%%%%%%%%%%%%%%%%%%%%%%%%%%%%%%%%%%%%%%%%%%%%%%%%%%%%%%%%%%%%
%%%%%%%%%%%%%%%%%%%%%%%%%%%%%%%%%%%%%%%%%%%%%%%%%%%%%%%%%%%%%%%%%%%%%%%%%%%%%%%%%%%%%%%%%%%%%%%%%%%%%%%
%%%%%%%%%%%%%%%%%%%%%%%%%%%%%%%%%%%%%%%%%%%%%%%%%%%%%%%%%%%%%%%%%%%%%%%%%%%%%%%%%%%%%%%%%%%%%%%%%%%%%%%
%%%%%%%%%%%%%%%%%%%%%%%%%%%%%%%%%%%%%%%%%%%%%%%%%%%%%%%%%%%%%%%%%%%%%%%%%%%%%%%%%%%%%%%%%%%%%%%%%%%%%%%
%%%%%%%%%%%%%%%%%%%%%%%%%%%%%%%%%%%%%%%%%%%%%%%%%%%%%%%%%%%%%%%%%%%%%%%%%%%%%%%%%%%%%%%%%%%%%%%%%%%%%%%
%%%%%%%%%%%%%%%%%%%%%%%%%%%%%%%%%%%%%%%%%%%%%%%%%%%%%%%%%%%%%%%%%%%%%%%%%%%%%%%%%%%%%%%%%%%%%%%%%%%%%%%
%%%%%%%%%%%%%%%%%%%%%%%%%%%%%%%%%%%%%%%%%%%%%%%%%%%%%%%%%%%%%%%%%%%%%%%%%%%%%%%%%%%%%%%%%%%%%%%%%%%%%%%

\begin{corollary}\label{Pres2z=1}
Let $k,b,n,m\in \N^*$ with $k,bn,m\geq 2$ and assume $n\wedge m=1$. Write $m=qn+r$ with $q\geq 0$ and $0\leq r\leq n-1$. Then $W_b(k,bn,m)$ admits the following presentation:

\begin{equation}\label{PresPres2z=1}
\begin{aligned}
&(1) \,\, \mathrm{Generators}\!:\,  \{x_1,\dots,x_n,y\};\\
&(2) \,\, \mathrm{Relations}\!: \,\\
&x_i^k=y^b=1\, \forall i=1,\cdots,n, \\
       & x_{i+1}\cdots x_ny\delta^{q-1}x_1\cdots x_{i+r}=x_i\cdots x_ny\delta^{q-1}x_1\cdots x_{i+r-1}, \, \forall 1\leq i \leq n-r,\\
      &  x_{i+1}\cdots x_ny\delta^qx_1\cdots x_{i+r-n}=x_i\cdots x_ny\delta^qx_1\cdots x_{i+r-n-1},\, \forall n-r+1\leq i \leq n.
\end{aligned}
\end{equation}
where $\delta$ denotes $x_1\cdots x_ny$.\\
In terms of the canonical generators of $J(k,bn,m)$, we have $x_i=t^{i-1}st^{1-i}$ for all $i\in [n]$ and $y=t^n$.
\end{corollary}

\begin{proof}
It follows immediately from Theorem \ref{Pres2}, replacing $\delta$ by $x_1\cdots x_ny$ and noting that the second line of the presentation becomes trivial.
\end{proof}

\begin{example}
A presentation of $W_2(3,8,5)=J\begin{pmatrix} 3 & 8 & 5\\ & 4 & 5\end{pmatrix}$ is
\begin{equation*}\left\langle \begin{array}{l|cl}

 x_1,x_2,x_3,x_4\, &\,x_1^3=x_2^3=x_3^3=x_4^3=y^2=1\\
\,\,\,\,\,\,\,\,\,\,\,\,\,\, y \,&\, x_1x_2x_3x_4yx_1=x_2x_3x_4yx_1x_2=x_3x_4yx_1x_2x_3=x_4yx_1x_2x_3x_4\\
 &\, x_4yx_1x_2x_3x_4y=yx_1x_2x_3x_4yx_1
\end{array}
\right \rangle.\end{equation*}

\end{example}

\begin{corollary}\label{Pres2y=1}
Let $k,n,c,m\in \N^*$ with $k,n,cm\geq 2$ and assume $n\wedge m=1$. Write $m=qn+r$ with $q\geq 0$ and $0\leq r \leq n-1$. Then $W^c(k,n,cm)$ admits the presentation 
\begin{equation}\label{PresPres2y=1}
\begin{aligned}
&(1) \,\, \mathrm{Generators}\!:\,  \{x_1,\dots,x_n,z\};\\
&(2) \,\, \mathrm{Relations}\!: \,\\
&x_i^k=z^c=1\, \forall i=1,\cdots,n, \\
  &  x_1\cdots x_nz=zx_1\cdots x_n,\\
       & x_{i+1}\cdots x_nz\delta^{q-1}x_1\cdots x_{i+r}=x_i\cdots x_nz\delta^{q-1}x_1\cdots x_{i+r-1}, \, \forall 1\leq i \leq n-r,\\
      &  x_{i+1}\cdots x_nz\delta^qx_1\cdots x_{i+r-n}=x_i\cdots x_nz\delta^qx_1\cdots x_{i+r-n-1},\, \forall n-r+1\leq i \leq n.
\end{aligned}
\end{equation}
where $\delta$ denotes $x_1\cdots x_n$.\\
In terms of the canonical generators of $J(k,n,cm)$, we have $x_i=t^{i-1}st^{1-i}$ for all $i\in [n]$ and $z=u^m$.
\end{corollary}

\begin{proof}
It follows immediately from Proposition \ref{PresentationNorM1}, replacing $\delta$ by $x_1\cdots x_n$.
\end{proof}

\begin{example}
A presentation of $W^2(3,3,8)=J\begin{pmatrix} 3 & 3 & 8 \\ & 3 & 4\end{pmatrix}$ is
\begin{equation*}\left\langle \begin{array}{l|cl}

                  x_1,x_2,x_3\,  &\,x_1^3=x_2^3=x_3^3=z^2=1\\
\,\,\,\,\,\,\,\,\,\,z\, &\, x_1x_2x_3z=zx_1x_2x_3 \\\,&\,x_1x_2x_3zx_1=x_2x_3zx_1x_2=x_3zx_1x_2x_3
\end{array}
\right \rangle.\end{equation*}
\end{example}

Presentation \eqref{GeneralPresW} also retrieves Presentation \eqref{PresToric}:
\begin{corollary}
Let $k,n,m\in \N_{\geq 2}$. Then $W(k,n,m)$ admits the presentation 
\begin{equation*}\left\langle \begin{array}{c|c}x_1,\dots,x_n& x_i^k=1\, \forall i\in [n],\, \, x_i\cdots x_{i+m-1}=x_j\cdots x_{j+m-1} \, \forall 1\leq i<j\leq n\end{array}\right\rangle, \end{equation*}
where indices are taken modulo $n$. In terms of the canonical generators of $J(k,n,m)$, we have $x_i=t^{i-1}st^{1-i}$ for all $i\in [n]$.
\end{corollary}

%%%%%%%%%%%%%%%%%%%%%%%%%%%%%%%%%%%%%%%%%%%%%%%%%%%%%%%%%%%%%%%%%%%%%%%%%%%%%%%%%%%%%%%%%%%%%%%%%%%%%%%
%%%%%%%%%%%%%%%%%%%%%%%%%%%%%%%%%%%%%%%%%%%%%%%%%%%%%%%%%%%%%%%%%%%%%%%%%%%%%%%%%%%%%%%%%%%%%%%%%%%%%%%
%%%%%%%%%%%%%%%%%%%%%%%%%%%%%%%%%%%%%%%%%%%%%%%%%%%%%%%%%%%%%%%%%%%%%%%%%%%%%%%%%%%%%%%%%%%%%%%%%%%%%%%
%%%%%%%%%%%%%%%%%%%%%%%%%%%%%%%%%%%%%%%%%%%%%%%%%%%%%%%%%%%%%%%%%%%%%%%%%%%%%%%%%%%%%%%%%%%%%%%%%%%%%%%
%%%%%%%%%%%%%%%%%%%%%%%%%%%%%%%%%%%%%%%%%%%%%%%%%%%%%%%%%%%%%%%%%%%%%%%%%%%%%%%%%%%%%%%%%%%%%%%%%%%%%%%
%%%%%%%%%%%%%%%%%%%%%%%%%%%%%%%%%%%%%%%%%%%%%%%%%%%%%%%%%%%%%%%%%%%%%%%%%%%%%%%%%%%%%%%%%%%%%%%%%%%%%%%
%%%%%%%%%%%%%%%%%%%%%%%%%%%%%%%%%%%%%%%%%%%%%%%%%%%%%%%%%%%%%%%%%%%%%%%%%%%%%%%%%%%%%%%%%%%%%%%%%%%%%%%
%%%%%%%%%%%%%%%%%%%%%%%%%%%%%%%%%%%%%%%%%%%%%%%%%%%%%%%%%%%%%%%%%%%%%%%%%%%%%%%%%%%%%%%%%%%%%%%%%%%%%%%
%%%%%%%%%%%%%%%%%%%%%%%%%%%%%%%%%%%%%%%%%%%%%%%%%%%%%%%%%%%%%%%%%%%%%%%%%%%%%%%%%%%%%%%%%%%%%%%%%%%%%%%
%%%%%%%%%%%%%%%%%%%%%%%%%%%%%%%%%%%%%%%%%%%%%%%%%%%%%%%%%%%%%%%%%%%%%%%%%%%%%%%%%%%%%%%%%%%%%%%%%%%%%%%
%%%%%%%%%%%%%%%%%%%%%%%%%%%%%%%%%%%%%%%%%%%%%%%%%%%%%%%%%%%%%%%%%%%%%%%%%%%%%%%%%%%%%%%%%%%%%%%%%%%%%%%
%%%%%%%%%%%%%%%%%%%%%%%%%%%%%%%%%%%%%%%%%%%%%%%%%%%%%%%%%%%%%%%%%%%%%%%%%%%%%%%%%%%%%%%%%%%%%%%%%%%%%%%
%%%%%%%%%%%%%%%%%%%%%%%%%%%%%%%%%%%%%%%%%%%%%%%%%%%%%%%%%%%%%%%%%%%%%%%%%%%%%%%%%%%%%%%%%%%%%%%%%%%%%%%
%%%%%%%%%%%%%%%%%%%%%%%%%%%%%%%%%%%%%%%%%%%%%%%%%%%%%%%%%%%%%%%%%%%%%%%%%%%%%%%%%%%%%%%%%%%%%%%%%%%%%%%
%%%%%%%%%%%%%%%%%%%%%%%%%%%%%%%%%%%%%%%%%%%%%%%%%%%%%%%%%%%%%%%%%%%%%%%%%%%%%%%%%%%%%%%%%%%%%%%%%%%%%%%

\begin{corollary}\label{xiConjugate}
Let $k,b,n,c,m\in \N^*$ such that $k,bn,cm\geq 2$ and assume $n\wedge m=1$. With the notation of Theorem \ref{Pres2}, all $x_i$'s are conjugate.
\end{corollary}

\begin{proof}
Write $m=qn+r$ with $q\geq 0$ and $0\leq r \leq n-1$. For all $i=1,\dots,n-r$, the relation 
$$x_{i+1}\dots x_nyz\delta^{q-1}x_1\dots x_{i+r}=x_i\dots x_nyz\delta^{q-1}x_1\dots x_{i+r-1}$$
conjugates $x_i$ with $x_{i+r}$. \\
Similarly, for all $i=n-r+1,\dots,n$, the relation 
$$ x_{i+1}\dots x_nyz\delta^qx_1\dots x_{i+r-n}=x_i\dots x_nyz\delta^qx_1\dots x_{i+r-n-1}$$
conjugates $x_i$ with $x_{i+r-n}$. Taking indices modulo $n$, we get that for each $i\in [n]$, $x_i$ is conjugate to $x_{i+r}$. Since $n$ and $m$ are coprime $n$ and $r$ are as well so that all $x_i$'s are conjugate. This concludes the proof.
\end{proof}
%%%%%%%%%%%%%%%%%%%%%%%%%%%%%%%%%%%%%%%%%%%%%%%%%%%%%%%%%%%%%%%%%%%%%%%%%%%%%%%%%%%%%%%%%%%%%%%%%%%%%%%
%%%%%%%%%%%%%%%%%%%%%%%%%%%%%%%%%%%%%%%%%%%%%%%%%%%%%%%%%%%%%%%%%%%%%%%%%%%%%%%%%%%%%%%%%%%%%%%%%%%%%%%
\begin{remark}[{\cite[Corollary 2.13]{Gobet Toric}}]
The case $b=c=1$ of Corollary \ref{xiConjugate} is that of toric reflection groups, in which case the result is already known.
\end{remark}
%%%%%%%%%%%%%%%%%%%%%%%%%%%%%%%%%%%%%%%%%%%%%%%%%%%%%%%%%%%%%%%%%%%%%%%%%%%%%%%%%%%%%%%%%%%%%%%%%%%%%%%
%%%%%%%%%%%%%%%%%%%%%%%%%%%%%%%%%%%%%%%%%%%%%%%%%%%%%%%%%%%%%%%%%%%%%%%%%%%%%%%%%%%%%%%%%%%%%%%%%%%%%%%
%%%%%%%%%%%%%%%%%%%%%%%%%%%%%%%%%%%%%%%%%%%%%%%%%%%%%%%%%%%%%%%%%%%%%%%%%%%%%%%%%%%%%%%%%%%%%%%%%%%%%%%
%%%%%%%%%%%%%%%%%%%%%%%%%%%%%%%%%%%%%%%%%%%%%%%%%%%%%%%%%%%%%%%%%%%%%%%%%%%%%%%%%%%%%%%%%%%%%%%%%%%%%%%
%%%%%%%%%%%%%%%%%%%%%%%%%%%%%%%%%%%%%%%%%%%%%%%%%%%%%%%%%%%%%%%%%%%%%%%%%%%%%%%%%%%%%%%%%%%%%%%%%%%%%%%
%%%%%%%%%%%%%%%%%%%%%%%%%%%%%%%%%%%%%%%%%%%%%%%%%%%%%%%%%%%%%%%%%%%%%%%%%%%%%%%%%%%%%%%%%%%%%%%%%%%%%%%
%%%%%%%%%%%%%%%%%%%%%%%%%%%%%%%%%%%%%%%%%%%%%%%%%%%%%%%%%%%%%%%%%%%%%%%%%%%%%%%%%%%%%%%%%%%%%%%%%%%%%%%
%%%%%%%%%%%%%%%%%%%%%%%%%%%%%%%%%%%%%%%%%%%%%%%%%%%%%%%%%%%%%%%%%%%%%%%%%%%%%%%%%%%%%%%%%%%%%%%%%%%%%%%
%%%%%%%%%%%%%%%%%%%%%%%%%%%%%%%%%%%%%%%%%%%%%%%%%%%%%%%%%%%%%%%%%%%%%%%%%%%%%%%%%%%%%%%%%%%%%%%%%%%%%%%
%%%%%%%%%%%%%%%%%%%%%%%%%%%%%%%%%%%%%%%%%%%%%%%%%%%%%%%%%%%%%%%%%%%%%%%%%%%%%%%%%%%%%%%%%%%%%%%%%%%%%%%

\begin{corollary}\label{QuotientCircularToric}
Let $b,c\in \N^*$, $k,n,m\geq 2$ and assume $n\wedge m=1$. We have the following commutative square of quotients: 
\begin{equation}\label{CommSquareFINDANAME}\begin{tikzcd}
	{W_b^c(k,bn,cm)} & {W^c(k,n,cm)} \\
	{W_b(k,bn,m)} & {W(k,n,m)}
	\arrow["{y=1}", two heads, from=1-1, to=1-2]
	\arrow["{z=1}", two heads, from=1-2, to=2-2]
	\arrow["{z=1}"', two heads, from=1-1, to=2-1]
	\arrow["{y=1}"', two heads, from=2-1, to=2-2]
\end{tikzcd}\end{equation}

\end{corollary}

\begin{proof}
It follows from Theorem \ref{Pres2} that setting $b$ (respectively $c$) to be equal to $1$ in Presentation \eqref{GeneralPresW} yields a presentation of $W^c(k,n,cm)$ (respectively of $W_b(k,bn,m)$). Finally, setting $b$ to be equal to 1 in Presentation \eqref{Pres2z=1} yields Presentation \eqref{PresToric} since $m=qn+r$. This concludes the proof.
\end{proof}

\begin{remark}
If $n=1$ and $bn\geq 2$ (respectively $m=1$ and $cm\geq 2$), we have an epimorphism $W_b^c(k,bn,cm)\xto{z=1}W_b(k,bn,m)$ (resp. $W_b^c(k,bn,cm)\xto{y=1}W^c(k,n,cm)$) but the Square \eqref{CommSquareFINDANAME} contains non-defined groups.
\end{remark}

We split the proof of Theorem \ref{Pres2} in four lemmas, using the following strategy: We use Reidemeister-Schreier Algorithm to get a presentation of $W_b^c(k,bn,cm)$ and we simplify it slightly. Then, we use this presentation to derive some useful equalities in $W_b^c(k,bn,cm)$, which in turn will allow us to simplify our given presentation even more and finally arrive to the presentation given in Theorem \ref{Pres2}.

\begin{lemma}\label{ProofPres21}
A presentation of $W_b^c(k,bn,cm)$ is the following:
\begin{equation}\label{Int1Pres2}
\begin{aligned}
&(1) \,\, \mathrm{Generators}\!:\,  \{s_0,\dots,s_{n-1},t_0,\dots,t_{n-1},z_0\dots,z_{n-1}\};\\
&(2) \,\, \mathrm{Relations}\!: \,\\
  &  t_i=1 \,\, \forall i=0,\dots,n-2 \,;\, s_i^k=t_{n-1}^b=z_0^c=1 \, \, \forall i=0,\dots,n-1, \\
  &  s_0\cdots s_{n-1}t_{n-1}z_0=z_0s_0\cdots s_{n-1}t_{n-1},\\
  &z_i=(\prod_{j=0}^{i-1}s_j)^{-1}z_0(\prod_{j=0}^{i-1}s_j)\, \, \forall i=1,\dots,n-1,\\
  & t_iz_{i+1}(\prod_{j=1}^{m-1}s_{i+j}t_{i+j})s_{i+m}=z_is_it_i (\prod_{j=1}^{m-1}s_{i+j}t_{i+j})  \,  \,\, \forall i=0,\dots,n-1.
  \end{aligned}
  \end{equation}
where indices are taken modulo $n$.\\
In terms of the generators $\{s,t,z\}$ of $W_{bn}^c(k,bn,cm)$, we have $s_i=t^{i}st^{-i}$ and $z_i=t^{i}zt^{-i}$ for all $i=0,\dots,n-1$. Moreover, we have $t_i=1$ for all $i=0,\dots,n-2$ and $t_{n-1}=t^n$.
\end{lemma}

%%%%%%%%%%%%%%%%%%%%%%%%%%%%%%%%%%%%%%%%%%%%%%%%%%%%%%%%%%%%%%%%%%%%%%%%%%%%%%%%%%%%%%%%%%%%%%%%%%%%%%%%%%%%%%%%%%%%%
%%%%%%%%%%%%%%%%%%%%%%%%%%%%%%%%%%%%%%%%%%%%%%%%%%%%%%%%%%%%
%%%%%%%%%%%%%%%%%%%%%%%%%%%%%%%%%%%%%%%%%%%%%%%%%%%%%%%%%%%%
%%%%%%%%%%%%%%%%%%%%%%%%%%%%%%%%%%%%%%%%%%%%%%%%%%%%%%%%%%%%%%%%%%%%%%%%%%%%%%%%%%%%%%%%%%%%%%%%%%%%%%%%%%%%%%%%%%%%%%%%%%%%%%%%%%%%%%%%%%%%%%%%%%%%%%%%%%%%%%%%%%%%%%%%%%%%%%%%%%%%
%%%%%%%%%%%%%%%%%%%%%%%%%%%%%%%%%%%%%%%%%%%%%%%%%%%%%%%%%%%%
%%%%%%%%%%%%%%%%%%%%%%%%%%%%%%%%%%%%%%%%%%%%%%%%%%%%%%%%%%%%%%%%%%%%%%%%%%%%%%%%%%%%%%%%%%%%%%%%%%%%%%%%%%%%%%%%%%%%%%%%
%%%%%%%%%%%%%%%%%%%%%%%%%%%%%%%%%%%%%%%%%%%%%%%%%%%%%%%%%%%%
%%%%%%%%%%%%%%%%%%%%%%%%%%%%%%%%%%%%%%%%%%%%%%%%%%%%%%%%%%%%
%%%%%%%%%%%%%%%%%%%%%%%%%%%%%%%%%%%%%%%%%%%%%%%%%%%%%%%%%%%%%%%%%%%%%%%%%%%%%%%%%%%%%%%%%%%%%%%%%%%%%%%%%%%%%%%%%%%%%%%%%%%%%%%%%%%%%%%%%%%%%%%%%%%%%%%%%%%%%%%%%%%%%%%%%%%%%%%%%%%%
%%%%%%%%%%%%%%%%%%%%%%%%%%%%%%%%%%%%%%%%%%%%%%%%%%%%%%%%%%%%
%%%%%%%%%%%%%%%%%%%%%%%%%%%%%%%%%%%%%%%%%%%%%%%%%%%%%%%%%%%%%%%%%%%%%%

\begin{proof}
Combining Corollary \ref{NormalClosureSmall} with $b'=n$ and $c'=1$ and Corollary \ref{xiConjugate} and since $x_1=s$, we have $$W_b^c(k,bn,cm)=\llangle x_1,\cdots,x_n,t^n,z \rrangle_{W_{bn}^c(k,bn,cm)}=\llangle s,t^n,z\rrangle_{W_{bn}^c(k,bn,cm)},$$ hence Proposition \ref{PresentationNorM1} implies that $W_{bn}^c(k,bn,cm)/W_b^c(k,bn,cm)\cong\Z/n$. A Schreier Transversal for $W_b^c(k,bn,cm)$ is thus $\{1,t,\dots,t^{n-1}\}$. Using Reidemeister-Schreier algorithm as in Proposition \ref{Algo} with Equation \eqref{tau2} and considering Presentation \eqref{NorM1} for $W_{bn}^c(k,bn,cm)$, the group $W_b^c(k,bn,cm)$ admits the following presentation:
\begin{equation}\label{DualWLater}
\begin{aligned}
&(1) \,\, \mathrm{Generators}\!:\,  \{s_0,\dots,s_{n-1},t_0,\dots,t_{n-1},z_0\dots,z_{n-1}\};\\
&(2) \,\, \mathrm{Relations}\!: \,\\
  &  t_i=1 \,\, \forall i=0,\dots,n-2 ; \, \, s_i^k=(t_it_{i+1}\cdots t_{i+n-1})^b=z_i^c=1 \, \, \forall i=0,\dots,n-1, \\
  &  s_it_iz_{i+1}=z_is_it_i       \,\, \forall i=0,\dots,n-1,\\
  & t_iz_{i+1}\prod_{j=1}^{m-1}(s_{i+j}t_{i+j})s_{i+m}=z_is_it_i \prod_{j=1}^{m-1}(s_{i+j}t_{i+j})  \,  \,\, \forall i=0,\dots,n-1.
  \end{aligned}
\end{equation}
where indices are taken modulo $n$. It is then enough to show that the first two lines of relations of Presentations \eqref{Int1Pres2} and \eqref{DualWLater} are equivalent to conclude. In this proof we write L1 to refer to the set of relations $$t_i=1\, \, \forall i=0,\dots,n-2.$$\\
\noindent
Combining L1 with $(t_it_{i+1}\cdots t_{i+n-1})^b=1\,\, \forall i=0,\dots,n-1$, we get 
\begin{equation}\label{Torsion t}
t_{n-1}^b=1.
\end{equation}
Moreover, one sees that the set of relations
\begin{equation}\label{PresLemma11}
(s_it_iz_{i+1}=z_is_it_i \, \, \forall i=0,\dots,n-2 )\, + \, \text{L1}
\end{equation}
\noindent
is equivalent to the set of relations
\begin{equation}\label{PresLemma12}
(z_{i+1}=s_i^{-1}z_is_i \, \, \forall i=0,\dots,n-2) \, + \, \text{L1}.
\end{equation}
Iterating the conjugation by $s_i$, the set of relations given by Equation \eqref{PresLemma12} is equivalent to\begin{equation}\label{PresLemma13}
(z_i=(\prod_{j=0}^{i-1}s_j)^{-1}z_0(\prod_{j=0}^{i-1}s_j)\, \, \forall i=1,\dots,n-1) \, + \, \text{L1}.
\end{equation}
We thus get that the set of relations
\begin{equation}\label{PresLemma13Add}(s_{n-1}t_{n-1}z_0=z_{n-1}s_{n-1}t_{n-1})+\eqref{PresLemma13}
\end{equation}
is equivalent to the set of relations 
\begin{equation}\label{PresLemma14}
(s_{n-1}t_{n-1}z_0=(s_0\cdots s_{n-2})^{-1}z_0s_0\cdots s_{n-1}t_{n-1})+\eqref{PresLemma13},
\end{equation}
which in turn is equivalent to 
\begin{equation}\label{PresLemma15}
(s_0\cdots s_{n-1}t_{n-1}z_0=z_0s_0\cdots s_{n-1}t_{n-1})+\eqref{PresLemma13}.
\end{equation}
Finally, by Proposition \ref{Algo}, we have $s_i=t^{i}st^{-i}$ and $z_i=t^{i}zt^{-i}$ for all $i=0,\dots,n-1$. Moreover, we have $t_i=1$ for all $i=0,\dots,n-2$ and $t_{n-1}=t^n$. This concludes the proof.
\end{proof}

%%%%%%%%%%%%%%%%%%%%%%%%%%%%%%%%%%%%%%%%%%%%%%%%%%%%%%%%%%%%%%%%%%%%
%%%%%%%%%%%%%%%%%%%%%%%%%%%%%%%%%%%%%%%%%%%%%%%%%%%%%%%%%%%%%%%%%%%%%%%%%%%%%%%%%%%%%%%%%%%%%%%%%%%%%%%
%%%%%%%%%%%%%%%%%%%%%%%%%%%%%%%%%%%%%%%%%%%%%%%%%%%%%%%%%%%%%%%%%%%%%%%%%%%%%%%%%%%%%%%%%%%%%%%%%%%%%%%
%%%%%%%%%%%%%%%%%%%%%%%%%%%%%%%%%%%%%%%%%%%%%%%%%%%%%%%%%%%%%%%%%%%%%%%%%%%%%%%%%%%%%%%%%%%%%%%%%%%%%%%

\begin{lemma}\label{ProofPres22}
(i) For $0\leq i \leq n-r-1$, in $W_b^c(k,bn,cm)$ we have
\begin{equation}\label{PropLemma21}
\prod_{j=1}^{m-1}s_{i+j}t_{i+j}=(s_{i+1}\cdots s_{n-1}t_{n-1}s_0\cdots s_i)^q(\prod_{j=1}^{r-1}s_{i+j})
.\end{equation}
(ii) For $n-r\leq i \leq n-1$, we have
\begin{equation}\label{PropLemma22}
\prod_{j=1}^{m-1}s_{i+j}t_{i+j}=s_{i+1}\cdots s_{n-1}t_{n-1}(s_0\cdots s_{n-1}t_{n-1})^qs_0\cdots s_{i+r-n-1}
.\end{equation}
\end{lemma}

%%%%%%%%%%%%%%%%%%%%%%%%%%%%%%%%%%%%%%%%%%%%%%%%%%%%%%%%%%%%%%%%%%%%%%%%%%%%%%%%%%%%%%%%%%%%%%%%%%%%%%%%%%%%%%%%%%%%%
%%%%%%%%%%%%%%%%%%%%%%%%%%%%%%%%%%%%%%%%%%%%%%%%%%%%%%%%%%%%
%%%%%%%%%%%%%%%%%%%%%%%%%%%%%%%%%%%%%%%%%%%%%%%%%%%%%%%%%%%%
%%%%%%%%%%%%%%%%%%%%%%%%%%%%%%%%%%%%%%%%%%%%%%%%%%%%%%%%%%%%%%%%%%%%%%%%%%%%%%%%%%%%%%%%%%%%%%%%%%%%%%%%%%%%%%%%%%%%%%%%%%%%%%%%%%%%%%%%%%%%%%%%%%%%%%%%%%%%%%%%%%%%%%%%%%%%%%%%%%%%
%%%%%%%%%%%%%%%%%%%%%%%%%%%%%%%%%%%%%%%%%%%%%%%%%%%%%%%%%%%%
%%%%%%%%%%%%%%%%%%%%%%%%%%%%%%%%%%%%%%%%%%%%%%%%%%%%%%%%%%%%%%%%%%%%%%%%%%%%%%%%%%%%%%%%%%%%%%%%%%%%%%%%%%%%%%%%%%%%%%%%
%%%%%%%%%%%%%%%%%%%%%%%%%%%%%%%%%%%%%%%%%%%%%%%%%%%%%%%%%%%%
%%%%%%%%%%%%%%%%%%%%%%%%%%%%%%%%%%%%%%%%%%%%%%%%%%%%%%%%%%%%
%%%%%%%%%%%%%%%%%%%%%%%%%%%%%%%%%%%%%%%%%%%%%%%%%%%%%%%%%%%%%%%%%%%%%%%%%%%%%%%%%%%%%%%%%%%%%%%%%%%%%%%%%%%%%%%%%%%%%%%%%%%%%%%%%%%%%%%%%%%%%%%%%%%%%%%%%%%%%%%%%%%%%%%%%%%%%%%%%%%%
%%%%%%%%%%%%%%%%%%%%%%%%%%%%%%%%%%%%%%%%%%%%%%%%%%%%%%%%%%%%
%%%%%%%%%%%%%%%%%%%%%%%%%%%%%%%%%%%%%%%%%%%%%%%%%%%%%%%%%%%%%%%%%%%%%%
\begin{proof}
\fbox{(i) $0\leq i \leq n-r-1$:}
For all $i\leq n-r-1$, we have $i+r-1\leq n-2$ so that $t_{i+1},\cdots,t_{i+r-1}$ are all equal to $1$ by L1. We get that $\prod_{j=i+1}^{i+r-1}s_jt_j=\prod_{j=1}^{r-1}s_{i+j}$. This allows the following computation:
\begin{equation}\label{PresLemma21}
\begin{aligned}
\prod_{j=1}^{m-1}s_{i+j}t_{i+j}&=((\prod_{j=i+1}^{n-1}s_jt_j)(\prod_{j=0}^{i}s_jt_j))^{q}(\prod_{j=i+1}^{i+r-1}s_jt_j)\\
&=(s_{i+1}\cdots s_{n-1}t_{n-1}s_0\cdots s_i)^q(\prod_{j=1}^{r-1}s_{i+j}).
\end{aligned}
\end{equation}
\fbox{(ii) $n-r\leq i\leq n-1$:}
In this case, we have 
\begin{equation}\label{PresLemma22}
\prod_{j=1}^{r-1}(s_{i+j}t_{i+j})=(\prod_{j=i+1}^{n-1}(s_jt_j))\prod_{j=0}^{i+r-n-1}(s_jt_j)=s_{i+1}\cdots s_{n-1}t_{n-1}s_0\cdots s_{i+r-n-1}.
\end{equation}
Using Equation \eqref{PresLemma22}, we do the following computation:
\begin{equation}\label{PresLemma23}
\begin{aligned}
\prod_{j=1}^{m-1}s_{i+j}t_{i+j}&=((\prod_{j=i+1}^{n-1}s_jt_j)(\prod_{j=0}^{i}s_jt_j))^{q}(\prod_{j=1}^{r-1}s_{i+j}t_{i+j})\\
&=(s_{i+1}\cdots s_{n-1}t_{n-1}s_0\cdots s_i)^qs_{i+1}\cdots s_{n-1}t_{n-1}s_0\cdots s_{i+r-n-1}\\
&=s_{i+1}\cdots s_{n-1}t_{n-1}(s_0\cdots s_{n-1}t_{n-1})^{q}s_0\cdots s_{i+r-n-1},
\end{aligned}
\end{equation}
which concludes the proof.
\end{proof}

Combining Lemmas \ref{ProofPres21} and \ref{ProofPres22}, we are now ready to simplify Presentation \eqref{Int1Pres2}.
For this, we introduce the following notation: L4 will denote the set of relations 
$$t_iz_{i+1}(\prod_{j=1}^{m-1}s_{i+j}t_{i+j})s_{i+m}=z_is_it_i (\prod_{j=1}^{m-1}s_{i+j}t_{i+j})  \,  \,\, \forall i=0,\dots,n-r-1$$
whereas L4' will denote the set of relations
$$t_iz_{i+1}(\prod_{j=1}^{m-1}s_{i+j}t_{i+j})s_{i+m}=z_is_it_i (\prod_{j=1}^{m-1}s_{i+j}t_{i+j})  \,  \,\, \forall i=n-r,\dots,n-1.$$
Moreover, for convenience we will denote $$\delta:=s_0\cdots s_{n-1}t_{n-1},$$ and observe that Relation \eqref{PresLemma15} reads $\delta z_0=z_0\delta$. One sees that in Presentation \eqref{Int1Pres2} each generator $t_0,\dots,t_{n-2}$ represents the neutral element and that the generators $z_1,\dots,z_{n-1}$ are redundant. The goal of the next two lemmas is to remove these generators from our presentation.

%%%%%%%%%%%%%%%%%%%%%%%%%%%%%%%%%%%%%%%%%%%%%%%%%%%%%%%%%%%%%%%%%%%%
%%%%%%%%%%%%%%%%%%%%%%%%%%%%%%%%%%%%%%%%%%%%%%%%%%%%%%%%%%%%%%%%%%%%%%%%%%%%%%%%%%%%%%%%%%%%%%%%%%%%%%%
%%%%%%%%%%%%%%%%%%%%%%%%%%%%%%%%%%%%%%%%%%%%%%%%%%%%%%%%%%%%%%%%%%%%%%%%%%%%%%%%%%%%%%%%%%%%%%%%%%%%%%%
%%%%%%%%%%%%%%%%%%%%%%%%%%%%%%%%%%%%%%%%%%%%%%%%%%%%%%%%%%%%%%%%%%%%%%%%%%%%%%%%%%%%%%%%
\begin{lemma}\label{ProofPres23}
The set of relations L4+\eqref{PresLemma13}+\eqref{PresLemma15}+\eqref{PropLemma21} is equivalent to the set of relations
($s_{i+1}\cdots s_{n-1}t_{n-1}z_0\delta^{q-1}s_0\cdots s_{i+r}=s_i\cdots s_{n-1}t_{n-1}z_0\delta^{q-1}s_0\cdots s_{i+r-1}$, for all $i\in \llbracket 0,n-r-1\rrbracket+$ \eqref{PresLemma13}+\eqref{PresLemma15}+\eqref{PropLemma21}).
\end{lemma}

%%%%%%%%%%%%%%%%%%%%%%%%%%%%%%%%%%%%%%%%%%%%%%%%%%%%%%%%%%%%%%%%%%%%%%%%%%%%%%%%%%%%%%%%%%%%%%%%%%%%%%%%%%%%%%%%%%%%%
%%%%%%%%%%%%%%%%%%%%%%%%%%%%%%%%%%%%%%%%%%%%%%%%%%%%%%%%%%%%
%%%%%%%%%%%%%%%%%%%%%%%%%%%%%%%%%%%%%%%%%%%%%%%%%%%%%%%%%%%%
%%%%%%%%%%%%%%%%%%%%%%%%%%%%%%%%%%%%%%%%%%%%%%%%%%%%%%%%%%%%%%%%%%%%%%%%%%%%%%%%%%%%%%%%%%%%%%%%%%%%%%%%%%%%%%%%%%%%%%%%%%%%%%%%%%%%%%%%%%%%%%%%%%%%%%%%%%%%%%%%%%%%%%%%%%%%%%%%%%%%
%%%%%%%%%%%%%%%%%%%%%%%%%%%%%%%%%%%%%%%%%%%%%%%%%%%%%%%%%%%%
%%%%%%%%%%%%%%%%%%%%%%%%%%%%%%%%%%%%%%%%%%%%%%%%%%%%%%%%%%%%%%%%%%%%%%%%%%%%%%%%%%%%%%%%%%%%%%%%%%%%%%%%%%%%%%%%%%%%%%%%
%%%%%%%%%%%%%%%%%%%%%%%%%%%%%%%%%%%%%%%%%%%%%%%%%%%%%%%%%%%%
%%%%%%%%%%%%%%%%%%%%%%%%%%%%%%%%%%%%%%%%%%%%%%%%%%%%%%%%%%%%
%%%%%%%%%%%%%%%%%%%%%%%%%%%%%%%%%%%%%%%%%%%%%%%%%%%%%%%%%%%%%%%%%%%%%%%%%%%%%%%%%%%%%%%%%%%%%%%%%%%%%%%%%%%%%%%%%%%%%%%%%%%%%%%%%%%%%%%%%%%%%%%%%%%%%%%%%%%%%%%%%%%%%%%%%%%%%%%%%%%%
%%%%%%%%%%%%%%%%%%%%%%%%%%%%%%%%%%%%%%%%%%%%%%%%%%%%%%%%%%%%
%%%%%%%%%%%%%%%%%%%%%%%%%%%%%%%%%%%%%%%%%%%%%%%%%%%%%%%%%%%%%%%%%%%%%%
\begin{proof}
Let $i\in \llbracket 0,n-r-1 \rrbracket$. To prove this Lemma, we show that the left-hand side (respectively right-hand side) of relations L4 represents the same element as\\
$s_{i+1}\cdots s_{n-1}t_{n-1}z_0\delta^{q-1}s_0\cdots s_{i+r}$ (respectively as $s_i\dots s_{n-1}t_{n-1}z_0\delta^{q-1}s_0\cdots s_{i+r-1}$) under the assumption that the set of relations $\eqref{PresLemma13}+\eqref{PresLemma15}+\eqref{PropLemma21}$ holds.\\
\fbox{Left-hand side of L4:} First, observe that $s_{i+m}=s_{i+r}$ since indices are taken modulo $n$. We thus have
\begin{equation*}
\begin{aligned}
 t_iz_{i+1}(\prod_{j=1}^{m-1}s_{i+j}t_{i+j})s_{i+m}&\underset{\eqref{PresLemma13}}=(\prod_{j=0}^{i}s_j)^{-1}z_0(\prod_{j=0}^{i}s_j)(\prod_{j=1}^{m-1}s_{i+j}t_{i+j})s_{i+m} \\
 \underset{\eqref{PropLemma21}}=&(\prod_{j=0}^{i}s_j)^{-1}z_0(\prod_{j=0}^{i}s_j)(s_{i+1}\cdots s_{n-1}t_{n-1}s_0\cdots s_i)^q(\prod_{j=1}^{r-1}s_{i+j})s_{i+r}\\
 \underset{\eqref{PresLemma15}}=&s_{i+1}\cdots s_{n-1}t_{n-1}z_0\delta^{q-1}s_0\cdots s_{i+r}.
 \end{aligned}
\end{equation*}
\fbox{Right-hand side of L4:} We have
\begin{equation*}
\begin{aligned}
z_is_it_i (\prod_{j=1}^{m-1}s_{i+j}t_{i+j})&\underset{\eqref{PresLemma13}}=(\prod_{j=0}^{i-1}s_j)^{-1}z_0(\prod_{j=0}^{i-1}s_j)s_i(\prod_{j=1}^{m-1}s_{i+j}t_{i+j}) \\
\underset{\eqref{PropLemma21}}=&(\prod_{j=0}^{i-1}s_j)^{-1}z_0(\prod_{j=0}^{i-1}s_j)s_i(s_{i+1}\cdots s_{n-1}t_{n-1}s_0\cdots s_i)^q(\prod_{j=1}^{r-1}s_{i+j})\\
\underset{\eqref{PresLemma15}}=&s_i\cdots s_{n-1}t_{n-1}z_0\delta^{q-1}s_0\cdots s_{i+r-1}.
\end{aligned}
\end{equation*}
This concludes the proof.
\end{proof}

%%%%%%%%%%%%%%%%%%%%%%%%%%%%%%%%%%%%%%%%%%%%%%%%%%%%%%%%%%%%%%%%%%%%
%%%%%%%%%%%%%%%%%%%%%%%%%%%%%%%%%%%%%%%%%%%%%%%%%%%%%%%%%%%%%%%%%%%%%%%%%%%%%%%%%%%%%%%%%%%%%%%%%%%%%%%
%%%%%%%%%%%%%%%%%%%%%%%%%%%%%%%%%%%%%%%%%%%%%%%%%%%%%%%%%%%%%%%%%%%%%%%%%%%%%%%%%%%%%%%%%%%%%%%%%%%%%%%
%%%%%%%%%%%%%%%%%%%%%%%%%%%%%%%%%%%%%%%%%%%%%%%%%%%%%%%%%%%%%%%%%%%%%%%%%%%%%%%%%%%%%%%%
\begin{lemma}\label{ProofPres24}
The set of relations L4'+\eqref{PresLemma13}+\eqref{PresLemma15}+\eqref{PropLemma22} is equivalent to the set of relations ($s_{i+1}\cdots s_{n-1}t_{n-1}z_0\delta^{q}s_0\cdots s_{i+r-n}=s_i\cdots s_{n-1}t_{n-1}z_0\delta^{q}s_0\cdots s_{i+r-n-1}$, for all $i\in \llbracket n-r,n-1\rrbracket$ +\eqref{PresLemma13}+\eqref{PresLemma15}+\eqref{PropLemma22}).
\end{lemma}

%%%%%%%%%%%%%%%%%%%%%%%%%%%%%%%%%%%%%%%%%%%%%%%%%%%%%%%%%%%%%%%%%%%%%%%%%%%%%%%%%%%%%%%%%%%%%%%%%%%%%%%%%%%%%%%%%%%%%
%%%%%%%%%%%%%%%%%%%%%%%%%%%%%%%%%%%%%%%%%%%%%%%%%%%%%%%%%%%%
%%%%%%%%%%%%%%%%%%%%%%%%%%%%%%%%%%%%%%%%%%%%%%%%%%%%%%%%%%%%
%%%%%%%%%%%%%%%%%%%%%%%%%%%%%%%%%%%%%%%%%%%%%%%%%%%%%%%%%%%%%%%%%%%%%%%%%%%%%%%%%%%%%%%%%%%%%%%%%%%%%%%%%%%%%%%%%%%%%%%%%%%%%%%%%%%%%%%%%%%%%%%%%%%%%%%%%%%%%%%%%%%%%%%%%%%%%%%%%%%%
%%%%%%%%%%%%%%%%%%%%%%%%%%%%%%%%%%%%%%%%%%%%%%%%%%%%%%%%%%%%
%%%%%%%%%%%%%%%%%%%%%%%%%%%%%%%%%%%%%%%%%%%%%%%%%%%%%%%%%%%%%%%%%%%%%%%%%%%%%%%%%%%%%%%%%%%%%%%%%%%%%%%%%%%%%%%%%%%%%%%%
%%%%%%%%%%%%%%%%%%%%%%%%%%%%%%%%%%%%%%%%%%%%%%%%%%%%%%%%%%%%
%%%%%%%%%%%%%%%%%%%%%%%%%%%%%%%%%%%%%%%%%%%%%%%%%%%%%%%%%%%%
%%%%%%%%%%%%%%%%%%%%%%%%%%%%%%%%%%%%%%%%%%%%%%%%%%%%%%%%%%%%%%%%%%%%%%%%%%%%%%%%%%%%%%%%%%%%%%%%%%%%%%%%%%%%%%%%%%%%%%%%%%%%%%%%%%%%%%%%%%%%%%%%%%%%%%%%%%%%%%%%%%%%%%%%%%%%%%%%%%%%
%%%%%%%%%%%%%%%%%%%%%%%%%%%%%%%%%%%%%%%%%%%%%%%%%%%%%%%%%%%%
%%%%%%%%%%%%%%%%%%%%%%%%%%%%%%%%%%%%%%%%%%%%%%%%%%%%%%%%%%%%%%%%%%%%%%

\begin{proof}
Let $i\in \llbracket n-r,n-1\rrbracket$. To prove this Lemma, we show that the left-hand side (respectively right-hand side) of relations L4' represents the same element as $s_{i+1}\cdots s_{n-1}t_{n-1}z_0\delta^{q}s_0\cdots s_{i+r-n}$ (respectively as $s_i\cdots s_{n-1}t_{n-1}z_0\delta^{q}s_0\cdots s_{i+r-n-1}$) under the assumption that the relations \eqref{PresLemma13}+\eqref{PresLemma15}+\eqref{PropLemma22} hold.
We separate the proof into two cases, depending on whether $n-r\leq i \leq n-2$ or $i=n-1$.\\
\fbox{Case $n-r\leq i \leq n-2$:}\\
\underline{Left-hand side of L4':}
Since indices are taken modulo $n$, for all $i\in \llbracket n-r,n-1\rrbracket$ we have $s_{i+m}=s_{i+r-n}$ . We thus have
\begin{equation*}
\begin{aligned}
 t_iz_{i+1}(\prod_{j=1}^{m-1}s_{i+j}t_{i+j})s_{i+m}&\underset{\eqref{PresLemma13}}=(\prod_{j=0}^{i}s_j)^{-1}z_0(\prod_{j=0}^{i}s_j)(\prod_{j=1}^{m-1}s_{i+j}t_{i+j})s_{i+m} \\
 \underset{\eqref{PropLemma22}}=(\prod_{j=0}^{i}s_j)^{-1}z_0(\prod_{j=0}^{i}&s_j)s_{i+1}\cdots s_{n-1}t_{n-1}(s_0\cdots s_{n-1}t_{n-1})^qs_0\cdots s_{i+r-n-1}s_{i+r-n} \\
 \underset{\eqref{PresLemma15}}=s_{i+1}\cdots s_{n-1}t_{n-1}&z_0\delta^{q}s_0\cdots s_{i+r-n}.
 \end{aligned}
\end{equation*}
\underline{Right-hand side of L4':} We have
\begin{equation*}
\begin{aligned}
z_is_it_i (\prod_{j=1}^{m-1}s_{i+j}t_{i+j})&\underset{\eqref{PresLemma13}}=(\prod_{j=0}^{i-1}s_j)^{-1}z_0(\prod_{j=0}^{i-1}s_j)s_i(\prod_{j=1}^{m-1}s_{i+j}t_{i+j}) \\
\underset{\eqref{PropLemma22}}=(\prod_{j=0}^{i-1}&s_j)^{-1}z_0(\prod_{j=0}^{i-1}s_j)s_is_{i+1}\cdots s_{n-1}t_{n-1}(s_0\cdots s_{n-1}t_{n-1})^qs_0\cdots s_{i+r-n-1} \\
\underset{\eqref{PresLemma15}}=s_i\cdots& s_{n-1}t_{n-1}z_0\delta^{q}s_0\cdots s_{i+r-n-1}.
\end{aligned}
\end{equation*}
\fbox{Case $i=n-1$:}\\
\underline{Left-hand side of L4':} Observe that for all $i=0,\dots,n-1$, we have $s_{n-1+i}=s_{i-1}$ since indices are taken modulo $n$. We thus get
\begin{eqnarray*}
 t_{n-1}z_{(n-1)+1}(\prod_{j=1}^{m-1}s_{(n-1)+j}t_{(n-1)+j})s_{(n-1)+m}&=&t_{n-1}z_0(\prod_{j=1}^{m-1}s_{j-1}t_{j-1})s_{m-1}\\
 &=&t_{n-1}z_0\delta^q(s_0\cdots s_{r-1}).
\end{eqnarray*}
\underline{Right-hand side of L4':} We have
\begin{eqnarray*}
z_{n-1}s_{n-1}t_{n-1}(\prod_{j=1}^{m-1}s_{j-1}t_{j-1})&\underset{\eqref{PresLemma13}}=&(\prod_{j=0}^{n-2}s_j)^{-1}z_0(\prod_{j=0}^{n-2}s_j)s_{n-1}t_{n-1} (\prod_{j=1}^{m-1}s_{j-1}t_{j-1})\\
&=&(\prod_{j=0}^{n-2}s_j)^{-1}z_0(\prod_{j=0}^{n-2}s_j)s_{n-1}t_{n-1}\delta^{q}s_0\cdots s_{r-2}\\
&\underset{\eqref{PresLemma15}}=&s_{n-1}t_{n-1}z_0\delta^qs_0\cdots s_{r-2}.
\end{eqnarray*}
This concludes the proof.
\end{proof}

%%%%%%%%%%%%%%%%%%%%%%%%%%%%%%%%%%%%%%%%%%%%%%%%%%%%%%%%%%%%%%%%%%%%
%%%%%%%%%%%%%%%%%%%%%%%%%%%%%%%%%%%%%%%%%%%%%%%%%%%%%%%%%%%%%%%%%%%%%%%%%%%%%%%%%%%%%%%%%%%%%%%%%%%%%%%
%%%%%%%%%%%%%%%%%%%%%%%%%%%%%%%%%%%%%%%%%%%%%%%%%%%%%%%%%%%%%%%%%%%%%%%%%%%%%%%%%%%%%%%%%%%%%%%%%%%%%%%
%%%%%%%%%%%%%%%%%%%%%%%%%%%%%%%%%%%%%%%%%%%%%%%%%%%%%%%%%%%%%%%%%%%%%%%%%%%%%%%%%%%%%%%%

\begin{proof}[Proof of Theorem \ref{Pres2}]
Combining Lemmas \ref{ProofPres21},  \ref{ProofPres23}, and  \ref{ProofPres24}, we obtain the following presentation for $W_b^c(k,bn,cm)$ :
\begin{equation*}
\begin{aligned}
&(1) \,\, \mathrm{Generators}\!:\,  \{s_0,\dots,s_{n-1},t_{n-1},z_0\};\\
&(2) \,\, \mathrm{Relations}\!: \,\\
&s_i^k=t_{n-1}^b=z_0^c=1 \, \, \forall i=0,\dots,n-1 , \\
  &  z_0s_0\cdots s_{n-1}t_{n-1}=s_0\cdots s_{n-1}t_{n-1}z_0,\\
       & s_{i+1}\cdots s_{n-1}t_{n-1}z_0\delta^{q-1}s_0\cdots s_{i+r}=s_i\cdots s_{n-1}t_{n-1}z_0\delta^{q-1}s_0\cdots s_{i+r-1}, \, (*),\\
      &  s_{i+1}\cdots s_{n-1}t_{n-1}z_0\delta^{q}s_0\cdots s_{i+r-n}=s_i\cdots s_{n-1}t_{n-1}z_0\delta^{q}s_0\cdots s_{i+r-n-1}, \, (**).\\
      & (*) \, \forall\, 0\leq i\leq n-r-1, \, \, (**) \, \forall\, n-r\leq i\leq n-1
\end{aligned}
\end{equation*}
where $\delta$ denotes $s_0\cdots s_{n-1}t_{n-1}$.
In terms of the generators $\{s,t,z\}$ of $W_{bn}^c(k,bn,cm)$, we have $t_{n-1}=t^n$, $z_0=z$ and $s_{i-1}=t^{i-1}st^{1-i}$. Using Proposition \ref{PresentationNorM1}, we thus get that $z_0=u^m$. Writing $x_i:=s_{i-1}$ for all $i=1,\dots,n$, $y:=t_{n-1}$ and $z:=z_0$, we get the desired presentation of $W_b^c(k,bn,cm)$.

\end{proof}

\begin{remark}
Comparing \cite[Tables 1-4]{BMR} with Presentation \eqref{PresToric} and Presentations \eqref{GeneralPresW}-\eqref{PresPres2y=1}, Theorem \ref{Pres2} provides a new proof that all rank two complex reflection groups are isomorphic to $J$-groups, as summarized in Table 1.
\end{remark}

\begin{table}[hbt]
\centering
 \renewcommand{\arraystretch}{1.25}
\begin{tabular}{|l|c|}
\hline
Rank two reflection group & $J$-reflection group\\
\hline
$G(3,3,2),G_4,G_8,G_{16}$ &  $W(k,2,3), \, k=2,3,4,5$
\\
\hline
 $G_{20}$ &  $W(3,2,5)$\\
\hline
 $G(2l+1,2l+1,2)$ &  $W(2,2,2l+1)$
\\
\hline
 $G_{12}$&  $W(2,3,4)$
\\
\hline
 $G_{22}$ &  $W(2,3,5)$\\
\hline
 $G_5,G_{10},G_{18}$ &  $W_3(k,3,2),\, k=3,4,5$
\\
\hline
$G(d(2l+1),2l+1,2)$ &  $W_d(2,d(2l+1),2)$
\\
\hline
 $G_6,G_9,G_{17}$ &  $W_2(k,2,3),\, k=3,4,5$
\\
\hline
 $G_{13}$ & $W_2(2,4,3)$
\\
\hline
 $G_{14}$ &  $W_3(2,3,4)$
\\
\hline 
 $G_{21}$ &  $W_3(2,3,5)$
\\
\hline
$G(2d,2d,2)$ &  $W_2(2,2,d)$
\\
\hline
$G_7,G_{11},G_{19}$ &  $W_2^3(k,2,3), \, k=3,4,5$
\\
\hline
$G_{15}$ &  $W_2^3(2,4,3)$
\\
\hline
 $G(2cd,2c,2)$ &  $W_2^c(2,2,cd)$
\\
\hline

\end{tabular}
\label{Table}
\caption{Rank two complex reflection groups as $J$-reflection groups.}
\end{table}

To end this section, in Proposition \ref{prop1} we make the reflection isomorphism between $W_b^c(k,bn,cm)$ and  $W_c^b(k,cm,bn)$ described in Remark \ref{isopermute} explicit in terms of Presentation \eqref{GeneralPresW}. To this end, denote by $\{a_1,\dots,a_m,p,q\}$ the elements of $W_c^b(k,cm,bn)$ corresponding to $\{x_1,\dots,x_n,y,z\}$ in Presentation \eqref{GeneralPresW}. Denote by $\{s',t',u'\}$ the canonical generators of $J(k,bn,cm)$ and $\{s,t,u\}$ the canonical generators of $J(k,cm,bn)$. Finally, denote by $f$ the isomorphism between $J(k,bn,cm)$ and $J(k,cm,bn)$ described in Remark \ref{isopermute} and by $\tilde f$ its restriction to $W_b^c(k,bn,cm)$.

\begin{proposition}\label{prop1}
For all $i\in [n]$, writing $i-1=g_im+h_i$ the euclidean division of $i-1$ by $m$ we have 
\begin{equation}
\tilde f(x_i)=((a_1\cdots a_mp)^{g_i}a_1\cdots a_{h_i})a_{h_i+1}^{-1}((a_1\cdots a_mp)^{g_i}a_1\cdots a_{h_i})^{-1}.
\end{equation}
\noindent
Moreover, we have
\begin{equation}
\tilde f(y)=q^{-1}, \, \tilde f(z)=p^{-1}.
\end{equation}
\end{proposition}
\noindent
To prove Proposition \ref{prop1}, we show the following lemma: 

\begin{lemma}\label{lem1}
For all $N\in \N$, we have $u^{-N}s^{-1}u^N=(st)^Ns^{-1}(st)^{-N}$. In terms of Presentation \eqref{GeneralPresW}, if we write $N=gm+h$ the euclidean division of $N$ by $m$, we have
$$u^{-N}s^{-1}u^N=(a_1\cdots a_mp)^ga_1\cdots a_{h}a_{h+1}^{-1}((a_1\cdots a_mp)^ga_1\cdots a_{h})^{-1}.$$
\end{lemma}
\begin{proof}
\fbox{$u^{-N}s^{-1}u^N=(st)^Ns^{-1}(st)^{-N}$:}
Since $u(st)=ust$ is central in $J(k,cm,bn)$, conjugating by $st$ amounts to conjugating by $u^{-1}$.\\
\fbox{$u^{-N}s^{-1}u^N=(a_1\cdots a_mp)^ga_1\cdots a_{h}a_{h+1}^{-1}((a_1\cdots a_mp)^ga_1\cdots a_{h})^{-1}$:} Let $N=gm+h$ be the euclidean division of $N$ by $m$. We have 
\begin{equation*}
    \begin{aligned}
        &u^{-N}s^{-1}u^N=(st)^Ns^{-1}(st)^{-N}=((st)^m)^g(st)^hs^{-1}(((st)^m)^g(st)^h)^{-1}\\
        &=(((\prod_{j=0}^{m-1}t^{j}st^{-j})t^m)^g(\prod_{j=0}^{h-1}t^jst^{-j})t^h)s^{-1}(((\prod_{j=0}^{m-1}t^{j}st^{-j})t^m)^g(\prod_{j=0}^{h-1}t^jst^{-j})t^h)^{-1}\\
        &\underset{\text{Thm.}\, \ref{Pres2}}=(a_1\cdots a_mp)^ga_1\cdots a_{h}t^hs^{-1}t^{-h}((a_1\cdots a_mp)^ga_1\cdots a_{h})^{-1}\\
        &\underset{\text{Thm.}\, \ref{Pres2}}=(a_1\cdots a_mp)^ga_1\cdots a_{h}a_{h+1}^{-1}((a_1\cdots a_mp)^ga_1\cdots a_{h})^{-1} \, \text{since $h\leq m-1$}
    \end{aligned}
\end{equation*}
This concludes the proof.
\end{proof}

\begin{proof}[Proof of Proposition \ref{prop1}]
By Theorem \ref{Pres2}, for all $i=1,\dots,n$ we have\\ $x_i=(t')^{i-1}s'(t')^{1-i}\in J(k,bn,cm)$. Thus, we have 
\begin{equation*}
    \begin{aligned}
        \tilde f(x_i)&=f((t')^{i-1}s'(t')^{1-i})=(u^{-1})^{i-1}s^{-1}(u^{-1})^{1-i}\\
        &=(a_1\cdots a_mp)^{g_i}a_1\cdots a_{h_i}a_{h_i+1}^{-1}((a_1\cdots a_mp)^{g_i}a_1\cdots a_{h_i})^{-1} \, \text{by Lemma \ref{lem1}.}
    \end{aligned}
\end{equation*}
Moreover, we have $\tilde f(y)=f((t')^n)=u^{-n}=q^{-1}$. Similarly, we have $\tilde f(z)=p^{-1}$. This concludes the proof.
\end{proof}

\subsection{\texorpdfstring{The center of $J$-reflection groups}{}}
The goal of this subsection is to determine the center of $J$-reflection groups.
As was noted by Achar \& Aubert in \cite[Section 8]{AA}, combining the correspondence between finite $J$-groups and rank two complex reflection groups established in \cite{AA} with \cite[Tables 1-4]{BMR} provides a description of the center of finite $J$-groups in terms of the canonical generators of their parent $J$-groups:
If the group $H=J\begin{pmatrix}a & b & c \\ a' & b' & c'\end{pmatrix}$ is finite its center is cyclic, generated by the smallest power of $stu$ belonging to $H$. 
In \cite{Gobet Toric}, Gobet extended this result to all (non necessarily finite) parent $J$-groups and all toric reflection groups.\\
In order to state the result, we introduce the following notation:
\begin{notation}
Let $k,n,m\in \N_{\geq 2}$. We denote by $W^+_{k,n,m}$ the group defined by the presentation 
\begin{equation}\label{PresAlternate}
\left\langle \begin{array}{c|c}
   a,b\, &\, a^k=b^n=(ba^{-1})^m=1
\end{array}
\right\rangle.\end{equation}
\end{notation}

This notation comes from the fact that $W^+_{k,n,m}$ is isomorphic to the alternating subgroup of the triangle Coxeter group with edge labels $k$,$n$ and $m$ (see \cite[Exercise 9, Section 1, Chapter 4]{BourbakiLie4-6}).

Using this notation, Gobet showed the following result:

\begin{theorem}[{\cite[Theorem 3.3 (2) and (3)]{Gobet Toric}}]\label{CenterToric}
Let $k,n,m\in \N_{\geq 2}$ with $n$ and $m$ coprime. \\
(1) We have the following commutative diagram of short exact sequences: 
\begin{equation}\label{SESToric}\begin{tikzcd}
	1 & {\langle (x_1\cdots x_n)^m\rangle} & {W(k,n,m)} & {W^+_{k,n,m}} & 1 \\
	1 & {\langle stu\rangle} & {J(k,n,m)} & {W^+_{k,n,m}} & 1
	\arrow[from=1-1, to=1-2]
	\arrow[from=1-2, to=1-3]
	\arrow[from=1-3, to=1-4]
	\arrow[from=1-4, to=1-5]
	\arrow[hook', from=1-2, to=2-2]
	\arrow[hook', from=1-3, to=2-3]
	\arrow["{=}", from=1-4, to=2-4]
	\arrow[from=2-1, to=2-2]
	\arrow[from=2-2, to=2-3]
	\arrow[from=2-4, to=2-5]
	\arrow[from=2-3, to=2-4]
\end{tikzcd}\end{equation}
(2) We have $Z(W(k,n,m))=\langle (x_1\cdots x_n)^m\rangle$ and $Z(J(k,n,m))=\langle stu\rangle$.
\end{theorem}
\begin{remark}\label{CenterNotPrime0}
As was noted in \cite[Proof of Theorem 3.3]{Gobet Toric}, we have $(x_1\cdots x_n)^m=(stu)^{nm}$. Moreover, the morphism $J(k,n,m)\to W^+_{k,n,m}$ sends $s$ to $a$, $t$ to $b^{-1}$ and $u$ to $ba^{-1}$.
\end{remark}

\begin{remark}\label{CenterNotPrime}
In \cite[Proof of Theorem 3.3 (2) and (3)]{Gobet Toric}, the fact that $n$ and $m$ are coprime is not used to prove the results about $J(k,n,m)$. In particular, they hold for all $k,n,m\geq 2$. We write $\phi$ for the morphism $J(k,n,m)\to W^+_{k,n,m}$ described in Remark \ref{CenterNotPrime0}. We also note $\pi$ the morphism $J(k,n,m)\to \Z/nm$ constructed in Remark \ref{QuotientAbParent}.
\end{remark}

A key ingredient for the proof of Theorem \ref{CenterToric} which we shall use to determine the center of $J$-reflection groups is the following proposition:
\begin{proposition}[{\cite[Proposition 3.1 + Proof of Theorem 3.3]{Gobet Toric}}]\label{CenterAlternating}
For all $k,n,m\in \N_{\geq 2}$ such that $W^+_{k,n,m}$ is infinite, the center of $W^+_{k,n,m}$ is trivial. Moreover, if $k,n$ and $m$ are not all even and $W^+_{k,n,m}$ is finite, its center is trivial.
\end{proposition}

In order to determine the center of $J$-reflection groups, we will combine Proposition \ref{CenterAlternating} with Remark \ref{QuotientAbParent}.

Recall that $\Inn(G)$ denotes the inner automorphism group of $G$, which is isomorphic to $G/Z(G)$.

\begin{theorem}\label{CenterFINDANAME}
Let $b,c,k,n,m\in \N^*$ with $k,bn,cm\geq 2$ and suppose $n\wedge m=1$.\\
(1) We have the following exact double complex: \begin{equation}\label{DoubleComplex}\begin{tikzcd}
	& 1 & 1 & 1 \\
	1 & \ker(\phi|_{_{W_b^c(k,bn,cm)}}) & {W_b^c(k,bn,cm)} & {W^+_{k,bn,cm}} & 1 \\
	1 & {\langle stu\rangle} & {J(k,bn,cm)} & {W^+_{k,bn,cm}} & 1 \\
	1 & {\langle\pi(stu)\rangle} & {\Z/nm} & 1 & 1 \\
	& 1 & 1 & 1
	\arrow[from=2-1, to=2-2]
	\arrow[from=2-2, to=2-3]
	\arrow["\phi", from=2-3, to=2-4]
	\arrow[from=2-4, to=2-5]
	\arrow[from=2-2, to=3-2]
	\arrow[from=2-3, to=3-3]
	\arrow["{=}", from=2-4, to=3-4]
	\arrow[from=3-4, to=4-4]
	\arrow["\pi", from=3-3, to=4-3]
	\arrow[from=3-2, to=4-2]
	\arrow[from=3-1, to=3-2]
	\arrow[from=3-2, to=3-3]
	\arrow["\phi", from=3-3, to=3-4]
	\arrow[from=3-4, to=3-5]
	\arrow[from=4-1, to=4-2]
	\arrow[from=4-2, to=4-3]
	\arrow[from=4-3, to=4-4]
	\arrow[from=4-4, to=4-5]
	\arrow[from=1-2, to=2-2]
	\arrow[from=1-3, to=2-3]
	\arrow[from=1-4, to=2-4]
	\arrow[from=4-2, to=5-2]
	\arrow[from=4-3, to=5-3]
	\arrow[from=4-4, to=5-4]
\end{tikzcd}\end{equation}
(2) With the notation of Theorem \ref{Pres2}, the center of $W_b^c(k,bn,cm)$ is equal to $\langle (x_1\cdots x_ny)^mz^n\rangle$. In particular, $J$-reflection groups have a cyclic center.\\
(3) We have $\ker(\phi|_{_{W_b^c(k,bn,cm)}})=Z(W_b^c(k,bn,cm))$. In particular, we have the following isomorphism: $$\mathrm{Inn}(W_b^c(k,bn,cm))\cong W^+_{k,bn,cm}.$$
\end{theorem}

Before giving the proof of Theorem \ref{CenterFINDANAME}, we make the following observations:

\begin{corollary}\label{CenterSpecialCases}
(i) If $m\geq 2$, the center of $W_b(k,bn,m)$ is $\langle (x_1\cdots x_ny)^m\rangle$.\\
(ii) If $n\geq 2$, the center of $W^c(k,n,cm)$ is $\langle (x_1\cdots x_n)^mz^n\rangle$.\\
(iii) If $n,m\geq 2$, we retrieve point (2) of Theorem \ref{CenterToric}, that is, the center of $W(k,n,m)$ is $\langle (x_1\cdots x_n)^m\rangle$.
\end{corollary}

\begin{proof}[Proof of Theorem \ref{CenterFINDANAME}]
\fbox{(1)}
That the squares are commutative follows from the definitions.
It remains to show that rows and columns are exact. The exactness of the middle row follows from Remark \ref{CenterNotPrime} and the exactness of the middle column follows from Remark \ref{QuotientAbParent}.  Moreover, Remark \ref{QuotientAbParent} shows that the order of $\pi(t)$ is $n$ and that the order of $\pi(u)$ is $m$, so that the order of $\pi(stu)=\pi(t)\pi(u)$ is $nm$ since $n\wedge m=1$. This shows the exactness of the bottom row. The exactness of the right column is trivial.
To show that the top row is exact we need to prove that $\phi|_{_{W_b^c(k,bn,cm)}}$ is surjective. By definition, the subgroup $W_b^c(k,bn,cm)$ is equal to $\llangle s,t^n,u^m\rrangle_{J(k,bn,cm)}$ and since $\phi$ is surjective we get that \begin{equation*}\im(\phi|_{_{W_b^c(k,bn,cm)}})=\llangle \phi(s),\phi(t)^n,\phi(u)^m\rrangle_{W^+_{k,bn,cm}}\underset{\text{Rem.}\, \ref{CenterNotPrime}}=\llangle a,b^{-n},(ba^{-1})^m\rrangle_{W^+_{k,bn,cm}}.
\end{equation*}

A presentation of $W^+_{k,bn,cm}/\im(\phi|_{_{W_b^c(k,bn,cm)}})$ is thus
\begin{equation}\label{PhiSurjCox}
\left\langle\begin{array}{c|c}
 a,b \, & \, a=b^{-n}=(ba^{-1})^m=1
\end{array}\right\rangle.
\end{equation}
Since $n$ and $m$ are coprime, Presentation \eqref{PhiSurjCox} is a presentation of the trivial group so that $\im(\phi|_{_{W_b^c(k,bn,cm)}})=W^+_{k,bn,cm}$, which is the desired equality.\\
Finally, since the middle and right columns are exact, the snake lemma induces an exact sequence 
\begin{equation}\label{exactsnake}
\begin{tikzcd}
	{\ker(\phi|_{_{W_b^c(k,bn,cm)}})} & {\ker(\phi)} & {\ker(\Z/nm\to 1)} \\
	{\coker(\phi|_{_{W_b^c(k,bn,cm)}})} & {\coker(\phi)} & {\coker(\Z/nm\to 1)}
	\arrow[from=1-1, to=1-2]
	\arrow[from=1-2, to=1-3]
	\arrow[from=1-3, to=2-1]
	\arrow[from=2-1, to=2-2]
	\arrow[from=2-2, to=2-3]
\end{tikzcd}
\end{equation}
By exactness of the middle and bottom rows of Diagram \eqref{DoubleComplex}, we have $\ker(\phi)=\langle stu\rangle$ and $\ker(\Z/nm\to 1)=\langle\pi(stu)\rangle$. Moreover, since all cokernels in \eqref{exactsnake} are trivial and the morphism $\ker(\phi|_{_{W_b^c(k,bn,cm)}})\to \langle stu\rangle $ is injective by construction, we get that the left column is exact.\\
\fbox{(2)+(3)}
If $W_b^c(k,bn,cm)$ is finite, using the correspondence between Presentation \eqref{GeneralPresW} and the BMR presentation and inspecting \cite[Tables 1-3]{BMR} gives the result.\\
Assume now that $W_b^c(k,bn,cm)$ is infinite. We have $$\ker(\phi|_{_{W_b^c(k,bn,cm)}})\subset Z(J(k,bn,cm))\cap W_b^c(k,bn,cm)\subset Z(W_b^c(k,bn,cm)).$$ Since $W_b^c(k,bn,cm)$ is infinite, Proposition \ref{CenterAlternating} applies and shows that $$\ker(\phi|_{_{W_b^c(k,bn,cm)}})=Z(W_b^c(k,bn,cm)),$$ which proves $(3)$. This in turn implies that $Z(W_b^c(k,bn,cm))=\langle (stu)^{nm}\rangle$ by exactness of the left column and last row of \eqref{DoubleComplex}.\\It remains to show that $(x_1\cdots x_ny)^mz^n=(stu)^{nm}$:
\begin{equation*}
    \begin{aligned}
       (x_1\cdots x_ny)^mz^n&=(s(tst^{-1})(t^2st^{-2})\cdots (t^{n-2}st^{2-n})(t^{n-1}st^{1-n})t^n)^m(u^m)^n\\
       &=((st)^n)^mu^{mn}=(stu)^{nm}\,  \text{since $(st)u=u(st)$.}
    \end{aligned}
\end{equation*}
This concludes the proof.
\end{proof}

\section{\texorpdfstring{Classification of $J$-reflection groups}{}}

The aim of this section is to provide a classification of $J$-reflection groups up to reflection isomorphisms. As stated in the next subsection, some results in this direction where proven in \cite{AA} and \cite{Gobet Toric}.
\subsection{\texorpdfstring{The reflection structure of $J$-reflection groups}{}}

An important notion for $J$-groups is that of reflecting hyperplanes, generalising the notion of reflecting hyperplanes of complex reflection groups. In this subsection, we define and study reflecting hyperplanes of $J$-reflection groups.\\
Recall that given a $J$-group $H=\llangle s^k,t^n,u^m\rrangle_G$ with parent $J$-group $G$, its set of reflections $R(H)$ is the set of conjugates in $G$ of non-trivial powers of $s^k,t^n$ or $u^m$ (see Definition \ref{DefReflections}).
\begin{definition}[\texorpdfstring{\cite[Definition 2.7]{Gobet Toric}}{}]\label{DefReflections}
Let $H$ be a $J$-group. Let $\sim$ be the equivalence relation on $R(H)$ generated by the relations $r^a\sim r^b$ for all $1\leq a,b<o(r)$, $r\in R(H)$ and write $[r]$ for the equivalence class of $r$ in $R(H)$. Now, define the set $\mathcal H(H)$ of \textbf{reflecting hyperplanes} of $H$ to be $\{[r]\}_{r\in R(H)}$.\\
Finally, for $r\in R(H)$, define $[r]_\mathfrak c$ as $\cup_{h\in H}[hrh^{-1}]$ and write $\mathcal H_\mathfrak c(H)$ for $\{[r]_\mathfrak c\}_{r\in R(H)}$.
\end{definition}

In the case where the $J$-group $H$ is finite, elements of $\mathcal H_\mathfrak c (H)$ are the unions of conjugacy classes of reflecting hyperplanes, so that $|\mathcal H_\mathfrak c (H)|$ is the number of conjugacy classes of reflecting hyperplanes. We keep this nomenclature for all $J$-groups and call the elements of $\mathcal H_\mathfrak c(H)$ the \textbf{conjugacy classes of reflecting hyperplanes} of $H$. 

\noindent
One can see (\cite[Example 2.14]{Gobet Toric}) that every rank two complex reflection group having a single conjugacy class of reflecting hyperplanes is a toric reflection group. It turns out that the latter groups are a generalisation of the former. This is formalised as follows:

\begin{lemma}[{\cite[Corollary 2.13]{Gobet Toric}}]\label{SingleW}
In a toric reflection group $W(k,n,m)$, all $x_i$'s are conjugate. As a consequence, by Proposition \ref{ToricReflections} the group $W(k,n,m)$ has a single conjugacy class of reflecting hyperplanes.
\end{lemma}

We generalise Lemma \ref{SingleW} by showing that the number of conjugacy classes of a $J$-reflection group is governed solely by the fact that it contains some powers of $t$ and $u$ or not:
\begin{proposition}\label{NumberConj}
The number of conjugacy classes of reflecting hyperplanes of the $J$-reflection group $W_b^c(k,bn,cm)$ is $3-\delta_{1,b}-\delta_{1,c}$. More precisely, with the notation of Theorem \ref{Pres2} the set $\mathcal H_\mathfrak c(W_b^c(k,bn,cm))$ is equal to $\{[x_1]_\mathfrak c,[y]_\mathfrak c,[z]_\mathfrak c\}\backslash \{\emptyset\}$, where we define $[1]_\mathfrak c$ to be $\emptyset$.
\end{proposition}

\begin{proof}
First, Lemma \ref{GeneratorsFINDANAME} shows that a reflection $r$ of $W_b^c(k,bn,cm)$ is conjugate to a power of an element of the set $\{x_1,\dots,x_n,y,z\}\backslash\{1\}$, thus $[r]_\mathfrak c$ belongs to $\{[x_1]_\mathfrak c,\dots,[x_n]_\mathfrak c,[y]_\mathfrak c,[z]_\mathfrak c\}\backslash \{\emptyset\}$. Corollary \ref{xiConjugate} then implies that for all $i=1,\dots,n$ we have $[x_i]_\mathfrak c=[x_1]_\mathfrak c$, so that $\mathcal H_\mathfrak c(W_b^c(k,bn,cm))=\{[x_1]_\mathfrak c,[y]_\mathfrak c,[z]_\mathfrak c\}\backslash \{\emptyset\}$.
Finally, Lemma $\ref{NoConjRef}$ ensures that $[x_1]_\mathfrak c$, $[y]_\mathfrak c$ and $[z]_\mathfrak c$ are all distinct, which shows that $\mathcal H_\mathfrak c(W_b^c(k,bn,cm))$ contains $3-\delta_{1,b}-\delta_{1,c}$ elements. This concludes the proof.
\end{proof}

As one would expect to be true, a very useful fact to classify $J$-reflection groups is that their number of conjugacy classes of reflecting hyperplanes is invariant under reflection isomorphisms. 

\begin{lemma}\label{InvariantPre}
Let $H_1,H_2$ be two $J$-reflection groups and assume that $H_1\xto\varphi H_2$ is a reflection isomorphism. Then for all $r\in R(H_1)$, we have $\varphi([r])=[\varphi(r)]$ and $\varphi([r]_\mathfrak c)=[\varphi(r)]_\mathfrak c$.
\end{lemma}

\begin{proof}
Since $r\in [r]$ and $\varphi(r)\in [\varphi(r)]$, we have $\varphi([r])\cap[\varphi(r)]\neq \emptyset$ so that it is enough to show that $\varphi([r])$ is an equivalence class for $\sim$ in $H_2$ to prove that $\varphi([r])=[\varphi(r)]$.\\
Since $\varphi$ sends reflections to reflections, if $t$ is a reflection of $H_1$ and $t^a\neq 1$ then $\varphi(t^a)=\varphi(t)^a$ is a non trivial power of a reflection of $H_2$. In particular, since $\sim$ is generated by relations of the form $t^a\sim t^b$, we get that if $x\sim y$ in $H_1$, then $\varphi(x)\sim \varphi(y)$ in $H_2$. Using the same argument for $\varphi^{-1}$ shows that $x\sim y$ in $H_1$ if and only if $\varphi(x)\sim \varphi(y)$ in $H_2$.\\
Let then $x\in \varphi([r])$ and assume $x\sim y$. Then we have $\varphi^{-1}(x)\sim \varphi^{-1}(y)$ and since $\varphi^{-1}(x)\in [r]$ we have $\varphi^{-1}(y)\in [r]$, which shows that $y=\varphi(\varphi^{-1}(y))\in \varphi([r])$. This shows that $\varphi([r])=[\varphi(r)]$. We then have 
\begin{equation*}
\begin{aligned}
\varphi([r]_\mathfrak c)&=\varphi(\cup_{h\in H_1}[hrh^{-1}])=\cup_{h\in H_1}\varphi([hrh^{-1}]) \\&=\cup_{h\in H_1}[\varphi(h)\varphi(r)\varphi(h)^{-1}]=[\varphi(r)]_\mathfrak c \, \text{since $\varphi$ is an isomorphism.}
\end{aligned}
\end{equation*}
This concludes the proof.
\end{proof}
\begin{corollary}\label{InvariantHyperplane}
Let $H_1,H_2$ be two $J$-reflection groups and assume that $H_1\xto \varphi H_2$ is a reflection isomorphism. Then $\varphi$ induces a bijection $\mathcal H_\mathfrak c(H_1)\xto{\tilde\varphi} \mathcal H_\mathfrak c(H_2)$ and for each $h\in \mathcal H_\mathfrak c(H_1)$, the map $\varphi|_{_h}$ is a bijection from $h$ to $\tilde\varphi(h)$.
\end{corollary}

\begin{proof}
Let $r_1,r_2\in R(H_1)$. Then we have 
\begin{equation*}
\begin{aligned}
\, [r_1]_\mathfrak c=[r_2]_\mathfrak c &\Leftrightarrow\, \varphi([r_1]_\mathfrak c)=\varphi([r_2]_\mathfrak c) \, \text{since $\varphi$ is an isomorphism}\\
&\Leftrightarrow\,[\varphi(r_1)]_\mathfrak c=[\varphi(r_2)]_\mathfrak c \, \text{by Lemma \ref{InvariantPre}.}
\end{aligned}
\end{equation*}
This shows that the map $\mathcal{H}_\mathfrak c(H_1)\xto{\tilde\varphi} \mathcal H_\mathfrak c(H_2)$ sending $[r]_\mathfrak c$ to $[\varphi(r)]_\mathfrak c$ is well defined and injective. Moreover, Lemma \ref{InvariantPre} implies that for all $r\in R(H_2)$ we have $\varphi([\varphi^{-1}(r)]_\mathfrak c)=[r]_\mathfrak c$ so that $\tilde \varphi$ is surjective, thus bijective. The second part of the claim follows from Lemma \ref{InvariantPre}, by combining the fact that $\varphi$ is bijective with $\varphi([r]_\mathfrak c)=[\varphi(r)]_\mathfrak c$.
\end{proof}

\begin{corollary}\label{Numberof1}
If $W_b^c(k,bn,cm)\cong_{ref}W_{b'}^{c'}(k',b'n',c'm')$, we have $|\{b,c\}\backslash \{1\}|=|\{b',c'\}\backslash \{1\}|$.
\end{corollary}
\begin{proof}
It follows by combining Proposition \ref{NumberConj} and Corollary \ref{InvariantHyperplane}.
\end{proof}

The last reflection isomorphism invariant datum of $J$-groups we will need to be able to classify $J$-reflection groups is the orders of their (conjugacy classes of) reflections:

\begin{definition}\label{OrderReflection}
The \textbf{orders} of a conjugacy class of reflecting hyperplanes $C$ is the set $O(C):=\{o(x),x\in C\}$.
With this notation, given a $J$-group $H$, we define $O(H)$ to be the multiset $\sqcup_{C\in \mathcal H_\mathfrak c(H)}O(C)$.
\end{definition}

\begin{remark}\label{OrderIsInvariant}
Since isomorphisms preserve orders, if two $J$-reflection groups $H_1$ and $H_2$ are isomorphic in reflection, Corollary \ref{InvariantHyperplane} implies that $O(H_1)=O(H_2)$.
\end{remark}

\begin{lemma}\label{OrderFINDANAMERef}
Let $H$ be a $J$-group with parent $J$-group $G$ and let $r\in R(H)$. Assume that $r$ is of the form $gx^ag^{-1}$ with $g\in G$ and $x\in \{s,t,u\}$. Then all elements of $[r]_\mathfrak c$ are conjugate in $G$ to a power of $x$.
\end{lemma}

\begin{proof}
Since $\sim$ is generated by relations of the form $r_1^p\sim r_1^q$, it is enough to show that if two reflections $r_1,r_2$ of $H$ satisfy $r_1^{p_1}=r_2^{p_2}$ for some $p_1,p_2\in \N^*$, then $r_1$ and $r_2$ are conjugate to some power of the same element of $\{s,t,u\}$.\\
By definition, there exists $g_1,g_2\in G$, $x_1,x_2\in \{s,t,u\}$ and $q_1,q_2\in \N^*$ such that $r_i=g_ix_i^{q_i}g_i^{-1}$ for $i=1,2$. We then have 
\begin{equation*}
g_1x_1^{p_1q_1}g_1^{-1}=r_1^{p_1}=r_2^{p_2}=g_2x_2^{p_2q_2}g_2^{-1}.
\end{equation*}
Lemma \ref{NoConjRef} thus implies that $x_1=x_2$, which concludes the proof.
\end{proof}

\subsection{\texorpdfstring{Classification of $J$-reflection groups}{}}
Theorem \ref{Classification1} is a first step towards the classification of $J$-groups up to reflection isomorphisms. In \cite{Gobet Toric}, Gobet extends this classification result to toric reflection groups:
\begin{theorem}[{\cite[Theorem 1.2]{Gobet Toric}}]\label{Classification2}
Let $H_1,H_2$ be toric reflection groups. Then $H_1$ and $H_2$ are isomorphic in reflection if and only if $C(H_1)=C(H_2)$.
\end{theorem}

The goal of this subsection is to generalise Theorem \ref{Classification2} to all $J$-reflection groups.
\begin{theorem}\label{Classification3}
Let $H_1,H_2$ be $J$-reflection groups. Then $H_1$ and $H_2$ are isomorphic in reflection if and only if $C(H_1)=C(H_2)$.
\end{theorem}

In order to prove Theorem \ref{Classification3}, we need to derive some reflection invariants of $J$-reflection groups. To do so, we state the following useful result: 

\begin{proposition}[{\cite[Proposition 4.6]{Gobet Toric}}]\label{MaximalSubgroup}
Let $k,n,m\geq 2$ be such that $W^+_{k,n,m}$ is infinite. There are three conjugacy classes of maximal finite subgroups of $W^+_{k,n,m}$. The finite groups in these three classes are isomorphic to $\Z/k,\Z/n$ and $\Z/m$.
\end{proposition}

%In order to be able to fully use this result, we observe that $W^+_{k,n,m}$ is finite precisely when $J(k,n,m)$ is as well:
\begin{remark}\label{CoxFinite}
After an inspection of \cite[Table 1]{AA} and using the classification of finite Coxeter groups (see \cite{CoxeterClassification}), we get that $J(k,n,m)$ is infinite if and only if $W^+_{k,n,m}$ is infinite.
\end{remark}

%%%%%%%%%%%%%%%%%%%%%%%%%%%%%%%%%%%%%%%%%%%%%%%%%%%%%%%%%%%%
%%%%%%%%%%%%%%%%%%%%%%%%%%%%%%%%%%%%%%%%%%%%%%%%%%%%%%%%%%%%
%%%%%%%%%%%%%%%%%%%%%%%%%%%%%%%%%%%%%%%%%%%%%%%%%%%%%%%%%%%%%%%%%%%%%%%%%%%%%%%%%%%%%%%%%%%%%%%%%%%%%%%%%%%%%%%%%%%%%%%%%%%%%%%%%%%%%%%%%%%%%%%%%%%%%%%%%%%%%%%%%%%%%%%%%%%%%%%%%%%%
%%%%%%%%%%%%%%%%%%%%%%%%%%%%%%%%%%%%%%%%%%%%%%%%%%%%%%%%%%%%
%%%%%%%%%%%%%%%%%%%%%%%%%%%%%%%%%%%%%%%%%%%%%%%%%%%%%%%%%%%%%%%%%%%%%%%%%%%%%%%%%%%%%%%%%%%%%%%%%%%%%%%%%%%%%%%%%%%%%%%%
%%%%%%%%%%%%%%%%%%%%%%%%%%%%%%%%%%%%%%%%%%%%%%%%%%%%%%%%%%%%
%%%%%%%%%%%%%%%%%%%%%%%%%%%%%%%%%%%%%%%%%%%%%%%%%%%%%%%%%%%%
%%%%%%%%%%%%%%%%%%%%%%%%%%%%%%%%%%%%%%%%%%%%%%%%%%%%%%%%%%%%%%%%%%%%%%%%%%%%%%%%%%%%%%%%%%%%%%%%%%%%%%%%%%%%%%%%%%%%%%%%%%%%%%%%%%%%%%%%%%%%%%%%%%%%%%%%%%%%%%%%%%%%%%%%%%%%%%%%%%%%
%%%%%%%%%%%%%%%%%%%%%%%%%%%%%%%%%%%%%%%%%%%%%%%%%%%%%%%%%%%%
%%%%%%%%%%%%%%%%%%%%%%%%%%%%%%%%%%%%%%%%%%%%%%%%%%%%%%%%%%%%%%%%%%%%%%%%%%%%%%%%%%%%%%%%%%%%%%%%%%%%%%%%%%%%%%%%%%%%%%%%%%%%%%%%%%%%%%%%%%%%%%%%%%%%%%%%%%%%%%%%%%%%%%%%%%%%%%%%%%%%
%%%%%%%%%%%%%%%%%%%%%%%%%%%%%%%%%%%%%%%%%%%%%%%%%%%%%%%%%%%%
%%%%%%%%%%%%%%%%%%%%%%%%%%%%%%%%%%%%%%%%%%%%%%%%%%%%%%%%%%%%%%%%%%%%%%%%%%%%%%%%%%%%%%%%%%%%%%%%%%%%%%%%%%%%%%%%%%%%%%%%%%%%%%%%%%%%%%%%%%%%%%%%%%%

These facts provide a useful reflection invariant of $J$-reflection groups. 
\begin{lemma}\label{SameParent}
Let $b,b',c,c',k,k',n,n',m,m'\in \N^*$ with $k,k',bn,b'n',cm,c'm'\geq 2$ and assume that $n$ and $m$ are coprime and $n'$ and $m'$ are coprime. If $W_b^c(k,bn,cm)$ and $W_{b'}^{c'}(k',b'n',c'm')$ are isomorphic and infinite, the multisets $\multl k,bn,cm\multr$ and $\multl k',b'n',c'm'\multr$ are equal.
\end{lemma}

\begin{proof}
If $W_b^c(k,bn,cm)$ and $W_{b'}^{c'}(k',b'n',c'm')$ are (abstractly) isomorphic, their inner automorphism groups are isomorphic as well. By Theorem \ref{CenterFINDANAME}, this implies that $W^+_{k,bn,cm}$ is isomorphic to $W^+_{k',b'n',c'm'}$. Now, using Remark \ref{CoxFinite} we get that both $W^+_{k,bn,cm}$ and  $W^+_{k',b'n',c'm'}$ are infinite so that Proposition \ref{MaximalSubgroup} applies, which concludes the proof.
\end{proof}

\begin{lemma}\label{SameOrder}
Let $b,b',c,c',k,k',n,n',m,m'\in \N^*$ with $k,k',bn,b'n',cm,c'm'\geq 2$, and assume that $n\wedge m=1=n'\wedge m'$. If $W_b^c(k,bn,cm)$ is isomorphic in reflection to $W_{b'}^{c'}(k',b'n',c'm')$, the multisets $\multl k,b,c\multr$ and $\multl k',b',c'\multr$ are equal.
\end{lemma}

\begin{proof}
We do the proof for $b$ and $c$ different from one. Recall that $O(W_b^c(k,bn,cm))$ is the multiset consisting of the orders of the conjugacy classes of reflecting hyperplanes of $W_b^c(k,bn,cm)$. By Corollary \ref{Numberof1}, the fact that $1\notin \{b,c\}$ implies that $b'$ and $c'$ are different from one as well. In this case, Proposition \ref{NumberConj} implies that $\mathcal H_\mathfrak c(W_b^c(k,bn,cm))=\{[x_1]_\mathfrak c,[y]_\mathfrak c,[z]_\mathfrak c\}$. Now, Lemma \ref{OrderFINDANAMERef} implies that $O(W_b^c(k,bn,cm))=\mathrm d (k)\sqcup \mathrm d(b)\sqcup \mathrm d(c),$ where $\mathrm d(l)$ denotes the set of positive divisors of $l$.\\
Similarly, we have $O(W_{b'}^{c'}(k',b'n',c'm'))=\mathrm d (k')\sqcup \mathrm d(b')\sqcup \mathrm d(c')$. Since $W_{b}^{c}(k,n,m)$ and $W_{b'}^{c'}(k',b'n',c'm')$ are isomorphic in reflection, Remark \ref{OrderIsInvariant} implies that $\mathrm d (k)\sqcup \mathrm d(b)\sqcup \mathrm d(c)=\mathrm d (k')\sqcup \mathrm d(b')\sqcup \mathrm d(c')$. This identifies $M=\max(k,b,c)$ with $M'=\max(k',b',c')$. Removing $\mathrm d(M)=\mathrm d(M')$ from those sets, we can then identify $\max(\{k,b,c\}\backslash\{M\})$ with $\max(\{k',b',c'\}\backslash \{M'\})$. Repeating this process one more time concludes the proof for $b$ and $c$ different than one. A similar proof applies for the case $1\in \{b,c\}$, which concludes the proof.
\end{proof}

\begin{corollary}\label{nm=}
If $W_b^c(k,bn,cm)$ and $W_{b'}^{c'}(k',b'n',c'm')$ are isomorphic in reflection, we have $nm=n'm'$.
\end{corollary}

\begin{proof}
Lemma \ref{SameParent} implies that $kbcnm=k'b'c'n'm'$ and Lemma \ref{SameOrder} implies that $kbc=k'b'c'$. This concludes the proof.
\end{proof}

We are almost ready to prove Theorem \ref{Classification3}. The last ingredients needed are the following results:
\begin{remark}\label{QuotientOrder}
Combining Presentation \eqref{GeneralPresW} with Corollary \ref{xiConjugate} gives us the following isomorphisms for all $J$-reflection groups $W_b^c(k,bn,cm)$:\\
(i) $W_b^c(k,bn,cm)/\llangle x_1,y\rrangle\cong \Z/c$\\
(ii) $W_b^c(k,bn,cm)/\llangle x_1,z\rrangle\cong \Z/b$.
\end{remark}

\begin{remark}\label{ClosureConj}
Let $b,c,k,n,m\in \N^*$ with $k,bn,cm\geq 2$ and assume that $n$ and $m$ are coprime. Let $H=W_b^c(k,bn,cm)$. For all $r\in \{x_1,y,z\}$, we have $\llangle [r]\rrangle_H=\llangle [r]_\mathfrak c\rrangle_H$. Indeed, all elements of $[r]_\mathfrak c$ are conjugate to a power of $r$, hence are in $\llangle r\rrangle_H$.
\end{remark}
\begin{proposition}\label{PhiTransport}
Let $H_1,H_2$ be $J$-reflection groups having three conjugacy classes of reflecting hyperplanes and assume that there is a reflection isomorphism $H_1\xto \varphi H_2$. Write $a_1,p,q$ the elements of $H_2$ corresponding to $x_1,y,z$ in Presentation \eqref{GeneralPresW}. Finally, let $v,w\in \{x_1,y,z\}$ and $r,s\in \{a_1,p,q\}$ be such that $\tilde\varphi([v]_\mathfrak c)=[r]_\mathfrak c$ and $\tilde \varphi([w]_\mathfrak c)=[s]_\mathfrak c$ where $\tilde\varphi$ is as in Corollary \ref{InvariantHyperplane}.\\
Then we have $\varphi(\llangle v,w\rrangle_{H_1})=\llangle r,s\rrangle_{H_2}$.
\end{proposition}

\begin{proof}
First, recall that given a group $G$ and two normal subgroups $N,M\lhd G$, we have $NM=\langle N,M\rangle$ and if $N=\llangle X\rrangle_G$, $M=\llangle Y\rrangle_G$, we have $NM=\llangle X\cup Y\rrangle_G$. Now, we have
\begin{equation*}
    \begin{aligned}
        \varphi(\llangle v,w\rrangle_{H_1})&=\varphi(\llangle v\rrangle_{H_1}\llangle w\rrangle_{H_1})=\varphi(\llangle v\rrangle_{H_1})\varphi(\llangle w\rrangle_{H_1})\\&=\varphi(\llangle [v]_\mathfrak c\rrangle_{H_1})\varphi(\llangle [w]_\mathfrak c\rrangle_{H_1}) \, \text{by Remark \ref{ClosureConj}}\\
        &=\llangle \varphi([v]_\mathfrak c)\rrangle_{H_2}\llangle \varphi([w]_\mathfrak c)\rrangle_{H_2}\, \text{since $\varphi$ is surjective}\\
        &=\llangle [r]_\mathfrak c\rrangle \llangle [s]_\mathfrak c\rrangle_{H_2}=\llangle r\rrangle_{H_2}\llangle s\rrangle_{H_2} \, \text{by Remark \ref{ClosureConj}}\\
        &=\llangle r,s\rrangle_{H_2}.
    \end{aligned}
\end{equation*}
This concludes the proof.
\end{proof}

\begin{proof}[Proof of Theorem \ref{Classification3}]
If $H_1$ and $H_2$ are finite, this is the content of Theorem \ref{Classification1}. From now on, we assume that $H_1$ and $H_2$ are infinite. Write $H_1=W_b^c(k,bn,cm)$ and $H_2=W_{b'}^{c'}(k',b'n',c'm')$ with $n\leq m$ and $n'\leq m'$. For this proof, denote by $a_1,p,q$ the elements of $H_2$ corresponding to $x_1,y,z$ in Presentation \eqref{GeneralPresW}. Using Corollary \ref{Numberof1}, we can separate the proof in five cases:\\
(i) $b=b'=c=c'=1$.\\
(ii) $b,b'\neq 1$, $c=c'=1$.\\
(iii) $b=b'=1$, $c,c'\neq 1$.\\
(iv) $b=c'=1$, $c,b'\neq 1$.\\
(v) $1\notin\{b,c\}$.\\
Note that the case $c=b'=1$, $b=c'\neq 1$ is symmetric to (iv). Case (i) is the content of Theorem \ref{Classification2}.\\
\fbox{(ii) $b,b'\neq 1$, $c=c'=1$:} In this case, we have that $m$ and $m'$ are strictly greater than $1$ since $cm,c'm'\geq 2$ by assumption. By Proposition \ref{NumberConj}, we have $\mathcal H_\mathfrak c(H_1)=\{[x_1]_\mathfrak c,[y]_\mathfrak c\}$ and $\mathcal H_\mathfrak c(H_2)=\{[a_1]_\mathfrak c,[p]_\mathfrak c\}$. Now, by Corollary \ref{InvariantHyperplane} we have two cases to treat.\\
\underline{(ii.1) $\ivl{x_1}{a_1}$, $\ivl yp$:} In this case, Remark \ref{ClosureConj} implies that 
\begin{equation}\label{case11}
    \varphi(\llangle y\rrangle_{H_1})=\varphi(\llangle [y]_\mathfrak c\rrangle_{H_1})\underset{\varphi \, \text{iso.}}=\llangle \varphi([y]_\mathfrak c)\rrangle_{H_2}=\llangle [p]_\mathfrak c\rrangle_{H_2} =\llangle p\rrangle_{H_2}.
\end{equation}
Equation \eqref{case11} shows that the morphism $H_1\to H_2/\llangle p\rrangle_{H_2}$ has kernel $\llangle y\rrangle_{H_1}$, hence Diagram \eqref{CommSquareFINDANAME} provides the isomorphisms
\begin{equation*}
W(k,n,m)\cong H_1/\llangle y\rrangle_{H_1}\cong H_2\llangle p\rrangle_{H_2}\cong W(k',n',m').
\end{equation*}
\noindent
If $n=1$, we get $W(k,n,m)\cong \Z/k\cong W(k',n',m')$ so that Lemma \ref{ToricCyclic} implies that $1\in \{n',m'\}$. Since $n'\leq m',$ we get $n'=1$. Moreover, Lemma \ref{ToricCyclic} implies that $k'=k$. Finally, Corollary \ref{nm=} implies that $m=m'$ and we get $(k,n,m)=(k',n',m')$. If $n\neq 1$, since $n\leq m$ Lemma \ref{ToricCyclic} implies that $W(k,n,m)$ is not cyclic hence neither is $W(k',n',m')$ so that Theorem \ref{Classification2} applies and identifies $(k,n,m)$ with $(k',n',m')$ since $n\leq m$ and $n'\leq m'$. Similarly, Remark \ref{QuotientOrder} provides an isomorphism 
\begin{equation*}
\Z/b\cong H_1/\llangle x_1\rrangle_{H_1}\cong H_2/\llangle a_1\rrangle_{H_2}\cong \Z/b'.
\end{equation*}
We get that $b=b'$, hence we have
\begin{equation*}
\begin{aligned}
    C(H_1)&=\col{k}{1}{bn}{n}{cm}{m}=\col{k'}{1}{b'n'}{n'}{m'}{m'}\\&=\col{k'}{1}{b'n'}{n'}{c'm'}{m'}=C(H_2).
\end{aligned}
\end{equation*}
\underline{(ii.2) $\ivl{x_1}{p}$, $\ivl y{a_1}$:} In this case, by the same reasoning as above we have the following isomorphisms:
\begin{equation*}
W(k,n,m)\cong H_1/\llangle y\rrangle_{H_1}\cong H_2/\llangle a_1\rrangle_{H_2}\cong \Z/b'.
\end{equation*}
Similarly, we have $W(k',n',m')\cong \Z/b$. Since $m,m'>1$, by Lemma \ref{ToricCyclic} we get $n=n'=1$, $k'=b$ and $b'=k$. Moreover, Corollary \ref{nm=} implies that $m=m'$. We get
\begin{equation*}
\begin{aligned}
    C(H_1)&=\col{k}{1}{bn}{n}{cm}{m}=\col{b'}{1}{k'}{1}{m'}{m'}\\&=\col{k'}{1}{b'n'}{n'}{c'm'}{m'}=C(H_2).
    \end{aligned}
\end{equation*}
\fbox{(iii) $b=b'=1,$ $c,c'\neq 1$:} In this case, we have that $n$ and $n'$ are strictly greater than 1 since $bn,b'n'\geq 2$ by assumption. Since $m\geq n$ and $m'\geq n'$, we also have $m,m'>1$. By Proposition \ref{NumberConj}, we have $\mathcal H_\mathfrak c(H_1)=\{[x_1]_\mathfrak c,[z]_\mathfrak c\}$ and $\mathcal H_\mathfrak c(H_2)=\{[a_1]_\mathfrak c,[q]_\mathfrak c\}$. Now, by Corollary \ref{InvariantHyperplane} we have two cases to treat.\\
\underline{(iii.1) $\ivl{x_1}{a_1},$ $\ivl zq$:} In this case, since $k,n,m,k',n',m'>1$, Theorem \ref{Classification2} applies so that the quotient isomorphism 
\begin{equation*}
    W(k,n,m)\cong H_1/\llangle z\rrangle_{H_1}\cong H_2/\llangle q\rrangle_{H_2}\cong W(k',n',m')
\end{equation*}
identifies $(k,n,m)$ with $(k',n',m')$ and the quotient isomorphism
\begin{equation*}
\Z/c\cong H_1/\llangle x_1\rrangle_{H_1}\cong H_2/\llangle a_1\rrangle_{H_2}\cong \Z/c'
\end{equation*}
identifies $c$ with $c'$. We get
\begin{equation*}
\begin{aligned}
    C(H_1)&=\col{k}{1}{bn}{n}{cm}{m}=\col{k'}{1}{n'}{n'}{c'm'}{m'}\\&=\col{k'}{1}{b'n'}{n'}{c'm'}{m'}=C(H_2).
    \end{aligned}
\end{equation*}
\underline{(iii.2) $\ivl{x_1}q$, $\ivl z{a_1}$:} In this case, the quotient isomorphism
\begin{equation*}
W(k,n,m)\cong H_1/\llangle z\rrangle_{H_1}\cong H_2/\llangle a_1\rrangle_{H_2} \cong \Z/c'
\end{equation*}
implies that $W(k,n,m)$ is cyclic but since $k,n,m>1$, we get a contradiction using Lemma \ref{ToricCyclic}.

\noindent
\fbox{(iv) $b=c'=1$, $c,b'\neq 1$:} In this case, we have that $n$ and $m'$ are strictly larger than $1$ since $bn,c'm'\geq 2$ by assumption. Since $m\geq n$, we also have $m>1$.
By Proposition \ref{NumberConj}, we have $\mathcal H_\mathfrak c(H_1)=\{[x_1]_\mathfrak c, [z]_\mathfrak c\}$ and $\mathcal H_\mathfrak c(H_2)=\{[a_1]_\mathfrak c,[p]_\mathfrak c\}$. Now, by Corollary \ref{InvariantHyperplane} we have two cases to treat.\\
\underline{(iv.1) $\ivl{x_1}{a_1}$, $\ivl zp$:} In this case, since $k,n,m>1$, Lemma \ref{ToricCyclic} implies that the group $W(k,n,m)$ is not cyclic. In particular, the quotient isomorphism
\begin{equation*}
    W(k,n,m)\cong H_1/\llangle z\rrangle_{H_1}\cong H_2/\llangle p\rrangle_{H_2}\cong W(k',n',m')
\end{equation*}
shows that $W(k',n',m')$ is not cyclic either. This shows that $n'\geq 2$ and Theorem \ref{Classification2} identifies $(k,n,m)$ with $(k',n',m')$. Moreover, the quotient isomorphism 
\begin{equation*}\Z/c\cong H_1/\llangle x_1\rrangle_{H_1}\cong H_2/\llangle a_1\rrangle_{H_2}\cong \Z/b'
\end{equation*}
identifies $c$ with $b'$. Using Lemma \ref{SameParent}, we have $\multl k,bn,cm\multr=\multl k',b'n',c'm'\multr$. Moreover, we have $b=c'=1$, hence $(c,k,n,m)=(b',k',n',m')$ so that $\multl n,cm\multr=\multl cn,m\multr$. Finally, since $c\neq 1$ we have $n=m$, which leads to a contradiction because $n$ and $m$ are coprime and larger than or equal to 2.\\
\underline{(iv.2) $\ivl{x_1}p$, $\ivl{z}{a_1}$:} In this case, the quotient isomorphism
\begin{equation*}
    W(k,n,m)\cong H_1/\llangle z\rrangle_{H_1}\cong H_2/\llangle a_1\rrangle_{H_2}\cong \Z/b'
\end{equation*}
implies that $W(k,n,m)$ is cyclic, which leads to a contradiction using Lemma \ref{ToricCyclic} because $n$ and $m$ are larger than or equal to 2.

\noindent
\fbox{(v) $1\notin \{b,c\}$:} By Proposition \ref{NumberConj}, we have $\mathcal H_\mathfrak c(H_1)=\{[x_1]_\mathfrak c,[y]_\mathfrak c,[z]_\mathfrak c\}$ and $\mathcal H_\mathfrak c(H_2)=\{[a_1]_\mathfrak c,[p]_\mathfrak c,[q]_\mathfrak c\}$. Now, by Corollary \ref{InvariantHyperplane} we have six cases to treat.\\
\underline{(v.1) $\ivl{x_1}{a_1}$, $\ivl yp$, $\ivl zq$:} First, by Proposition \ref{PhiTransport} we have $\varphi(\llangle y,z\rrangle_{H_1})=\llangle p,q\rrangle_{H_2}$. This shows that the morphism $H_1\to H_2/\llangle p,q\rrangle_{H_2}$ has kernel $\llangle y,z\rrangle_{H_1}$, hence Diagram \eqref{CommSquareFINDANAME} provides the isomorphisms 
\begin{equation*}
W(k,n,m)\cong H_1\llangle y,z\rrangle_{H_1}\cong H_2\llangle p,q\rrangle_{H_2}\cong W(k',n',m').
\end{equation*}
\noindent
By the same argument as for the case (ii.1), we conclude that $(k,n,m)=(k',n',m')$. Moreover, the quotient isomorphisms
\begin{equation*}
    \Z/c\cong H_1/\llangle x_1,y\rrangle_{H_1} \cong H_2/\llangle a_1,p\rrangle_{H_2}\cong \Z/c'
\end{equation*}
and 
\begin{equation*}
    \Z/b\cong H_1/\llangle x_1,z\rrangle_{H_1} \cong H_2/\llangle a_1,q\rrangle_{H_2}\cong \Z/b'
\end{equation*}
identify $(b,c)$ with $(b',c')$, hence we have
\begin{equation*}
    C(H_1)=\col{k}{1}{bn}{n}{cm}{m}=\col{k'}{1}{b'n'}{n'}{c'm'}{m'}=C(H_2).
\end{equation*}
\underline{(v.2) $\ivl{x_1}{a_1}$, $\ivl yq$, $\ivl zp$:} In this case, the same argument as for the case (v.1) identifies $(k,n,m)$ with $(k',n,'m')$. Moreover, the quotient isomorphisms 
\begin{equation*}
    \Z/c\cong H_1/\llangle x_1,y\rrangle_{H_1} \cong H_2/\llangle a_1,q\rrangle_{H_2}\cong \Z/b'
\end{equation*}
and 
\begin{equation*}
    \Z/b\cong H_1/\llangle x_1,z\rrangle_{H_1} \cong H_2/\llangle a_1,p\rrangle_{H_2}\cong \Z/c'
\end{equation*}
identify $(b,c)$ with $(c',b')$. Combining this with Lemma \ref{SameParent}, we have 
\begin{equation*}
\multl bn,cm\multr=\multl k,bn,cm\multr\backslash\multl k\multr=\multl k',b'n',c'm'\multr\backslash\multl k'\multr=\multl cn,bm\multr.
\end{equation*}
If $bn=cn$, we get $b=c=b'=c'$ so that 
\begin{equation*}
\begin{aligned}
    C(H_1)&=\col{k}{1}{bn}{n}{cm}{m}=\col{k'}{1}{bn'}{n'}{bm'}{m'}\\&=\col{k'}{1}{b'n'}{n'}{c'm'}{m'}=C(H_2).
    \end{aligned}
\end{equation*}
If $bn=bm$, we get $n=m$. Since $n$ and $m$ are coprime, this gives $n=m=1$, which in turn implies $n'=m'=1$ since $(k,n,m)=(k',n',m')$. We thus have
\begin{equation*}
\begin{aligned}
    C(H_1)&=\col{k}{1}{bn}{n}{cm}{m}=\col{k'}{1}{c'}{1}{b'}{1}\\&=\col{k'}{1}{b'n'}{n'}{c'm'}{m'}=C(H_2).
    \end{aligned}
\end{equation*}
\underline{(v.3) $\ivl{x_1}{p}$, $\ivl y{a_1}$, $\ivl zq$:} In this case, the quotient isomorphism 
\begin{equation*}
    W(k,n,m)\cong H_1/\llangle y,z\rrangle_{H_1}\cong H_2/\llangle a_1,q\rrangle_{H_2}\cong \Z/b'
\end{equation*}
shows that $W(k,n,m)$ is cyclic. Since $n\leq m$, Lemma \ref{ToricCyclic} identifies $k$ with $b'$ and $n$ with $1$. Similarly, the quotient isomorphism
\begin{equation*}
   \Z/b\cong H_1/\llangle x_1,z\rrangle_{H_1}\cong H_2/\llangle p,q\rrangle_{H_2}\cong W(k',n',m')
\end{equation*}
shows that $W(k',n',m')$ is cyclic. Since $n'\leq m'$, Lemma \ref{ToricCyclic} identifies $b$ with $k'$ and $n'$ with $1$. By Corollary \ref{nm=}, we get $m=m'$. Finally, Lemma \ref{SameOrder} implies that $c=c'$. We thus get
\begin{equation*}\begin{aligned}
    C(H_1)&=\col{k}{1}{bn}{n}{cm}{m}=\col{b'}{1}{k'}{1}{c'm'}{m'}\\&=\col{k'}{1}{b'n'}{n'}{c'm'}{m'}=C(H_2).
    \end{aligned}
\end{equation*}
\underline{(v.4) $\ivl{x_1}{p}$, $\ivl y{q}$, $\ivl z{a_1}$:}
In this case, the quotient isomorphism 
\begin{equation*}
    W(k,n,m)\cong H_1/\llangle y,z\rrangle_{H_1}\cong H_2/\llangle a_1,q\rrangle_{H_2}\cong \Z/b'
\end{equation*}
shows that $W(k,n,m)$ is cyclic. Since $n\leq m$, Lemma \ref{ToricCyclic} identifies $k$ with $b'$ and $n$ with $1$. Similarly, the quotient isomorphism
\begin{equation*}
   \Z/c\cong H_1/\llangle x_1,y\rrangle_{H_1}\cong H_2/\llangle p,q\rrangle_{H_2}\cong W(k',n',m')
\end{equation*}
shows that $W(k',n',m')$ is cyclic. Since $n'\leq m'$, Lemma \ref{ToricCyclic} identifies $c$ with $k'$ and $n'$ with $1$. By Corollary \ref{nm=}, we get $m=m'$. Finally, Lemma \ref{SameOrder} implies that $b=c'$. If $m=m'=1,$ we have
\begin{equation*}
\begin{aligned}
    C(H_1)&=\col{k}{1}{bn}{n}{cm}{m}=\col{b'}{1}{c'}{1}{k'}{1}\\&=\col{k'}{1}{b'n'}{n'}{c'm'}{m'}=C(H_2).
    \end{aligned}
\end{equation*}
Assume then that $m\neq 1$. Lemma \ref{SameParent} implies that 
\begin{equation}\label{eqv.4}
    \multl k,b,cm\multr=\multl k,bn,cm\multr=\multl k',b'n',c'm'\multr=\multl c,k,bm\multr.
\end{equation}
Since $m\geq 2$, Equation \eqref{eqv.4} implies that $b=c$. Since $k'=c=b=c'$, $k=b'$ and $m=m'$ we get
\begin{equation*}
\begin{aligned}
    C(H_1)&=\col{k}{1}{bn}{n}{cm}{m}=\col{b'}{1}{b}{1}{bm}{m}\\&=\col{k'}{1}{b'n'}{n'}{c'm'}{m'}=C(H_2).
    \end{aligned}
\end{equation*}
\underline{(v.5) $\ivl{x_1}{q}$, $\ivl y{a_1}$, $\ivl z{p}$:}
This case can be deduced from case (v.4) by exchanging the role of $H_1$ and $H_2$ and replacing $\varphi$ by $\varphi^{-1}$.\\
\underline{(v.6) $\ivl{x_1}{q}$, $\ivl y{p}$, $\ivl z{a_1}$:}
In this case, the quotient isomorphism 
\begin{equation*}
    W(k,n,m)\cong H_1/\llangle y,z\rrangle_{H_1}\cong H_2/\llangle a_1,p\rrangle_{H_2}\cong \Z/c'
\end{equation*}
shows that $W(k,n,m)$ is cyclic. Since $n\leq m$, Lemma \ref{ToricCyclic} identifies $k$ with $c'$ and $n$ with $1$. Similarly, the quotient isomorphism
\begin{equation*}
   \Z/c\cong H_1/\llangle x_1,y\rrangle_{H_1}\cong H_2/\llangle p,q\rrangle_{H_2}\cong W(k',n',m')
\end{equation*}
shows that $W(k',n',m')$ is cyclic. Since $n'\leq m'$, Lemma \ref{ToricCyclic} identifies $c$ with $k'$ and $n'$ with $1$. By Corollary \ref{nm=}, we get $m=m'$. Finally, Lemma \ref{SameOrder} implies that $b=b'$. If $m=m'=1,$ we have
\begin{equation*}
\begin{aligned}
    C(H_1)&=\col{k}{1}{bn}{n}{cm}{m}=\col{c'}{1}{b'}{1}{k'}{1}\\&=\col{k'}{1}{b'n'}{n'}{c'm'}{m'}=C(H_2).
    \end{aligned}
\end{equation*}
Assume then that $m\neq 1$. Lemma \ref{SameParent} implies that 
\begin{equation}\label{eqv.6}
    \multl k,b,cm\multr=\multl k,bn,cm\multr=\multl k',b'n',c'm'\multr=\multl c,b,km\multr.
\end{equation}
Since $m\geq 2$, Equation \eqref{eqv.6} implies that $k=c$. Since $c'=k=c=k'$, $b=b'$ and $m=m'$ we get
\begin{equation*}
\begin{aligned}
    C(H_1)&=\col{k}{1}{bn}{n}{cm}{m}=\col{c}{1}{b'}{1}{cm}{m}\\&=\col{k'}{1}{b'n'}{n'}{c'm'}{m'}=C(H_2).
    \end{aligned}
\end{equation*}
This concludes the proof.
\end{proof}

\end{document}